\documentclass[11pt]{amsart}
\usepackage{amsaddr}
\usepackage{amscd}
\usepackage{amsfonts}
\usepackage{amsmath}
\usepackage{amssymb}
\usepackage{amsthm}
\usepackage{cases}
\usepackage{chngcntr}
\usepackage{cite}
\usepackage{color}
\usepackage{graphicx}
\usepackage{latexsym}
\usepackage{mathtools}
\usepackage{microtype}
\usepackage[active]{srcltx}
\usepackage{url}

\usepackage{hyperref}
\usepackage{enumitem}
\usepackage{cleveref}

\usepackage[UKenglish]{babel}
\usepackage[UKenglish]{datetime}

\usepackage[charter]{mathdesign}

\usepackage{bbm} % extra blackboard symbols, serif and sans serif
\newcommand{\bbfont}{\mathbbm}

\usepackage[mathscr]{euscript}
\usepackage{calrsfs}
\DeclareMathAlphabet{\mathcalalt}{OMS}{cmsy}{m}{n}
\newcommand{\poalgfont}{\mathcalalt}
\newcommand{\borelfont}{\mathscr}
\newcommand{\lebfont}{\mathcal}
\newcommand{\opstrufont}{\mathcalalt}

\theoremstyle{plain}

\newtheorem{theorem0}{Theorem}[section]
\newtheorem{theorem}[theorem0]{Theorem}
\newtheorem{proposition}[theorem0]{Proposition}
\newtheorem{lemma}[theorem0]{Lemma}
\newtheorem{corollary}[theorem0]{Corollary}

\newtheorem*{theorem*}{Theorem}
\newtheorem*{proposition*}{Proposition}
\newtheorem*{lemma*}{Lemma}
\newtheorem*{corollary*}{Corollary}

\theoremstyle{definition}

\newtheorem{definition}[theorem0]{Definition}
\newtheorem{example}[theorem0]{Example}
\newtheorem{remark}[theorem0]{Remark}

\newtheorem*{definition*}{Definition}
\newtheorem*{example*}{Example}
\newtheorem*{remark*}{Remark}

\crefname{theorem}{Theorem}{Theorems}
\crefname{proposition}{Proposition}{Propositions}
\crefname{lemma}{Lemma}{Lemmas}
\crefname{corollary}{Corollary}{Corollaries}
\crefname{definition}{Definition}{Definitions}
\crefname{example}{Example}{Examples}
\crefname{remark}{Remark}{Remarks}

\crefname{section}{Section}{Sections}
\crefname{subsection}{Section}{Sections}
\crefname{subsubsection}{Section}{Sections}

\crefname{equation}{equation}{equations}
\crefname{enumi}{part}{parts}
\crefname{enumii}{part}{parts}
\crefname{enumiii}{part}{parts}
\crefname{enumiv}{part}{parts}

\setlist[enumerate,1]{label=\textup{(\arabic*)},ref=\textup{(\arabic*)}}
\setlist[enumerate,2]{label=\textup{(\alph*)},ref=\textup{(\alph*)}}
\setlist[enumerate,3]{label=\textup{(\roman*)},ref=\textup{(\roman*)}}
\setlist[enumerate,4]{label=\textup{(\Alph*)},ref=\textup{(\Alph*)}}

\newlist{enumerate_arabic}{enumerate}{1}
\setlist[enumerate_arabic,1]{label=\textup{(\arabic*)},ref=\textup{(\arabic*)}}

\newlist{enumerate_alpha}{enumerate}{1}
\setlist[enumerate_alpha,1]{label=\textup{(\alph*)},ref=\textup{(\alph*)}}

\newlist{enumerate_roman}{enumerate}{1}
\setlist[enumerate_roman,1]{label=\textup{(\roman*)},ref=\textup{(\roman*)}}

\newlist{enumerate_Alpha}{enumerate}{1}
\setlist[enumerate_Alpha,1]{label=\textup{(\Alph*)},ref=\textup{(\Alph*)}}

\numberwithin{equation}{section}
\allowdisplaybreaks

\newcommand{\tposskip}{\hskip   0.05555555555em}
\newcommand{\tnegskip}{\hskip -0.05555555555em}

\newcommand{\CC}{{\bbfont C}}

\newcommand{\RR}{{\bbfont R}}
\newcommand{\ZZ}{{\bbfont Z}}

\newcommand{\upc}{{\mathrm{c}}}
\newcommand{\upd}{{\mathrm{d}}}
\newcommand{\upe}{{\mathrm{e}}}
\newcommand{\upi}{{\mathrm{i}}}
\newcommand{\upo}{{\mathrm{o}}}
\newcommand{\upr}{{\mathrm{r}}}
\newcommand{\upB}{{\mathrm{B}}}
\newcommand{\upC}{{\mathrm{C}}}
\newcommand{\upL}{{\mathrm{L}}}

\newcommand{\ulp}{{\textup{(}}}
\newcommand{\urp}{{\textup{)}}}
\newcommand{\uppars}[1]{\ulp #1\urp}

\newcommand{\abs}[1]{{\lvert #1 \rvert}}
\newcommand{\norm}[1]{{\lVert #1 \rVert}}

\newcommand{\lrnorm}[1]{{\left\lVert #1 \right\rVert}}

\newcommand{\di}[1]{\,\upd #1}

\newcommand{\linear}{{\opstrufont L}}
\newcommand{\regular}{\linear_{\mathrm r}}

\newcommand{\Calgebra}{\ensuremath{{\upC}^\ast}\!-algebra}
\newcommand{\Csubalgebra}{\ensuremath{{\upC}^\ast}\!-subalgebra}
\newcommand{\Calgebras}{\Calgebra s}
\newcommand{\Csubalgebras}{\Csubalgebra s}

\newcommand{\idmap}{{\mathrm{id}}}
\newcommand{\idop}{I}

\newcommand{\Ker}{\operatorname{Ker}}

\newcommand{\Wed}[1]{\operatorname{Wed}\left[#1\right]}

\newcommand{\up}{{\Uparrow}}
\newcommand{\down}{{\Downarrow}}
\newcommand{\ups}{{\uparrow}}
\newcommand{\downs}{{\downarrow}}

\newcommand{\mc}{mo\-no\-tone com\-plete}
\newcommand{\smc}{$\sigma$-mono\-tone com\-plete}
\newcommand{\Dc}{De\-de\-kind com\-plete}
\newcommand{\sDc}{$\sigma$-De\-de\-kind com\-plete}

\newcommand{\smcont}{$\sigma$-mono\-tone con\-tin\-u\-ous}
\newcommand{\mcont}{mono\-tone con\-tin\-u\-ous}

\newcommand{\csp}{countable sup property}

\newcommand{\rsup}{\sup}

\newcommand{\psup}{\bigvee}
\newcommand{\pinf}{\bigwedge}

\newcommand{\supp}{\mathrm{supp}\,}

\newcommand{\SOT}{{\ensuremath{\mathrm{SOT}}}}

\newcommand{\SOTlim}{\SOT\text{--}\!\lim}

\newcommand{\zerofunction}{\textbf{0} }
\newcommand{\onefunction}{\textbf{1}}
\newcommand{\indicator}[1]{\chi_{#1}}

\newcommand{\pos}[1]{{#1^+}}
\newcommand{\negt}[1]{{#1^-}}

\newcommand{\largest}{\infty}
\newcommand{\f}[1]{(#1)}

\newcommand{\lrinp}[1]{\left\langle #1\right\rangle}

\newcommand{\seq}[1]{\{{#1}_n\}_{n=1}^{\infty}}
\newcommand{\net}[1]{\{{#1}_\lambda\}_{\lambda\in \Lambda}}

\newcommand{\Ell}{\upL}

\newcommand{\sa}{\mathrm{sa}}

\newcommand{\oa}{\poalgfont{A}}
\newcommand{\oatwo}{\poalgfont{B}}

\newcommand{\os}{E}
\newcommand{\ostwo}{F}

\newcommand{\hilbert}{H}

\newcommand{\pset}{X}

\newcommand{\ts}{X}

\newcommand{\posos}{\pos{\os}}
\newcommand{\osext}{\overline{\os}}
\newcommand{\pososext}{\overline{\posos}}

\newcommand{\posoa}{\pos{\oa}}

\newcommand{\posoaext}{\overline{\posoa}}

\newcommand{\posR}{\pos{\RR}}

\newcommand{\posRext}{\overline{\pos{\RR}}}

\newcommand{\posmap}{\pi}

\newcommand{\opint}[1]{I_{#1}}
\newcommand{\opintm}{\opint{\npm}}
\newcommand{\opintP}{\opint{\psm}}

\newcommand{\cont}[1]{{\upC}(#1)}
\newcommand{\conto}[1]{\upC_0(#1)}
\newcommand{\contc}[1]{\upC_{\upc}(#1)}
\newcommand{\contts}{\cont{\ts}}
\newcommand{\contots}{\conto{\ts}}
\newcommand{\contcts}{\contc{\ts}}

\newcommand{\contCts}{\cont{\ts;\CC}}
\newcommand{\contoCts}{\conto{\ts;\CC}}
\newcommand{\contcCts}{\contc{\ts;\CC}}

\newcommand{\odual}[1]{{#1^{\thicksim}}}
\newcommand{\ndual}[1]{{#1^{\ast}}}
\newcommand{\ocdual}[1]{{#1_{\mathrm {oc}}^{\thicksim}}}
\newcommand{\socdual}[1]{{#1_{\sigma\mathrm {oc}}^{\thicksim}}}

\newcommand{\odualos}{\odual{\os}}
\newcommand{\ndualos}{\ndual{\os}}
\newcommand{\ocdualos}{\ocdual{\os}}
\newcommand{\socdualos}{\socdual{\os}}

\newcommand{\ndualoa}{\ndual{\oa}}
\newcommand{\ocdualoa}{\ocdual{\oa}}

\newcommand{\boundedh}{\opstrufont{B}(\hilbert)}

\newcommand{\ocontinuous}{\linear_{\mathrm{oc}}}
\newcommand{\socontinuous}{\linear_{\sigma\mathrm{oc}}}
\newcommand{\linearop}[1]{\linear(#1)}

\newcommand{\regularop}[1]{\regular(#1)}
\newcommand{\ocontop}[1]{\ocontinuous(#1)}
\newcommand{\socontop}[1]{\socontinuous(#1)}

\newcommand{\alg}{\Omega}

\newcommand{\borel}{\borelfont B}

\newcommand{\mss}{\Delta}

\newcommand{\ms}{(\pset,\alg)}
\newcommand{\msm}{(\pset,\alg,\npm,\os)}

\newcommand{\npm}{\mu}

\newcommand{\psm}{\npm} % no distinction between spectral measures and others anymore

\newcommand{\elemfun}{{{\lebfont E}(\pset,\alg;\posR)}}
\newcommand{\integrableelemfun}{{{\lebfont E}^1(\pset,\alg;\posR)}}

\newcommand{\integrablefun}{{{\lebfont L}^1(\pset,\alg,\npm;\RR)}}
\newcommand{\integrablefunprime}{{{\lebfont L}^1(\pset,\alg,\npm\sp\prime;\RR)}}
\newcommand{\posintegrablefun}{{{\lebfont L}^1(\pset,\alg,\npm;\posR)}}
\newcommand{\aezerofun}{{{\lebfont N}(\pset,\alg,\npm;\RR)}}

\newcommand{\boundedmeasfun}{{{\lebfont{B}(\pset,\alg;\RR)}}}
\newcommand{\posboundedmeasfun}{{{\lebfont B}(\pset,\alg;\posR)}}
\newcommand{\ellone}{{{\mathrm L}^1(\pset,\alg,\npm;\RR)}}
\newcommand{\boundedmeasfunae}{{{\upB(\pset,\alg,\npm;\RR)}}}
\newcommand{\posboundedmeasfunae}{{{\upB(\pset,\alg,\npm;\pos{\RR})}}}

\newcommand{\posellone}{{{\mathrm L}^1(\pset,\alg,\npm;\posR)}}

\newcommand{\integrablefunspectral}{{{\lebfont L}^1(\pset,\alg,\psm;\RR)}}

\newcommand{\aezerofunspectral}{{{\lebfont N}(\pset,\alg,\psm;\RR)}}

\newcommand{\ellonespectral}{{{\mathrm L}^1(\pset,\alg,\psm;\RR)}}

\newcommand{\integrableelemfunts}{{{\lebfont E}^1(\ts,\borel;\posR)}}

\newcommand{\integrablefunts}{{{\lebfont L}^1(\ts,\borel,\npm;\RR)}}
\newcommand{\posintegrablefunts}{{{\lebfont L}^1(\ts,\borel,\npm;\posR)}}

\newcommand{\boundedmeasfunts}{{{\lebfont{B}(\ts,\borel;\RR)}}}
\newcommand{\boundedmeasfunCts}{{{\lebfont{B}(\ts,\borel;\CC)}}}
\newcommand{\boundedmeasfunCaets}{{{\upB(\ts,\borel,\npm;\CC)}}}
\newcommand{\posboundedmeasfunts}{{{\lebfont B}(\ts,\borel;\posR)}}
\newcommand{\ellonets}{{{\mathrm L}^1(\ts,\borel,\npm;\RR)}}
\newcommand{\posellonets}{{{\mathrm L}^1(\ts,\borel,\npm;\posR)}}
\newcommand{\boundedmeasfunaets}{{{\upB(\ts,\borel,\npm;\RR)}}}
\newcommand{\posboundedmeasfunaets}{{\upB(\ts,\borel,\npm;\posR)}}

\newcommand{\rs}{\os}

\newcommand{\orderintegral}[3]{{\int_{#1}^{\mathrm{o}}\! {#2}\di {#3}}}
\newcommand{\ointm}[1]{\orderintegral{\pset}{#1}{\npm}}
\newcommand{\ointmprime}[1]{\orderintegral{\pset}{#1}{\npm\sp\prime}}

\newcommand{\leftidalgnonassociative}[2]{_{#1}{#2}}
\newcommand{\rightidalgnonassociative}[2]{{#2}_{#1}}

\newcommand{\leftidalg}[2]{#1#2}
\newcommand{\rightidalg}[2]{#2#1}
\newcommand{\leftrightidalg}[2]{#1#2#1}

\newcommand{\eclass}[1]{\llbracket #1\rrbracket}

\begin{document}

%%%%%%%%%%%%%%%%%%%%%%%%%%%%%%%%% BEGIN FRONTMATTER %%%%%%%%%%%%%%%%%%%%%%%%%%%%%%%%%%%%
\title[Spectral theorems]{Spectral theorems for positive algebra homomorphisms}

\author{Marcel de Jeu}
\address[Marcel de Jeu]{Mathematical Institute, Leiden University, P.O.\ Box 9512, 2300 RA Leiden, The Netherlands\\
	and\\
	Department of Mathematics and Applied Mathematics, University of Pretoria, Corner of Lynnwood Road and Roper Street, Hatfield 0083, Pretoria,
	South Africa}
\email[Marcel de Jeu]{mdejeu@math.leidenuniv.nl}
%\thanks{Thanks for Author Two.}

\author{Xingni Jiang}
\address[Xingni Jiang]{College of Mathematics, Sichuan University, No.\ 24, South Section, First Ring Road, Chengdu, P.R.\ China}
\email[Xingni Jiang]{x.jiang@scu.edu.cn}
%\thanks{Thanks for Author Two.}

%\makeatletter{\renewcommand*{\@makefnmark}{}
%	\date{}\footnote{{File: \currfilebase. Compiled: \today,\,\currenttime.}}
%	\makeatother}

\subjclass[2010]{Primary 47B99; Secondary 06F25, 28B15}
%\dedicatory{Dedicated to }
%\thanks{This paper is in final form and no version of it will be submitted for
%publication elsewhere.}
\keywords{Riesz representation theorem, positive algebra homomorphism, par\-tial\-ly or\-dered algebra, spectral measure, ups and downs, Banach lattice, Hilbert space, JBW-algebra}

\begin{abstract}
Let $X$ be a locally compact Hausdorff space, let $\mathcalalt A$ be a partially ordered algebra, and let $\pi\colon \mathrm{C}_{\mathrm c}(X)\to\mathcalalt A$ be a positive algebra homomorphism. Under  conditions on $\mathcalalt A$ that are satisfied in a good number of cases of practical interest, it is shown that $\pi$ is represented by a unique regular spectral measure $\mu$ on the Borel $\sigma$-algebra of $X$, taking its values in the positive idempotents in $\mathcalalt A$. The measure $\mu$, which is $\sigma$-additive in an ordered sense, represents $\pi$ via the order integral (a generalisation of the Lebesgue integral) that goes back to J.D.M.\ Wright and which was investigated earlier by the authors. \\
The positive algebra homomorphism $\pi$ can be extended from $\mathrm{C}_{\mathrm c}(X)$ to a positive linear map from the accompanying $\mathcal L\sp 1$-space of $\mu$ into $\mathcalalt A$. It is shown that, quite often, this
$\mathcal L\sp 1$-space is closed under multiplication, so that it is a vector lattice algebra, and that the extended map from $\mathcal L\sp 1$ into $\mathcalalt A$ is not only an algebra homomorphism but, even when $\mathcalalt A$ is not a vector lattice, also a vector lattice homomorphism in a sense that is explained in the paper. When $\mathcalalt A$ has the countable sup property, the image of $\mathcal L^1$ (or of its positive cone) is described in terms of consecutive ups and downs of the image of ${\mathrm C}_{\mathrm c}(X)$ (or of its positive cone).\\
The general results are applied in three different contexts, showing how various spectral theorems have a common order-theoretical root. \\
For positive algebra homomorphisms from ${\mathrm C}_0(X)$ into the order continuous operators on a Banach lattice, this leads to an improvement of earlier work by the first author and Ruoff on positive representations on KB-spaces. For representations of ${\mathrm C}_0(X,\mathbb C)$ on Hilbert spaces, a rather precise spectral theorem for such (possibly degenerate) representations results, including explicit formulas for the spectral measure. Under a condition that is satisfied when the Hilbert space is separable, the image of the Borel functional calculus is shown to be the strongly closed subalgebra that is generated by the image of ${\mathrm C}_0(X,\mathbb C)$. \\
The algebra $\mathcalalt A$ need not be an algebra of operators. This allows a final application to JBW-algebras, where the existence of the Borel functional calculus and of the spectral resolution for an element now become consequences of the existence of a spectral measure.
\end{abstract}

\maketitle

%%%%%%%%%%%%%%%%%%%%%%%%%%%%%%%%%%%%%%%%%%%%%%%%%%%%%%%%%%%%%%%%%%%%%%%%%%%%%%%%%%%%%%%%%%%%%%%%%%%%%%%%%%%%%%%%%%%%%%%%%%%%%%%%%%%%

\section{Introduction and overview}\label{3_sec:introduction_and_overview}

%%%%%%%%%%%%%%%%%%%%%%%%%%%%%%%%%%%%%%%%%%%%%%%%%%%%%%%%%%%%%%%%%%%%%%%%%%%%%%%%%%%%%%%%%%%%%%%%%%%%%%%%%%%%%%%%%%%%%%%%%%%%%%%%%%%%

\noindent In this paper, we take a direct, order-theoretical, approach to aspects of spectral theory. For comparison, we start by briefly reviewing some known facts.

 Let $\ts$ be a compact Hausdorff space, and suppose that  $\posmap\colon \contCts\to\boundedh$ is a representation of its continuous complex-valued functions on a complex Hilbert space $\hilbert$. For $x,x^\prime\in\hilbert$, the Riesz representation theorem furnishes a regular complex Borel measure $\npm_{x,x^\prime}$ on the Borel $\sigma$-algebra $\borel$ of $\ts$ such that
\begin{equation}\label{3_eq:intro_measure_1}
\lrinp{\posmap(f)x,x^\prime}=\int_\ts\! f\di{\npm_{x,x^\prime}}
\end{equation}
for $f\in\contCts$. Turning the tables, \cref{3_eq:intro_measure_1} can be used as a definition to extend $\posmap$ to a map $\posmap\colon \boundedmeasfunCts\to\boundedh$ from the bounded Borel measurable functions on $\ts$ into $\boundedh$. This extended map can then be shown to be a representation again, and a spectral measure $\npm$ on $\borel$ is obtained by setting $\npm(\mss)\coloneqq \posmap(\indicator{\mss})$ for $\mss\in\borel$. It is such that
\begin{equation}\label{3_eq:intro_measure_2}
	\npm_{x,x^\prime}(\mss)=\lrinp{\npm(\mss)x,x^\prime}
\end{equation}
for $x,x^\prime\in\hilbert$. In this fashion, the existence of a (unique regular) spectral measure is establishes that generates the representation $\posmap$. It is uniquely determined by \cref{3_eq:intro_measure_1,3_eq:intro_measure_2}, and the validity of these two equations is what is meant by writing
\begin{equation}\label{3_eq:intro_measure_3}
\posmap(f)=\int_\ts\! f\di{\npm}
\end{equation}
for $f\in\contCts$.  We refer to, for example, \cite[Section~IX.1]{conway_A_COURSE_IN_FUNCTIONAL_ANALYSIS_SECOND_EDITION:1990} or \cite[Section~1.4]{folland_A_COURSE_IN_ABSTRACT_HARMONIC_ANALYSIS_SECOND_EDITION:2016} for this well-known material. A similar approach is used in \cite{de_jeu_ruoff:2016} to find a spectral measure for a positive representation $\posmap\colon \contots\to\regularop{\os}$ on a KB-space $\os$ of the real-valued functions on a locally compact Hausdorff space $\ts$ vanishing at infinity.

In the present paper, we take a different, direct, approach. Suppose that $\posmap\colon \contcts\to\oa$ is a positive algebra homomorphism from the real-valued compactly supported continuous functions on a locally compact Hausdorff space into a partially ordered algebra $\oa$. When $\oa$ satisfies reasonably mild conditions, which are satisfied in a number of cases of practical interest, we show that there exists a unique regular spectral measure $\npm$ on the Borel $\sigma$-algebra of $\ts$, taking its values in the positive idempotents in $\oa$ and $\sigma$-additive in an ordered sense, such that
\begin{equation}\label{3_eq:intro_measure_4}
	\posmap(f)=\ointm{f}
\end{equation}
for $f\in\contcts$. The integral in \cref{3_eq:intro_measure_4} is now not a symbolic notation as in \cref{3_eq:intro_measure_3}, but it is an actual integral that can be defined for measures taking values in (suitable) partially ordered vector spaces. This order integral, which is a generalisation of the Lebesgue integral, is studied in detail in \cite{de_jeu_jiang:2022a}. It goes back to J.D.M.\ Wright.

The algebra $\oa$ need not be an algebra of operators but, if it is, then, \emph{starting} from the measure $\npm$, one can define measures $\npm_{x,x^\prime}$ as in \cref{3_eq:intro_measure_2}. The validity of \cref{3_eq:intro_measure_1} then follows easily from the properties of the order integral. The algebra of self-adjoint operators in a commutative strongly closed \Csubalgebra\ of $\boundedh$ for a complex Hilbert space $\hilbert$ satisfies the appropriate conditions, and so does the algebra of regular operators on a KB-space.\footnote{As will become apparent in \cref{3_subsec:Banach_lattices}, one can actually do better than that.} Consequently, in these cases, our spectral measures for these algebras coincide with the ones that are found via the classical `weak' method summarised above. Our purely order-theoretical method via general partially ordered algebras is, however, essentially different in nature as the spectral measure comes \emph{before} the Borel functional calculus and the scalar valued measures. The difference (and advantage) becomes especially clear in the context of JBW-algebras. As these need not be algebras of operators, measures $\npm_{x,x^\prime}$ as above make no sense then. Yet our results are applicable, and they have the existence of a Borel functional calculus and the spectral resolution for an element of a JBW-algebra as an easy consequence.

The order-theoretical spectral theorems in the current paper are consequences of the  Riesz representation theorems in \cite{de_jeu_jiang:2022a} for positive linear maps $\posmap$ from $\contcts$ or $\contots$ into (suitable) partially ordered vector spaces. These apply, in particular, to the positive algebra homomorphisms from these spaces into $\oa$. The fact that the representing $\oa$-valued measure is then actually a spectral measure is a consequence of the multiplicativity of $\posmap$ and the explicit formulas from \cite{de_jeu_jiang:2022a} for the representing measure for the positive linear map $\posmap$. As for the special case of the classical Riesz representation theorem, where the partially ordered vector space consists of the real numbers, we have
\begin{align}
	\label{3_eq:intro_measure_5}
	\psm(V)&=\psup\{\posmap(f):f\in\contcts,\,\zerofunction\leq f\leq\onefunction,\ \supp{f}\subseteq V\}\\
\intertext{for an open subset $V$ of $\ts$, and}
\label{3_eq:intro_measure_6}
	\psm(K)&=\pinf\{\posmap(f):f\in\contcts,\,\zerofunction\leq f\leq\onefunction,\ f(x)=1\text{ for }x\in K\}
\end{align}
for a compact subset $K$ of $\ts$. In the classical real case, these formulas tend to fade into the background once they have served their purpose during the proof of the Riesz representation theorem.  In the present paper, however, they are very much in the foreground as they underlie the spectral property of $\npm$ when $\posmap$ is multiplicative. They are also instrumental to the up-down theorems that we shall establish. Even in the well-studied case of representations on complex Hilbert spaces these formulas appear to yield something new. For this, we recall that the existence of the extremum of a monotone net of self-adjoint operators in $\boundedh$ and the existence of its strong operator limit are equivalent and that, when they exists, they are equal. In this case, therefore, the supremum (resp.\ infimum) in \cref{3_eq:intro_measure_5} (resp.\ \cref{3_eq:intro_measure_6}) gives the spectral measures of non-empty open subsets and of compact subsets as explicit strong operator limits.\footnote{See \cref{3_rem:principled_order_theoretical_approach_2} for further comments.} Similar remarks apply to the regular operators on a Banach lattice with an order continuous norm.

\medskip

\noindent
This paper is organised as follows.

\cref{3_sec:preliminaries} contains the basic notion, definitions, and conventions, as well as  a few preparatory results. Of particular relevance are the partially ordered algebras in \cref{3_res:order_continuous_operators_are_suitable_algebra}, \cref{3_res:order_continuous_operators_on_Banach_lattices_are_suitable_algebra},  \cref{3_res:riesz_algebra_for_hilbert_spaces}, and \cref{res:JBW_algebra_monotone_continous_multiplication}. These are commonly occurring algebras to which the basic spectral theorem, \cref{3_res:positive_homomorphisms_into_partially_ordered_algebras}, applies. In \cref{3_subsec:moduli_preserving_operators}, we introduce a terminology to express that a linear map between two partially ordered vector spaces behaves to some extent  as a vector lattice homomorphism. We say that such maps preserve moduli. The necessary material from \cite{de_jeu_jiang:2022a} concerning the order integral is summarised in \cref{3_subsec:measures_and_integrals}, and \cref{3_subsec:spectral_measures} contains the definition of spectral measures in the general context. We also include an embedding result for $\contots$ for which we are not aware of a reference; see \cref{3_res:topological_embedding_with_closed_image}.

\cref{3_sec:commuting_idempotents_and_spectral_measures} contains the proof that the representing measure for a positive algebra homomorphism is spectral; see \cref{3_res:representing_measure_is_spectral}. A finitely additive measure on an algebra of sets, with values in an algebra without additive 2-torsion, and with the property that its image consists of commuting idempotents, is automatically a spectral measure; see \cref{3_res:idempotents_form_boolean_algebra}. Based on this, the spectrality of the representing measure is a surprisingly easy of \cref{3_eq:intro_measure_5}.

\cref{3_sec:relations_between_measures_algebra_homomorphisms_and_vector_lattice_homomorphisms} is mostly intended as a preparation for \cref{3_sec:ups_and_downs}, but it also has a value of its own. Its starting point is a measure $\mu$ with values in a vector lattice or in a partially ordered algebra. There is an associated map $\opintm$\textemdash defined by the order integral\textemdash from the corresponding $\lebfont L^1$-space into the vector lattice or partially ordered algebra. When is $\opintm$ a vector lattice homomorphism, or an algebra homomorphism? After answering these questions, we proceed to show that, when $\opintm$ is an algebra homomorphism (which is the case if and only if $\npm$ is spectral), it is, properly interpreted if necessary, also a vector lattice homomorphism; see \cref{3_res:integral_preserves_moduli,3_res:integral_into_riesz_algebra_is_riesz_algebra_homomorphism,3_res:eight_properties}.

\cref{3_sec:ups_and_downs} is not concerned with a positive algebra homomorphism, but, more generally, with a positive linear map $\posmap$ into a partially ordered vector space. Using \cref{3_eq:intro_measure_5,3_eq:intro_measure_6} and the regularity of the representing measure, it is not too difficult to see that various images of the associated operator $\opintm$ are contained in consecutive ups and downs of the image of (the positive cone of) $\contcts$ under $\posmap$. When the codomain has the countable sup property\footnote{As \cref{3_subsec:the_countable_sup_property} shows, this is, in practice, rather often the case.} and $\opintm$ preserves moduli, inclusions can be improved to equalities and net ups and downs can be replaced by their sequential counterparts. The monotone convergence theorem for the order integral (see \cite[Theorem~6.9]{de_jeu_jiang:2022a}) is the key to this.

In \cref{3_sec:riesz_representation_theorems_for_positive_algebra_homomorphisms}, we put the pieces together. It contains two spectral theorems for positive algebra homomorphisms. In the first one, the codomain is a Banach lattice algebra with an order continuous norm. Algebras of operators will only rarely fall into this category, but they \emph{do} tend to be in the range of the second theorem, where the codomain is a suitable partially ordered algebra. We have made some effort to collect all relevant results from \cite{de_jeu_jiang:2022a,de_jeu_jiang:2022b} and the present paper in these two theorems, which are the focal points of this paper.

In \cref{3_sec:special_positive_representations}, we apply the general theory to positive representations of $\contots$ on Banach lattices and, via its restriction to $\contots$, to representations of the complex algebra $\contoCts$ on Hilbert spaces. In the first case, a significant extension of the results in \cite{de_jeu_ruoff:2016} is obtained. In the second case, we obtain a spectral theorem for (possibly) degenerate representations of $\contoCts$ that appears to be more complete than is to be found in the literature.  When combined with \cite[Theorem~2.4.4]{pedersen_C-STAR-ALGEBRAS_AND_THEIR_AUTOMORPHISM_GROUPS:1979}, our results on ups and downs from \cref{3_sec:ups_and_downs} imply that, under a condition that is satisfied for separable Hilbert spaces, the image of the Borel functional calculus for a (possibly degenerate) representation of $\contoCts$ equals the strongly closed subalgebra that is generated by the image of $\contoCts$. Although material in this direction exists, we are not aware of a reference where this result in our generality is actually proved.\footnote{See \cref{3_rem:up-down} for further comments.}

\cref{3_sec:JBW-algebras} is concerned with JBW-algebras. As mentioned above, the classical `weak' approach to spectral theorems is not applicable here. Still, our methods apply. The introduction of a spectral measure in this context, which appears to be new, simplifies the picture and, as we believe, gives a better understanding. In the first part of the $20^{\mathrm{th}}$ century, spectral theorems for hermitian and normal operators on a Hilbert space were developed using spectral resolutions. Later, these were seen to be consequences of more general results on representations of unital commutative \Calgebras. Our approach in \cref{3_sec:JBW-algebras} is the analogue of this for JBW-algebras.

%%%%%%%%%%%%%%%%%%%%%%%%%%%%%%%%%%%%%%%%%%%%%%%%%%%%%%%%%%%%%%%%%%%%%%%%%%%%%%%%%%%%%%%%%%%%%%%%%%%%%%%%%%%%%%%%%%%%%%%%%%%%%%%%%%%%

\section{Preliminaries}\label{3_sec:preliminaries}

\noindent In this section, we collect the necessary notation, definitions, conventions, and preliminary results.

All vector spaces are over the real numbers, unless otherwise indicated.  Operators between two vector spaces are always supposed to be linear. An algebra homomorphism between two unital associative algebras need not be unital.

When $\os$ is a partially ordered set, we shall employ the usual notation in which $a_\lambda\uparrow$ means that $\net{a}$ is an increasing net in $\os$, and in which $a_\lambda\uparrow x$ means that $\net{a}$ is an increasing net in $\os$  with supremum $x$ in $\os$. The notations $a_\lambda\downarrow$ and $a_\lambda\downarrow x$ are similarly defined.

Suppose that $S$ is a non-empty subset of a partially ordered set $\os$. Then we shall say that \emph{$S^\vee$ exists in $\os$} when the supremum of finitely many arbitrary elements of $S$ exists in $\os$. In that case, we let $S^\vee$ denote the set consisting of all suprema of finitely many arbitrary elements of $S$. There are a similar definition and notation $S^\wedge$ for infima.

When $S$ is a subset of a set $\pset$, $\indicator{S}$ denotes its indicator function; we write $\zerofunction$ for $\indicator{\emptyset}$ and $\onefunction$ for $\indicator{\pset}$.

When $\ts$ is a topological space, we write $\contts$, $\contots$, and $\contcts$ for its (real-valued) continuous functions, its continuous functions that vanish at infinity, and its compactly supported continuous functions, respectively. Their respective complex-valued counterparts are denoted by $\contCts$, $\contoCts$, and $\contcCts$. When $S$ is a subset of $\ts$, we shall write $f\prec S$ to mean that $f\in\contcts$, that $\zerofunction\leq f\leq\onefunction$, and that $\supp f\subseteq S$; we shall write $S\prec f$ to mean that $f\in\contcts$, that $\zerofunction\leq f\leq\onefunction$, and that $f(x)=1$ for $x\in S$. The Borel $\sigma$-algebra of $\ts$ is the $\sigma$-algebra generated by the open subsets of $\ts$ and is denoted by $\borel$.

When $\os$ is a normed space, its norm dual will be denoted by $\ndualos$.

When $\hilbert$ is a Hilbert space, its inner product is denoted by $\lrinp{\,\cdot\,,\,\cdot\,}$. In the complex case, it is linear in the first variable. Its bounded operators are denoted by $\boundedh$.

When $\os$ is a non-empty set supplied with an equivalence relation, the set of equivalence classes corresponding to a subset $S$ of $E$ will be denoted by $\eclass{S}$ rather than the more customary $[S]$; this makes the formulas in \cref{3_sec:ups_and_downs} easier to read. When $x\in\os$, we write  $\eclass{x}$ for $\eclass{\{x\}}$.

 %%%%%%%%%%%%%%%%%%%%%%%%%%%%%%%%%%%%%%%%%%%%%%%%%%%%%%%%%%%%%%%%%%%%%%%%%%%%%%%%%%%%%%%%%%%%%%%%%%%%%%%%%%%%%%%%%%%%%%%%%%%%%%%%%%%%

%%%%%%%%%%%%%%%%%%%%%%%%%%%%%%%%%%%%%%%%%%%%%%%%%%%%%

\subsection{Partially ordered vector spaces}\label{3_subsec:partially_ordered_vector_spaces}

%%%%%%%%%%%%%%%%%%%%%%%%%%%%%%%%%%%%%%%%%%%%%%%%%%%%%

\noindent When $\os$ is a partially ordered vector space, we let $\pos{\os}$ denote its positive cone. We do not require that $\posos$ be generating, i.e., we do not require that $\os$ be directed, but we do require that $\posos$ be proper. It is always supposed that $\os$ is Archimedean: for all $x\in\posos$, $r_n x\downarrow 0$ whenever $\seq{r}$ is a sequence in $\posR$ such that $r_n\downarrow 0$.

When $\os$ and $F$ are vector spaces, $\linearop{\os,F}$ denotes the operators from $\os$ into $F$. An operator $T\colon \os\to\ostwo$ between two partially ordered vector spaces is \emph{positive} when $T(\pos{\os})\subseteq \pos{F}$, and \emph{regular} when it is the difference of two positive operators. We let $\regularop{\os,F}$ denote the vector space of regular operators from $\os$ into $F$. When $\pos{\os}$ is generating in $\os$, $\linearop{\os,F}$ and $\regularop{\os,F}$ are partially ordered vector spaces via their common positive cones $\pos{\regularop{\os,F}}$.

An operator $T\in\pos{\linearop{\os,F}}$ between two partially ordered vector spaces $\os$ and $\ostwo$ is called \emph{order continuous} (resp.\ \emph{$\sigma$-order continuous}) if $Tx_\lambda\downarrow 0$ in $F$ whenever $\net{x}$ is a net in $\os$ such that $x_\lambda\downarrow 0$ in $\os$ (resp.\ if $Tx_n\downarrow 0$ in $F$  whenever $\seq{x}$ is a sequence in $\os$ such that $x_n\downarrow 0$ in $\os$).\footnote{One can argue that it is better to speak of \emph{monotone} ($\sigma$-)order continuous operators, but we do not want to burden the terminology further than is necessary for our purposes.}  An operator in $\regularop{\os,F}$ is said to be order continuous (resp.\ $\sigma$-order continuous) if it is the difference of two positive order continuous (resp.\ $\sigma$-order continuous) operators.\footnote{When $\os$ and $\ostwo$ are both vector lattices, the present definitions specialise to those  in the literature for vector lattices; see \cite[Remark~3.5]{de_jeu_jiang:2022a}}. We let $\ocontop{\os,F}$ (resp.\ $\socontop{\os,F}$) denote the order continuous (resp.\ $\sigma$-order continuous) operators from $\os$ into $F$; they are linear subspaces of $\regularop{\os,F}$. When $\os$ is directed, they are partially ordered vector spaces with the positive order continuous (resp.\ $\sigma$-order continuous) operators as positive cones, which are generating by definition. We write $\linearop{E}$ for $\linearop{E,E}$, etc.; $\odualos$ for $\regularop{E,\RR}$; $\ocdualos$ for $\ocontop{\os,\RR}$; and $\socdualos$ for  $\socontop{\os,\RR}$.  When  $\os$ is  a Banach lattice, $\odualos$ coincides with the norm dual $\ndualos$ of $\os$.

\begin{definition}\label{3_def:normal_space}
	A partially ordered vector space $\os$ is called \emph{normal} when, for $x\in \os$, $\f{x,x^\prime}\geq 0$ for all $x^\prime\in\pos{(\ocdualos)}$ if and only if $x\in \pos{\os}$.\footnote{For a vector lattice $\os$, this is equivalent to the usual requirement that $\ocdualos$ separate the points of $\os$; see \cite[Lemma~3.7]{de_jeu_jiang:2022a}.}
\end{definition}

Clearly, when $\os$ is normal, $\pos{(\ocdualos)}$ separates the points of $\os$.

\medskip

A partially ordered vector space $\os$ is \emph{\mc} (resp.\ \emph{\smc}) when every increasing net (resp.\ sequence) in $\os$ that is bounded from above has a supremum; \emph{\Dc} when every non-empty subset of $\os$ that is bounded from above has a supremum; and \emph{\sDc} when every non-empty countable infinite subset of $\os$ that is bounded from above has a supremum.
As was observed in \cite[Lemma~1.1]{wright:1972}, every \smc\  partially vector space is automatically Archimedean.

For vector lattices, \Dc ness (resp.\ \sDc ness) and \mc ness (resp.\ \smc ness) are equivalent. If $\os$ has a generating positive cone and if $\os$ is \sDc, then, for $x_1,x_2\in\os$, the subset $\{x_1,x_2\}$ is bounded from above, so that it has a supremum. Hence $\os$ is then a vector lattice.

When $\os$ and $\ostwo$ are vector lattices, and $\ostwo$ is Dedekind complete, $\regularop{\os,\ostwo}$ is a Dedekind complete vector lattice. When $\os$ and $\ostwo$ are Banach lattices,  every regular operator from $\os$ into $\ostwo$ is continuous. When $\os$ and $\ostwo$ are Banach lattices where $\ostwo$ is Dedekind complete,  $\regularop{\os,\ostwo}$ is a Dedekind complete Banach lattice when supplied with the \emph{regular norm} $\norm{\,\cdot\,}_{\upr}$, defined by setting  $\norm{T}_{\upr}\coloneqq\norm{\abs{T}}$ for $T\in\regularop{\os,\ostwo}$; see \cite[Theorem~4.74]{aliprantis_burkinshaw_POSITIVE_OPERATORS_SPRINGER_REPRINT:2006}, for example.

\medskip

The condition that $\os$ be \mc\ as well as normal is an important one in this paper. The class of such spaces is, for practical purposes, rather large, and contains many spaces that are not vector lattices. It includes, e.g., the Banach lattices with order continuous norms; the regular operators on such Banach lattices; more generally: every subspace of $\linearop{\os,\ostwo}$ that contains $\regularop{\os,\ostwo}$, where $\os$ and $\ostwo$ are partially ordered vector spaces such that $\os$ is directed and $\ostwo$ is monotone complete and normal; the vector space consisting of all self-adjoint elements of a strongly closed complex linear subspace of $\boundedh$ for a complex Hilbert space $\hilbert$; JBW-algebras; and the regular operators on JBW-algebras. For this, and for more examples, we refer to  \cite[Section~3]{de_jeu_jiang:2022a}.

\medskip

In \cref{3_sec:riesz_representation_theorems_for_positive_algebra_homomorphisms}, the finiteness of a representing spectral measure can often be concluded when the codomain is a quasi-perfect partially ordered vector space.
This notion was introduced in \cite{de_jeu_jiang:2022b}; we recall its definition.

\begin{definition}\label{3_def:quasi_perfect_spaces}
	A partially ordered vector space $\os$ is \emph{quasi-perfect} when the two following conditions are both satisfied:
	\begin{enumerate}
		\item\label{3_part:quasi_perfect_spaces_1}
		$\os$ is normal;
		\item\label{3_part:quasi_perfect_spaces_2}
		if an increasing net $\net{x}$ in $\posos$ is such that $\rsup\,\f{x_\lambda,x^\prime}<\infty$ for each $x^\prime\in\pos{(\odualos)}$, then this net has a supremum in $\os$.
	\end{enumerate}
\end{definition}

Clearly,  a quasi-perfect partially ordered vector space  is monotone complete.
The terminology is motivated by an existing characterisation of perfect vector lattices. We recall that a vector lattice is called \emph{perfect} when the natural vector lattice homomorphism from $\os$ into $\ocdual{(\ocdualos)}$ is a surjective isomorphism. The following alternate characterisation, which we include for comparison, is due to Nakano; see \cite[Theorem~1.71]{aliprantis_burkinshaw_POSITIVE_OPERATORS_SPRINGER_REPRINT:2006}. It shows that a perfect vector lattice is a quasi-perfect partially ordered vector space.

\begin{theorem}\label{3_res:nakano}
	A vector lattice $\os$ is a perfect vector lattice if and only if	the following two conditions hold:
	\begin{enumerate}
		\item\label{3_part:nakano_1}
		$\os$ is normal;
		\item\label{3_part:nakano_2}
		if an increasing net $\net{x}$ in $\posos$ is such that $\rsup\,\f{x_\lambda,x^\prime}<\infty$ for each $x^\prime\in\pos{(\ocdualos)}$, then this net has a supremum in $\os$.
	\end{enumerate}
\end{theorem}

Quite a few spaces of practical interest are quasi-perfect. We give a number of examples in \cref{3_res:examples_of_quasi_perfect_spaces} and \cref{3_res:order_continuous_operators_are_quasi-perfect}; see also \cref{3_res:order_continuous_operators_are_suitable_algebra} and \cref{3_res:order_continuous_operators_on_Banach_lattices_are_suitable_algebra}.
As a preparation for some of our examples,  we recall that the norm on a Banach lattice is said to be a \emph{Levi norm} when every increasing norm bounded net in the positive cone has a supremum. It follows from the uniform boundedness principle that the norm on a Banach lattice $\os$ is a Levi norm precisely when $\os$ has the property in part~\ref{3_part:quasi_perfect_spaces_2} of \cref{3_def:quasi_perfect_spaces}.

A Banach lattice is a \emph{KB-space} when every increasing norm bounded net in the positive cone is norm convergent. KB-spaces have order continuous norms, and a reflexive Banach lattice is a KB-space; see \cite[p.~232]{aliprantis_burkinshaw_POSITIVE_OPERATORS_SPRINGER_REPRINT:2006}.
 It is not difficult to see that the KB-spaces are precisely the Banach lattices with Levi norms that are order continuous.

The following examples of quasi-perfect vector lattices are taken from \cite[Proposition~6.7]{de_jeu_jiang:2022b}.

\begin{proposition}\label{3_res:examples_of_quasi_perfect_spaces}
	The following spaces are quasi-perfect partially ordered vector spaces:
	\begin{enumerate}
		\item\label{3_part:examples_of_quasi_perfect_spaces_1}
		perfect vector lattices;
		\item\label{3_part:examples_of_quasi_perfect_spaces_2}
		normal Banach lattices with a Levi norm, such as KB-spaces and, still more in particular, reflexive Banach lattices;
		\item\label{3_part:examples_of_quasi_perfect_spaces_3}
		for \SOT-closed complex linear subspaces $L$ of $\boundedh$, where $\hilbert$ is a complex Hilbert space: the real vector spaces $L_\sa$ consisting of all self-adjoint elements of $L$;
		\item\label{3_part:examples_of_quasi_perfect_spaces_4} JBW-algebras.
	\end{enumerate}
\end{proposition}

Two other classes of quasi-perfect partially ordered vector spaces are in \cref{3_res:order_continuous_operators_are_quasi-perfect}. We start with by collecting a few properties of order continuous operators in the following result.

\begin{proposition}\label{3_res:order_continuous_operators_are_monotone_complete}
	Let $\os$ be a directed partially ordered vector space, and let $\ostwo$ be a \mc\ partially ordered vector space.
	\begin{enumerate}
		\item\label{3_part:order_continuous_operators_are_monotone_complete_1}
		If $\net{T}$ is a net in $\ocontop{\os,\ostwo}$ and $T_\lambda\uparrow T$ in $\regularop{\os,\ostwo}$ for some $T\in\regularop{\os,\ostwo}$,  then $T\in\ocontop{\os,\ostwo}$;
		\item\label{3_part:order_continuous_operators_are_monotone_complete_2}
		Let $\net{T}$ be a net in $\ocontop{\os,\ostwo}$, and let $T\in\ocontop{\os,\ostwo}$. Then $T_\lambda\uparrow T$ in $\ocontop{\os,\ostwo}$ if and only if $T_\lambda\uparrow T$ in $\regularop{\os,\ostwo}$;
		\item\label{3_part:order_continuous_operators_are_monotone_complete_3}
		$\ocontop{\os,\ostwo}$ is \mc;
		\item\label{3_part:order_continuous_operators_are_monotone_complete_4}
		Suppose, in addition, that $\ostwo$ is normal. Then $\ocontop{\os,\ostwo}$ is normal.
	\end{enumerate}
\end{proposition}

\begin{proof}
	Part~\ref{3_part:order_continuous_operators_are_monotone_complete_1} is well known when $\os$ and $\ostwo$ are vector lattices. The argument in that case (see \cite[Proof of Theorem~1.57]{aliprantis_burkinshaw_POSITIVE_OPERATORS_SPRINGER_REPRINT:2006}, for example) works equally well in the general case.
	
	The parts~\ref{3_part:order_continuous_operators_are_monotone_complete_2} and~\ref{3_part:order_continuous_operators_are_monotone_complete_3} follow from the combination of part~\ref{3_part:order_continuous_operators_are_monotone_complete_1} and the \mc ness of $\regularop{\os,\ostwo}$ (see \cite[Proposition~3.1]{de_jeu_jiang:2022a}).
	
	For part\ref{3_part:order_continuous_operators_are_monotone_complete_4}, we observe that the normality of $\ostwo$ implies that of $\regularop{\os,\ostwo}$; see \cite[Proposition~3.11]{de_jeu_jiang:2022a}. The normality of $\ocontop{\os,\ostwo}$ then follows from part~\ref{3_part:order_continuous_operators_are_monotone_complete_2}.
\end{proof}

Combining \cite[Proposition~3.11]{de_jeu_jiang:2022a} and \cref{3_res:order_continuous_operators_are_monotone_complete}, we have the following.

\begin{proposition}\label{3_res:regular_and_order_continuous_operators_are_monotone_complete}
	Let $\os$ be a directed partially ordered vector space, and let $\ostwo$ be a \mc\ and normal partially ordered vector space. Then $\regularop{\os,\ostwo}$ and $\ocontop{\os,\ostwo}$ are directed, \mc, and normal partially order vector spaces.
\end{proposition}

When~$F$ is quasi-perfect, one can do better.

\begin{proposition}\label{3_res:order_continuous_operators_are_quasi-perfect}
	Let $\os$ be a directed partially ordered vector space, and let $\ostwo$ be a quasi-perfect partially ordered vector space. Then $\regularop{\os,\ostwo}$ and $\ocontop{\os,\ostwo}$ are directed quasi-perfect partially ordered vector spaces.
\end{proposition}

\begin{proof}
We give the proof for $\ocontop{\os,\ostwo}$; the easier argument for $\regularop{\os,\ostwo}$ is similar.
Part~\ref{3_part:order_continuous_operators_are_monotone_complete_4} of \cref{3_res:order_continuous_operators_are_monotone_complete} shows that $\ocontop{\os,\ostwo}$ is normal. Suppose that $\net{T}$ is an increasing net in $\pos{\ocontop{\os,\ostwo}}$, and that $\sup_\lambda\f{T_\lambda,\varphi}<\infty$ for each $\varphi\in\pos{(\odual{\ocontop{\os,\ostwo}})}$. In particular, we then have that 	$\sup_{\lambda} \f{T_{\lambda}x,x'}<\infty$ for all  $x\in\pos{\os}$ and $x'\in\pos{(\odual{\ostwo})}$. Since $\ostwo$ is quasi-perfect, this implies that $\sup_\lambda T_\lambda x$ exists in $\ostwo$ for each $x\in\os$. By \cite[Proposition~3.1]{de_jeu_jiang:2022a}, the net $\net{T}$ has a supremum $T$ in $\regularop{\os,\ostwo}$. According to the  parts~\ref{3_part:order_continuous_operators_are_monotone_complete_1} and~\ref{3_part:order_continuous_operators_are_monotone_complete_2} of \cref{3_res:order_continuous_operators_are_monotone_complete}, $T$ is also the supremum of $\net{T}$ in $\ocontop{\os,\ostwo}$.
\end{proof}

%%%%%%%%%%%%%%%%%%%%%%%%%%%%%%%%%%

\subsection{The \csp}\label{3_subsec:the_countable_sup_property}

%%%%%%%%%%%%%%%%%%%%%%%%%%%%%%%%%%

Let $\os$ be a partially ordered vector space. Then $\os$ is said to have the \emph{\csp} when, for every net $\net{x}\subseteq\posos$ and $x\in\pos{\os}$ such that $x_\lambda\uparrow x$, there exists a countable set of indices $\{\lambda_n: n\geq 1 \}$ such that $x=\sup_{n\geq 1} x_{\lambda_n}$.\footnote{As in our definition of order continuous operators, we refrain from calling this the \emph{monotone} \csp.} In this case, there also always exist $\lambda_1\leq\lambda_2\leq\dotsb$ such that $x_{\lambda_n}\uparrow x$. For vector lattices, our \csp\ is equivalent to what is usually called the \csp\ in that context; namely, that every subset that has a supremum contains a countable subset with the same supremum.

The \csp\ is not only relevant to the properties of the $\Ell^1$-spaces that we shall introduce in \cref{3_subsec:measures_and_integrals} but, when combined with the monotone convergence theorem, it is also an essential ingredient to the proof of our main results on ups and downs, \cref{3_res:ups_and_downs,3_res:full_ups_and_downs}. It is for this reason that we mention a few facts here to show that this property is not at all uncommon, as it can often be obtained by pulling it back from a codomain via a strictly positive operator.

If $E$ is a \sDc\ vector lattice, $F$ is a \mc\ partially ordered vector space with the countable sup property, and $T\colon E\to F$ is a strictly positive $\sigma$-order continuous operator, then $E$ has the countable sup property, $E$ is Dedekind complete, and $T$ is order continuous; see \cite[Proposition~6.16]{de_jeu_jiang:2022a}.

For vector lattices, the situation is simpler. Suppose that there exists a strictly positive operator $\posmap\colon \os\to\ostwo$ between two vector lattices $\os$ and $\ostwo$. If $\ostwo$ has the \csp, then so does $\os$; see \cite[Theorem~1.45]{aliprantis_burkinshaw_LOCALLY_SOLID_RIESZ_SPACES_WITH_APPLICATIONS_TO_ECONOMICS_SECOND_EDITION:2003}. Consequently, a vector sublattice of a vector lattice with the \csp\ has the \csp; this also follows from \cite[Theorem~23.5]{luxemburg_zaanen_RIESZ_SPACES_VOLUME_I:1971}. As another consequence, every vector lattice that admits a strictly positive functional has the \csp. Consequently, every  separable Banach lattice has the \csp\ (see \cite[Exercise~4.1.4]{aliprantis_burkinshaw_POSITIVE_OPERATORS_SPRINGER_REPRINT:2006}).

Suppose that $\os$ is a separable Banach lattice, and that $F$ is a Dedekind complete normed vector lattice that admits a strictly positive continuous functional $\varphi$. Choose a sequence $\{e_n\}_{n=1}^\infty$ in $\os$ that is dense in the positive part of the unit ball of $\os$, and define the functional $T\mapsto  \sum_{n=1}^\infty 2^{-n}(Te_n,\varphi)$ on $\regularop{E,F}$. Then $\varphi$ is strictly positive, so that $\regularop{E,F}$ has the \csp. In particular, this is true when $E$ and $F$ are separable Banach lattices and $F$ is Dedekind complete.

If $\hilbert$ is a separable complex Hilbert space with orthonormal basis $\seq{e}$, then $T\mapsto \sum_{n=1}^\infty 2^{-n}\lrinp{ Te_n,e_n}$ is a strictly positive functional on $\boundedh_\sa$. Hence every linear subspace of $\boundedh_\sa$ that is a lattice in the restricted partial ordering has the \csp.

Suppose that $E$ and $F$ are vector lattices, that $F$ is Dedekind complete, that $E$ has a weak order unit $e$, and that $F$ admits a strictly positive order continuous functional. Then $T\mapsto (Te,\varphi)$ is a strictly positive functional on $\ocontop{E,F}$, so that $\ocontop{E,F}$ has the \csp.

Finally, a Banach lattice with an order continuous norm has the \csp; see \cite[Theorem~17.8]{zaanen_INTRODUCTION_TO_OPERATOR_THEORY_IN_RIESZ_SPACES:1997}, for example.

%%%%%%%%%%%%%%%%%%%%%%%%%%%%%%%%%%%%%%%%%%%%%%%%%%

\subsection{Partially ordered algebras}\label{3_subsec:partially_ordered_algebras}

%%%%%%%%%%%%%%%%%%%%%%%%%%%%%%%%%%%%%%%%%%%%%%%%%%

A \emph{partially ordered algebra} $\oa$ is an associative algebra that is also a partially ordered vector space such that $ab\in\posoa$ for all $a,b\in\posoa$. We do not suppose that $\oa$ has an identity element or, if so, that the identity element is positive. A partially ordered algebra that is also a vector lattice is a \emph{vector lattice algebra}. A \emph{normed vector lattice algebra} is a vector lattice algebra that is supplied with a norm satisfying $\norm{x}\leq\norm{y}$ whenever $x,y\in\oa$ are such that $\abs{x}\leq\abs{y}$. A possible identity element need not have norm 1. A \emph{Banach lattice algebra} is a normed vector lattice algebra with a complete norm.

If $\oa$ is a partially ordered algebra, then we shall say that the multiplication in $\oa$ is \emph{\mcont} if,  whenever $\net{a}\subseteq\posoa$ and $a\in\posoa$ are such that $a_\lambda\uparrow a$,  then $ba_\lambda\uparrow ba$ and $a_\lambda b\uparrow ab$ for all $b\in\posoa$; it is \emph{\smcont} if, whenever $\seq{a}\subseteq\posoa$ and $a\in\posoa$ are such that $a_n\uparrow a$, then $ba_n\uparrow ba$ and $a_nb\uparrow ab$ for all $b\in\posoa$. It is routine to verify that the analogous definitions using decreasing nets and sequences in $\posoa$ are equivalent to the above ones. % There are no requirements involving increasing nets or sequences in $\posoa$ that are unbounded.

\begin{remark}
	Suppose that $\oa$ is a vector lattice algebra. Then the multiplication in $\oa$ is \mcont\ (resp.\ \smcont) in our sense if and only if, for all $b\in\oa$, the multiplication operators $a\mapsto ba$ and $a\mapsto ba$ are order continuous (resp.\ $\sigma$-order continuous) operators on $\oa$ in the sense of \cite[p.~123]{zaanen_RIESZ_SPACES_VOLUME_II:1983}. This is an easy consequence of the fact that $\posoa$ generates $\oa$ and the observation that linear combinations of positive operators that are order continuous in the sense of \cite[p.~123]{zaanen_RIESZ_SPACES_VOLUME_II:1983} are again order continuous in the same sense.
\end{remark}

The terminology `monotone continuous multiplication', rather than `monotone left and right multiplications', is justified by the following result, which will be used in the proof of \cref{3_res:representing_measure_is_spectral} on spectral measures. Its proof is routine.

\begin{lemma}\label{3_res:multiplication_in_two_variables_is_monotone_order_continuous}
	Let $\oa$ be a partially ordered algebra with a \mcont\ multiplication. 	
	\begin{enumerate}
	\item\label{3_part:multiplication_in_two_variables_is_monotone_order_continuous_1}
If $\{a_{\lambda_1}\}_{\lambda_1\in\Lambda_1},\,\{b_{\lambda_2}\}_{\lambda_2\in\Lambda_2}\subseteq\posoa$, and $a,b\in\posoa$ are such that $a_{\lambda_1}\uparrow a$ and $b_{\lambda_2}\uparrow b$, then $a_{\lambda_1}b_{\lambda_2}\uparrow ab$;
	\item\label{3_part:multiplication_in_two_variables_is_monotone_order_continuous_2}
	If $\{a_{\lambda}\}_{\lambda\in\Lambda},\,\{b_{\lambda}\}_{\lambda\in\Lambda}\subseteq\posoa$, and $a,b\in\posoa$ are such that $a_{\lambda}\uparrow a$ and $b_{\lambda}\uparrow b$, then $a_{\lambda}b_{\lambda}\uparrow ab$.
	\end{enumerate}
	When the multiplication is \smcont, the analogous statements for the term-wise products of two sequences hold.
	
	The three analogous statements for decreasing nets and sequences in $\posoa$ are also true.		
\end{lemma}

The following result will be convenient later on.

\begin{lemma}\label{3_res:properties_inherited_by_algebras_associated_to_idempotent}
	Let $\oa$ be a partially ordered algebra with monotone continuous multiplication. Suppose that $p\in\posoa$ and that $p^2=p$.
	\begin{enumerate}
	\item\label{3_part:properties_inherited_by_algebras_associated_to_idempotent_1}
	If $\net{a}$ is a net in $\pos{(p\oa)}$ and $a_\lambda\uparrow a$  in $\oa$ for some $a\in\oa$, then $a\in \pos{(p\oa)}$.
	\end{enumerate}

	If, in addition, $\oa$ is \mc, then the following hold.
	\begin{enumerate}[resume]
	\item\label{3_part:properties_inherited_by_algebras_associated_to_idempotent_2}
	Let $\net{a}$ be a net in $\pos{(p\oa)}$ and let $a\in \pos{(p\oa)}$. Then $a_\lambda\uparrow a$ in $p\oa$ if and only if $a_\lambda\uparrow a$ in $\oa$;
	\item\label{3_part:properties_inherited_by_algebras_associated_to_idempotent_3}
	The subalgebra $p\oa$ of $\oa$ is a monotone complete partially ordered algebra with monotone continuous  multiplication.;
	
	\item\label{3_part:properties_inherited_by_algebras_associated_to_idempotent_3_extra} If $\oa$ is a normal partially ordered vector space, then so is $p\oa$;
	\item\label{3_part:properties_inherited_by_algebras_associated_to_idempotent_4}  Suppose, in addition, that $\oa$ has the \csp. Then so does $p\oa$.
	\end{enumerate}

Similar statements hold for $\oa p$ and $p\oa p$.
\end{lemma}

\begin{proof}
	We consider only the case $p\oa$; the other two are treated similarly.
	
	We prove part~\ref{3_part:properties_inherited_by_algebras_associated_to_idempotent_1}. Under the pertinent premises, it follows that $a_\lambda=pa_\lambda\uparrow pa$, so that $a=pa\in p\oa$.
	
	We prove part~\ref{3_part:properties_inherited_by_algebras_associated_to_idempotent_2}.
	Let the net $\net{a}$ in $\pos{(p\oa)}$ and $a\in\pos{(p\oa)}$ be such that $a_\lambda\uparrow a$ in $p\oa$. Since $\oa$ is \mc, there exists an $a^\prime\in\oa$ such that $a_\lambda\uparrow a^\prime$ in $\oa$. By  part~\ref{3_part:properties_inherited_by_algebras_associated_to_idempotent_1}, we have $a^\prime\in p\oa$. Hence $a=a^\prime$, so that $a_\lambda\uparrow a=a^\prime$ in $\oa$. The converse statement is trivial.
	
	The parts~\ref{3_part:properties_inherited_by_algebras_associated_to_idempotent_3},~\ref{3_part:properties_inherited_by_algebras_associated_to_idempotent_3_extra}, and~\ref{3_part:properties_inherited_by_algebras_associated_to_idempotent_3_extra} are now easy consequences of the parts part~\ref{3_part:properties_inherited_by_algebras_associated_to_idempotent_1} and ~\ref{3_part:properties_inherited_by_algebras_associated_to_idempotent_2}.
	
\end{proof}

The second main spectral theorem in \cref{3_sec:riesz_representation_theorems_for_positive_algebra_homomorphisms}, \cref{3_res:positive_homomorphisms_into_partially_ordered_algebras}, is in the context of \mc\ and normal partially ordered algebras with a monotone continuous multiplication. The monotone continuity of the multiplication may fail for algebras of general regular operators on partially ordered vector spaces, but for algebras of order continuous operators we have the following consequence of \cite[Proposition~3.1]{de_jeu_jiang:2022a} and  \cref{3_res:order_continuous_operators_are_monotone_complete,3_res:order_continuous_operators_are_quasi-perfect,3_res:regular_and_order_continuous_operators_are_monotone_complete}, exhibiting two major classes of algebras to which \cref{3_res:positive_homomorphisms_into_partially_ordered_algebras} applies.
	
\begin{proposition}\label{3_res:order_continuous_operators_are_suitable_algebra}
	Let $\os$ be a directed\ partially ordered vector space.
	\begin{enumerate}
		\item If $E$ is monotone complete and normal, then $\ocontop{E}$ is a directed, monotone complete, and normal partially ordered algebra with a monotone continuous multiplication.
		\item If $E$ is quasi-perfect, then $\ocontop{E}$ is a directed quasi-perfect partially ordered algebra with a monotone continuous multiplication.
	\end{enumerate}
\end{proposition}

\cref{3_res:order_continuous_operators_are_suitable_algebra} and \cref{3_res:examples_of_quasi_perfect_spaces} have the following consequence for Banach lattices.

\begin{corollary}\label{3_res:order_continuous_operators_on_Banach_lattices_are_suitable_algebra}
	Let $\os$ be a \Dc\ Banach lattice, and supply $\ocontop{E}$ with the regular norm.  If $\os$ is normal, i.e., is such that $\ocdualos$ separates the points of $\os$, then $\ocontop{\os}$ is a \Dc\ and normal Banach lattice algebra with a \mcont\ multiplication. If $\os$ is normal and the norm on $\os$ is a Levi norm, then $\ocontop{\os}$ is a \Dc\ quasi-perfect Banach lattice algebra with a \mcont\ multiplication.
\end{corollary}

Another important case where \cref{3_res:positive_homomorphisms_into_partially_ordered_algebras} applies is in the context of Hilbert spaces.

\begin{proposition}\label{3_res:riesz_algebra_for_hilbert_spaces} Let  $\hilbert$ be a complex Hilbert space, and let $\oa\subseteq\boundedh$ be a commutative \SOT-closed \Csubalgebra. Let $\oa_\sa$ be the real vector space that consists of the self-adjoint elements of $\oa$, supplied with the partial ordering that is inherited from the usual partial ordering on $\boundedh_\sa$. Then $\oa_\sa$ is a quasi-perfect Banach lattice algebra with a \mcont\ multiplication.
\end{proposition}

\begin{proof}
	There exists a locally compact Hausdorff space $\ts$ such that $\oa$ is isomorphic to $\contoCts$ as a \Calgebra. It is then clear that $\oa_\sa$ is a Banach lattice algebra. We know from part~\ref{3_part:examples_of_quasi_perfect_spaces_3} of \cref{3_res:examples_of_quasi_perfect_spaces} that $\oa_\sa$ is quasi-perfect.
	
	We turn to the multiplication. Suppose that $T\in\pos{\oa_\sa} $ and that $0\leq S_\lambda\uparrow S$ in $\oa_\sa$. Then $S=\SOTlim_{\lambda}S_\lambda$ by \cite[Proposition~3.2]{de_jeu_jiang:2022a}, and this implies that $TS=\SOTlim_{\lambda}TS_\lambda$. Since $0\leq TS_\lambda\uparrow\leq TS$ (this is clear in the $\contoCts$-model), \cite[Proposition~3.2]{de_jeu_jiang:2022a} implies that $TS_\lambda\uparrow TS$ in $\oa_\sa$. A similar argument shows that $S_\lambda T\uparrow ST$ in $\oa_\sa$. Hence the multiplication in $\oa_\sa$ is monotone order continuous.
\end{proof}

For an application of \cref{3_res:positive_homomorphisms_into_partially_ordered_algebras} to JBW-algebras, we record the following.

\begin{proposition}\label{res:JBW_algebra_monotone_continous_multiplication}
	Let~${\poalgfont M}$ be an associative JBW-algebra. Then ${\poalgfont M}$ is a directed quasi-perfect partially ordered algebra with a monotone continuous multiplication.
\end{proposition}

\begin{proof} In view of \cref{3_res:examples_of_quasi_perfect_spaces}, only the monotone continuity of the (commutative) multiplication needs proof.
	Suppose that $\net{a}\subseteq \pos{M}$ is a net and that $a_\lambda\uparrow a$ for some $a\in\pos{M}$. Take $b\in\pos{M}$. According to \cite[Proposition~2.5(ii)]{alfsen_shultz_GEOMETRY_OF_STATE_SPACES_OF_OPERATOR_ALGEBRAS:2003}, $a_\lambda\to a$ $\sigma$-strongly. By \cite[Proposition~2.4]{alfsen_shultz_GEOMETRY_OF_STATE_SPACES_OF_OPERATOR_ALGEBRAS:2003}, we also have $a_\lambda b\to ab$ $\sigma$-strongly. As $b\in\pos{M}$, the net $\{a_\lambda b\}$ is increasing and bounded from above by $ab$. Let $s$ be its supremum. Then $a_\lambda b\uparrow s$ implies that $a_\lambda b\to s$ $\sigma$-strongly by \cite[Proposition~2.5(ii)]{alfsen_shultz_GEOMETRY_OF_STATE_SPACES_OF_OPERATOR_ALGEBRAS:2003} again. Hence $ab=s$ and $a_\lambda b\uparrow ab$. The proof that $ba_\lambda\uparrow ba$ is similar.
\end{proof}

%%%%%%%%%%%%%%%%%%%%%%%%%%%%%%%%%%%%%%%%%%%%%%%%%%%%%%%%%%%%%%%

\subsection{Moduli preserving operators}\label{3_subsec:moduli_preserving_operators}

%%%%%%%%%%%%%%%%%%%%%%%%%%%%%%%%%%%%%%%%%%%%%%%%%%%%%%%%%%%%%%%

As we shall see in \cref{3_sec:relations_between_measures_algebra_homomorphisms_and_vector_lattice_homomorphisms}, a positive algebra homomorphism often gives rise to a vector lattice homomorphism. The present section provides the  necessary terminology and elementary preparatory results for this.

\begin{definition}\label{3_def:preserving_moduli}
	Let $\os$ and $\ostwo$ be partially ordered vector spaces, and let $T \colon \os\to\ostwo$ be an operator. Then $T $ \emph{preserves moduli} when, for all $x\in\os$ such that $\abs{x}$ exists in $\os$, $\abs{T x}$ exists in $\ostwo$ and $T \abs{x}=\abs{T x}$.
\end{definition}

A moduli preserving operator is positive. The moduli preserving operators between two vector lattices are precisely the vector lattice homomorphisms.

It is well known (and easy to see) that, for elements $x,y$ of a partially ordered vector space $\os$, the existence in $\os$ of $x\vee y$, $x\wedge y$, and $\abs{x-y}$ are all equivalent. In this case, we have $x+y=x\vee y + x\wedge y$, $x\vee y=\frac{1}{2}(x+y + \abs{x-y})$, and $x\wedge y=\frac{1}{2}(x+y-\abs{x-y})$. Using this, the following result is easily established, giving two equivalent definitions of moduli preserving operators.

\begin{lemma}\label{3_res:equivalent_definitions_of_lattice_preserving_maps}
	Let $\os$ and $\ostwo$ be partially ordered vector spaces, and let $T \colon \os\to \ostwo$ be an operator. The following are equivalent:
	\begin{enumerate}
		\item\label{3_part:equivalent_definitions_of_lattice_preservering_maps_3}
		$T$ preserves moduli;
		\item\label{3_part:equivalent_definitions_of_lattice_preserving_maps_1}
		whenever $x,y\in \os$ are such that $x\vee y$ exists in $\os$, $T x\vee T y$ exists in $\ostwo$, and $T (x\vee y)=T x\vee T y$;
		\item\label{3_part:equivalent_definitions_of_lattice_preserving_maps_2}
		whenever $x,y\in \os$ are such that $x\wedge y$ exists in $\os$, $T x\wedge T y$ exists in $\ostwo$, and $T (x\wedge y)=T x\wedge T y$.		
	\end{enumerate}
\end{lemma}

\begin{remark}\label{3_rem:binary_lattice_preserving_operators} Suppose that $T:\os\to\ostwo$ is a moduli preserving operator between two partially ordered vector spaces. Let $n\geq 3$,  and suppose that $x_1,\dotsc, x_n\in\os$ are such that $\sup\{x_1,\dotsc,x_n\}$ exists in $\os$. In this case, there is no guarantee that $\sup\{T x_1,\dotsc,T x_n\}$ exists in $\ostwo$. In the cases where we shall encounter moduli preserving operators, however, $\os$ is a vector lattice;  then an  inductive argument shows that $\sup\{T x_1,\dotsc,T x_n\}$ exists in $\ostwo$ and that also $T(\sup\{x_1,\dotsc,x_n\})=\sup\{T x_1,\dotsc,T x_n\}$ when  $n\geq 3$.
\end{remark}

The next result follows easily from the fact that the existence of suprema and infima and their values in a partially ordered vector space are compatible with translations.

\begin{lemma}\label{3_res:lattice_preserving_on_positive_cone_is_sufficient}
	Let $\os$ and $\ostwo$ be partially ordered vector spaces, where $\os$ is directed, and let $T \colon \os\to \ostwo$ be an operator. The following are equivalent:
	\begin{enumerate}
		\item\label{3_part:lattice_preserving_on_positive_cone_is_sufficient_1}
		$T$ preserves moduli;
		\item\label{3_part:lattice_preserving_on_positive_cone_is_sufficient_2} whenever $x,y\in \pos{\os}$ are such that $x\vee y$ exists in $\os$, $T x\vee T y$ exists in $\ostwo$, and $T (x\vee y)=T x\vee T y$;
		\item\label{3_part:lattice_preserving_on_positive_cone_is_sufficient_3} whenever $x,y\in \pos{\os}$ are such that $x\wedge y$ exists in $\os$, $T x\wedge T y$ exists in $\ostwo$, and $T (x\wedge y)=T x\wedge T y$.
	\end{enumerate}
\end{lemma}

The following result is a direct consequence of the definitions.

\begin{lemma}\label{3_res:lattice_preserving_map_give_vector_lattice_homomorphism}
	Let $\os$ be a vector lattice, let $\ostwo$ be a partially ordered vector space, and let $T \colon \os\to\ostwo$ be a moduli preserving operator. Supplied with the partial ordering inherited from $\ostwo$, the space $T(\os)$ is a vector lattice, and the map $T\colon \os\to T(E)$ is then a vector lattice homomorphism.
\end{lemma}

We shall apply \cref{3_res:lattice_preserving_map_give_vector_lattice_homomorphism} quite a few times in the sequel. Note, however, that some information regarding the ordering is lost when passing from its premises to its conclusion. Indeed, for $x\in\os$, the supremum of the set $\{T x,-T x\}$ even exists in $\ostwo$; it so happens that this supremum in the full space is already in $T(\os)$.  Similar remarks apply to $T x\vee T y$ and $T x\wedge T y$ for $x,y\in\os$. In the terminology of \cite[Definition~5.58]{abramovich_aliprantis_INVITATION_TO_OPERATOR_THEORY:2002}, $T(\os)$ is a lattice-subspace of $\ostwo$, but it is more than that.

\begin{remark}\label{3_rem:kalauch_stennder_van_gaans}
	A detailed investigation of the possible generalisations of the notion of a vector lattice homomorphism to the context of partially ordered vector spaces, and to that of pre-Riesz spaces in particular, is undertaken in \cite{kalauch_stennder_van_gaans:2021}. It is shown in \cite[Proposition~39]{kalauch_stennder_van_gaans:2021} that a positive operator between two pre-Riesz spaces $\os$ and $\ostwo$ preserves moduli if and only if it preserves disjointness on $\pos{\os}$. In spite of its naturality, the notion of a moduli preserving operator between general partially ordered vector spaces as in \cref{3_def:preserving_moduli} and the equivalences in \cref{3_res:equivalent_definitions_of_lattice_preserving_maps} appear to be new.
\end{remark}

%%%%%%%%%%%%%%%%%%%%%%%%%%%%%%%%%

\subsection{Measures and integrals}\label{3_subsec:measures_and_integrals}

%%%%%%%%%%%%%%%%%%%%%%%%%%%%%%%%%

In this section, we shall briefly outline part of the material in \cite{de_jeu_jiang:2022a} on measures with values in the extended positive cones of \smc\  partially ordered vector spaces, and on the associated (order) integrals. This extension of earlier work of Wright generalises the theory of the Lebesgue integral. We refer to  \cite[Section~7]{de_jeu_jiang:2022a} for a comparison with vector measures where it is argued, that, in the case of a \smc\ partially ordered Banach space, the measures and the (order) integrals as in the current section are a more convenient tool to work with than positive vector measures and their integrals.

Let $\os$ be a \smc\ partially ordered vector space. We adjoin a new element $\largest$ to $\os$, and extend the partial ordering from $\os$ to  $\osext\coloneqq\os\cup\{\infty\}$ by declaring that $x\leq \largest$ for all $x\in\osext$.
The addition on the extended positive cone $\pososext\coloneqq \posos\cup\{\infty\}$ is canonically defined, as well as the action of $\posR$ on $\pososext$. The elements of $\osext$ that are in $\os$ are called \emph{finite}.

The following definition is due to Wright; see \cite[p.~111]{wright:1969}. It generalises the notion of a measure with values in the extended positive real numbers.

\begin{definition}\label{3_def:positive_pososext_valued_measure}
	Let $\ms$ be a measurable space\footnote{In the earlier parts of \cite{de_jeu_jiang:2022a}, it was sufficient that $\alg$ be an algebra of subsets of $\pset$, but here we require it to be a $\sigma$-algebra from the outset.}, and let $\os$ be a \smc\  partially ordered vector space. An \emph{$\pososext$-valued measure} is a map $\npm\colon \alg\rightarrow \pososext$ such that:
	\begin{enumerate}
		\item $\npm(\emptyset)=0$;\label{3_part:pososext_valued_measure_1}
		\item If $\seq{\mss}$ is a pairwise disjoint sequence in $\alg$, then
		\begin{equation}\label{3_eq:sigma_additivity}
			\npm\left(\bigcup_{n=1}^\infty\mss_n\right)=\psup_{N=1}^\infty\sum_{n=1}^N\npm(\mss_n)
		\end{equation}
		in $\osext$.\label{3_part:pososext_valued_measure_2}
	\end{enumerate}
The quadruple $\msm$ is then a \emph{measure space}. If $\npm(\Omega)\subseteq \posos$ (equivalently: if $\npm(\pset)\in\posos$), then $\npm$ is called \emph{finite}, in which case we shall speak of an $\os$-valued measure.
\end{definition}

As a prelude to \cref{3_sec:special_positive_representations}, we collect a few results from \cite{de_jeu_jiang:2022a}. They show how the concept of a measure in an ordered context unites those of $\sigma$-additive measures in the strong operator topology for rather diverse spaces. Naturally, the same is true for spectral measures.

\begin{proposition}[{see \cite[Proposition~3.2 and Lemma~4.2]{de_jeu_jiang:2022a}}]\label{3_res:measures_with_values_in_L_sa}
	Let  $\hilbert$ be a complex Hilbert space, and let $L$ be a strongly closed complex linear subspace of $\boundedh$.  Let $L_\sa$ be the real vector space that consists of the self-adjoint elements of $L$, supplied with the partial ordering that is inherited from the usual partial ordering on $\boundedh_\sa$. Then $L_\sa$ is a monotone complete partially ordered vector space. Let $\ms$ be a measurable space, and let $\npm\colon \alg\to\pos{L_\sa}$ be a map such that $\npm(\emptyset)=0$.
	
	Then the following are equivalent:
	
	\begin{enumerate}
		\item\label{3_part:measures_with_values_in_L_sa_1}
		$\npm$ is a finite $\pos{L_\sa}$-valued measure in the sense of \cref{3_def:positive_pososext_valued_measure};
		
		\item\label{3_part:measures_with_values_in_L_sa_2}
		If $\seq{\mss}$ is a pairwise disjoint sequence in $\alg$, then $\npm\left(\bigcup_{n=1}^\infty\mss_n\right)x=\sum_{n=1}^\infty \npm(\mss_n) x$ in the norm topology of $\hilbert$ for all $x\in \os$.
	\end{enumerate}
	
\end{proposition}

\begin{proposition}[{see \cite[Lemma~4.3]{de_jeu_jiang:2022a}}]\label{3_res:measures_with_values_in_the_regular_operators}
	Let $\os$ be a Banach lattice with an order continuous norm. Then  $\regularop{\os}$ is a \Dc\ vector lattice. Let $\ms$ be a measurable space,  and let $\npm\colon \alg\to\pos{\regularop{\os}}$ be a map such that $\npm(\emptyset)=0$. The following are equivalent:
	
	\begin{enumerate}
		\item\label{3_part:measures_with_values_in_the_regular_operators_1}
		$\npm$ is a finite $\pos{\regularop{\os}}$-valued measure in the sense of \cref{3_def:positive_pososext_valued_measure};
		
		\item\label{3_part:measures_with_values_in_the_regular_operators_2}
		If $\seq{\mss}$ is a pairwise disjoint sequence in $\alg$, then $\npm\left(\bigcup_{n=1}^\infty\mss_n\right)x=\sum_{n=1}^\infty \npm(\mss_n) x$ in the norm topology of $\os$ for all $x\in \os$.
	\end{enumerate}
\end{proposition}

In a topological context, we distinguish various regularity properties of measures. Since the terminology in the literature is not entirely uniform, we mention them explicitly.

\begin{definition}\label{3_def: regularity of measures}
	When $\ts$ is a locally compact Hausdorff space, we let $\borel$ denote its Borel $\sigma$-algebra. Let $\os$ be a \mc\ partially ordered vector space, and let $\npm:\borel\to\pososext$ be a measure. Then $\npm$ is called:
	\begin{enumerate}
		\item a \emph{Borel measure \uppars{on $\ts$}} if $\npm(K)\in \os$ for all compact subset $K$;
		\item \emph{inner regular at $\mss\in\borel$} if $\npm(\mss)=\psup\{\npm(K): K\ \textup{is compact and}\ K\subseteq \mss\}$ in $\osext$;	
		\item \emph{outer regular at $\mss\in\borel$} if $\npm(\mss)=\pinf\{\npm(V): V\ \textup{is open and}\ \mss\subseteq V\}$ in $\osext$;
		\item a \emph{regular Borel measure \uppars{on $\ts$}} if $\npm$ is a Borel measure on $\ts$ that is inner regular at all open subsets of $\ts$ and outer regular at all Borel sets.
	\end{enumerate}
\end{definition}

Let $\msm$ be a measure space. A  (finite-valued) measurable function $\varphi\colon \pset\to\posR$ is an \emph{elementary function} when it  takes only finitely many values. It can be written (not generally uniquely) as a finite sum $\varphi=\sum_{i=1}^n r_i\indicator{\mss_i}$ for some $r_1,\dotsc,r_n\in\posR$ and $\mss_1,\dots,\mss_n\in\alg$. Here the $r_i$ are all finite, but it is allowed that $\npm(\mss_i)=\infty$ for some of the $\mss_i$. We let $\elemfun$ denote the set of elementary functions.

If $\varphi=\sum_{i=1}^n r_i\indicator{\mss_i}$ is an elementary function, where the $\mss_i$ have been chosen pairwise disjoint, then we define its (order) integral, which is an element of $\pososext$, by setting
\[
\ointm{\varphi}\coloneqq\sum_{i=1}^n r_i\npm(\mss_i).
\]
This is well defined.  For a measurable function $f\colon \pset\to\posRext$, we choose a sequence $\seq{\varphi}\subseteq\elemfun$ such that, for all $x\in\pset$,  $\varphi_n(x)\uparrow f(x)$ in $\posRext$.  We define the (order) integral of $f$, which is an element of $\pososext$, by setting
\[
\ointm{f}\coloneqq \psup_{n=1}^\infty \ointm{\varphi_n}.
\]
The value of the integral does not depend on the choice of the sequence $\seq{\varphi}$.

We let $\integrablefun$ denote the set of all (finite-valued) measurable functions $f\colon \pset\to\RR$ such that $\ointm{\abs{f}}$ is finite; its positive cone is denoted by $\posintegrablefun$. We write $\integrableelemfun$ for the elementary functions with finite integral. For $f\in\integrablefun$, we define $\ointm{f}$ by splitting $f$ into its positive and negative parts. $\integrablefun$ is a \sDc\ vector lattice, and the positive operator $\opintm\colon \integrablefun\to\os$, defined by setting $\opintm{f}\coloneqq\ointm{f}$ for $f\in\integrablefun$, is  a $\sigma$-order continuous operator from $\integrablefun$ into $\os$;  see  \cite[Proposition~6.14]{de_jeu_jiang:2022a}.  We shall use the same notation $\opintm$ to denote the restrictions of the original map to subspaces of $\integrablefun$, and also to denote the induced maps on quotients of such subspaces.

We let $\aezerofun$ denote the order ideal of $\integrablefun$ that consists of the measurable functions that vanish $\npm$-almost everywhere, and we let $\ellone$ denote the resulting quotient vector lattice. According to \cite[Theorem~6.17]{de_jeu_jiang:2022a}, $\ellone$ is a $\sigma$-Dedekind complete vector lattice, and the map $\opintm\colon \ellone\to\os$ is strictly positive and $\sigma$-order continuous. When $\os$ is \mc\ and has the \csp, $\ellone$ is a \Dc\  vector lattice with the \csp. Moreover, $\opintm$ is then order continuous.

We shall write $\boundedmeasfun$ for the bounded measurable functions on $\pset$; $\posboundedmeasfun$ for its positive cone; $\boundedmeasfunae$ for its quotient with respect to $\aezerofun\cap\boundedmeasfun$; and $\posboundedmeasfunae$ for the positive cone of $\boundedmeasfunae$. The essential supremum norm can be defined on $\boundedmeasfunae$ as in the case where $\os=\RR$, and $\boundedmeasfunae$ is then a Banach lattice algebra that is isometrically isomorphic to $\cont{\ts}$ for a compact Hausdorff space $\ts$, unique up to homeomorphism.

In \cite[Section~6.2]{de_jeu_jiang:2022b}, the monotone convergence theorem,  Fatou's lemma, and (when $\os$ is \sDc) the dominated convergence theorem for the order integral are established. Although the space $\integrablefun$ consists of finite-valued functions, the monotone convergence theorem is valid for functions with values in the extended positive real numbers. We shall benefit from this in the proof of  \cref{3_res:pulling_back_ups_and_downs} on ups and downs.

Finally, for later use in the context of JBW-algebras, we record the following on image measures. The proofs are analogous to those in \cite[\textsection 19]{bauer_MEASURE_AND_INTEGRATION_THEORY:2001}.

\begin{proposition}\label{3_res:image_measure}
	Let $(\pset,\alg)$ and $(\pset\sp\prime,\alg\sp\prime)$ be measurable spaces, and let $\Psi\colon\pset\to\pset\sp\prime$ be $\alg\text{-}\alg\sp\prime$-measurable. Let $\os$ be a \smc\ partially ordered vector space, and let $\npm\colon \alg\to\pososext$ be a measure. Define $\npm\circ\Psi\sp{-1}\colon \alg\sp\prime\to\pososext$ by setting $\npm\circ\Psi\sp{-1}(\mss\sp\prime)\coloneqq\npm(\Psi\sp{-1}(\mss\sp\prime))$ for $\mss\sp\prime\in\alg\sp\prime$. Then $\npm\circ\Psi\sp{-1}$ is a measure on $\alg\sp\prime$. If $f\sp\prime\colon\pset\sp\prime\to\posRext$ is $\alg\sp\prime$-measurable, then
	\begin{equation}\label{3_eq:image_measure}
	\int_{\pset\sp\prime}\sp{\upo}\! f\sp\prime\di(\npm\circ\Psi\sp{-1})=\ointm{f\sp\prime\circ\Psi}
	\end{equation}
	in $\pososext$. If $f\sp\prime\colon\pset\sp\prime\to\RR$ is $\alg\sp\prime$-measurable, then $f\sp\prime\in{\lebfont L}\sp 1(\pset\sp\prime,\alg\sp\prime,\npm\circ\Psi^{-1};\RR)$ if and only if
	$f\circ\Psi\in{\lebfont L}\sp 1(\pset,\alg,\npm;\RR)$, in which case \cref{3_eq:image_measure} holds in $\os$.
\end{proposition}

%%%%%%%%%%%%%%%%%%%%%%%%%%%%%%%%%

\subsection{Spectral measures}\label{3_subsec:spectral_measures}

%%%%%%%%%%%%%%%%%%%%%%%%%%%%%%%%%

Of particular interest in this paper are measures that take their values in partially ordered algebras. Motivated by the terminology in the literature for measures of various sorts that take their values in algebras of operators, we introduce the following terminology for measures that take their values in (partially ordered) algebras. %To do this, we extend the multiplication in $\pos{\oa}$ to $\posoaext$ by setting $0\cdot\infty=\infty\cdot 0\coloneqq 0$ and $x\cdot\infty=\infty\cdot x\coloneqq \infty$ for $x\in\posoaext\setminus\{0\}$.

\begin{definition}\label{3_def:spectral_measures}
Let $\ms$ be a measurable space, and let $\oa$ be a \smc\ partially ordered algebra. A measure $\psm\colon \alg\to\posoaext$ is a \emph{spectral measure} when $\psm(\mss_1\cap\mss_2)=\psm(\mss_1)\psm(\mss_2)$ for all $\mss_1,\mss_2\in\alg$ with finite measure.
\end{definition}

\begin{remark}\label{3_rem:spectral_measure_with_infinite_values}\quad
	\begin{enumerate}
		\item\label{3_part:spectral_measure_with_infinite_values_1}
		For a spectral measure, the finite elements among the $\psm(\mss)$ for $\mss\in\alg$ form a family of commuting idempotents.  Somewhat surprisingly, when an $\posoaext$-valued measure is actually finite, this property already implies that the measure is a spectral measure; see \cref{3_res:idempotents_form_boolean_algebra} for an even larger context in which this equivalence holds. This property may be easier to verify. In fact, we shall do just that  in the proof of the key result \cref{3_res:representing_measure_is_spectral}.
		\item\label{3_part:spectral_measure_with_infinite_values_2}
		Infinite spectral measures exist.  As an example, one can take a two-point space $\pset=\{x,y\}$, take $\alg=2^\pset$, take $\oa=\RR$, and let $\psm$ be the measure on $\alg$ such that $\psm(\{x_1\})=1$ and $\psm(\{x_2\})=\infty$.
		\item\label{3_part:spectral_measure_with_infinite_values_3}
		As a thought experiment, one can supply $\posoaext$ with a multiplication in the canonical fashion, and then require that $\psm(\mss_1\cap\mss_2)=\psm(\mss_1)\psm(\mss_2)$ for \emph{all} $\mss_1,\mss_2\in\alg$.  When $\psm(\pset)=\infty$, this implies that $\psm(\mss)=\psm(\mss\cap\pset)=\psm(\mss)\cdot\infty$ for $\mss\in\alg$. It follows from this that $\psm(\mss)\in\{0,\infty\}$ for all $\mss\in\alg$. As a consequence, all integrable functions have zero order integral. This means that the only such spectral measures that can lead to anything interesting at all, are actually the finite ones, in which case they are then spectral measures as in \cref{3_def:spectral_measures}. Hence such a more stringent multiplicativity condition is not imposed.
		\item\label{3_part:spectral_measure_with_infinite_values_4}
		Suppose that, in the context and notation of \cref{3_res:image_measure}, $\os$ is a \smc\ partially ordered algebra. If $\npm$ is a spectral measure, then so is the image measure $\npm\circ\Psi\sp{-1}$.
	\end{enumerate}
\end{remark}

For general measures, finite or infinite, an integrable function need not be essentially bounded. Also for spectral measures, finite or infinite, this still need not be the case, as is shown by the following example.

\begin{example}\label{3_ex:integrable_function_for_spectral_measure_need_not_be_bounded}
	Let $\pset$ be an infinite set, and let $\mathrm{Fun}(X;\RR)$ denote the \Dc\ vector lattice of all real-valued functions on $\pset$. Set $\oa\coloneqq\regularop{\mathrm{Fun}(X;\RR)}$. Then $\oa$ is a \Dc\ vector lattice algebra. Take $\alg\coloneqq 2\sp X$.  For $\mss\in\alg$, let $\npm(\mss)\in\oa $ be the pointwise multiplication by the characteristic function of $\mss$. Then $\npm\colon \alg\to\posoa$ is a finite spectral measure and $\integrablefun$ consists of \emph{all} functions on $\pset$: for $f\in\integrablefun$, $\ointm{f}\in\oa$ is the pointwise multiplication by $f$. Since the empty set is the only subset of $\pset$ of measure zero and $\pset$ is infinite, there are elements of $\integrablefun$ that are not essentially bounded.
	
	To get an example where the spectral measure is infinite, we fix a point $x_0$ in $\pset$, and define $\npm^\prime\colon \alg\to\posoaext$ by letting $\npm^\prime$ be the multiplication by the characteristic function of $\mss$ when $x_0\notin\mss$, and setting $\npm^\prime(\mss)\coloneqq \infty$ when $x_0\in\mss$. Then $\npm^\prime$ is an infinite spectral measure, and $\integrablefunprime$ consists of all functions on $\pset$ that vanish at $x_0$: for $f\in\integrablefunprime$, $\ointmprime{f}\in\oa$ is, again, the pointwise multiplication by $f$. As above, there are elements of $\integrablefunprime$ that are not essentially bounded.
\end{example}

In view of \cref{3_ex:integrable_function_for_spectral_measure_need_not_be_bounded}, the following result for \emph{normed} algebras is, perhaps, a pleasant surprise. It is the counterpart of \cite[Proposition~V.4]{ricker_OPERATOR_ALGEBRAS_GENERATE_BY_COMMUTING_PROJECTIONS:1999}, which states that an integrable function with respect to a spectral measure with values in the bounded operators on a Banach space (see  \cite[Definition~III.2 and Definition~III.5]{ricker_OPERATOR_ALGEBRAS_GENERATE_BY_COMMUTING_PROJECTIONS:1999} for definitions) is automatically essentially bounded. The proof in our (rather different) context is considerably simpler.

\begin{lemma}\label{3_res:integrable_function_for_spectral_measure_is_bounded}
	Let $\ms$ be a measurable space, and let $\oa$ be a \smc\ partially ordered normed algebra such that $\norm{x}\le\norm{y}$ for all $x,y\in\oa$ with $0\leq x\leq y$. Let $\npm\colon \alg\to\posoaext$ be a measure such that $\npm(\mss)^2=\npm(\mss)$ for all $\mss\in\alg$ with finite measure. Then every element of $\integrablefun$ is essentially bounded.
\end{lemma}	

\begin{proof}
	Take an $f\in\integrablefun$. For $n\geq 1$, set $\mss_n\coloneqq\{x\in\pset: \abs{f(x)}\geq n\}$. Then the monotonicity of the norm on $\posoa$ implies that
	\[
	\norm{\npm(\mss_n)}\leq \frac{1}{n}\lrnorm{\,\ointm{\abs{f}}\,}
	\]
	for $n\geq 1$. As each $\npm(\mss_n)$ is an idempotent, its norm is either zero or at least one. Hence $\npm(\mss_n)=0$ for all sufficiently large $n$.
\end{proof}

\begin{remark}\label{3_rem:embedding_of_ellone}
When every  element of $\integrablefun$ is essentially bounded, there is a canonically defined embedding of  $\ellone$ as a vector sublattice of $\boundedmeasfunae$. We shall allow ourselves to write $\ellone\subseteq\boundedmeasfunae$ in this case.
\end{remark}

%%%%%%%%%%%%%%%%%
\subsection{An embedding as a closed subalgebra}\label{3_subsec:a_closed_range_theorem}
It is known that, for a locally compact Hausdorff space $\ts$, the supremum norm is the minimal algebra norm on $\contots$ and $\contoCts$. We refer to \cite[Theorem~6.2]{kaplansky:1949} for this fact; the special case of $\contCts$ for compact $\ts$ is also covered by \cite[Theorem~1.2.4]{sakai_C-STAR-ALGEBRAS_AND_W-STAR-ALGEBRAS:1971}.   The alternate norm need not be complete, nor need it be unital when $\ts$ is compact. It is only supposed that $\norm{fg}\leq \norm{f}\norm{g}$ for $f,g\in\contots$ or $\contoCts$. This has the following consequence.

\begin{proposition}\label{3_res:topological_embedding_with_closed_image}
	Let $\ts$ be a locally compact Hausdorff space, let $\oa$ be a real normed algebra, and let $\posmap\colon  \contots\to\oa$ be a  continuous algebra homomorphism. Let $q\colon \contots\to\contots/\Ker \posmap$ be the quotient map, and let $\overline{\posmap}\colon  \contots/\Ker\posmap\to\oa$ be the algebra homomorphism such that $\posmap=\overline{\posmap}\circ q$. Then we have  $\norm{q(f)}\leq\norm{\overline{\posmap}(q(f))}\leq\norm{\posmap}\,\norm{q(f)}$ for $f\in \contots$. Consequently, $\overline{\posmap}$ is a topological embedding of $\contots/\Ker \posmap$  as the subalgebra $\posmap(\contots)$ of $\oa$, where $\posmap(\contots)$ is closed in $\oa$.
\end{proposition}

We mention explicitly that $\oa$ is not required to be complete or unital; when it is unital, it is not required that its identity element have norm 1; when $\oa$ is unital and $\ts$ is compact, it is not required that $\posmap$ be unital.

\begin{proof}
	All will be clear once we know that $\norm{q(f)}\leq\norm{\overline{\posmap}(q(f))}$ for $f\in\contots$. For this, we note that $\Ker\posmap\oplus{\upi}\Ker\posmap$ is a closed ideal of the complex \Calgebra\ $\contoCts$. A moment's thought shows that the canonical map from $\contots/\Ker\posmap$ into $\contoCts/(\Ker\posmap\oplus{\upi}\Ker\posmap)$ is an isometric embedding of $\contots/\Ker\posmap$ as the self-adjoint part of the commutative complex \Calgebra\ $\contoCts/(\Ker\posmap\oplus{\upi}\Ker\posmap)$. Hence $\contots/\Ker\posmap$ in its quotient norm is isometrically isomorphic to $\conto{\ts\sp\prime}$ for some locally compact Hausdorff space $\ts\sp\prime$. Since $q(f)\mapsto\norm{\overline\posmap (q(f))}$ provides an algebra norm on $\contots/\Ker\posmap$, the minimality of the supremum norm on $\conto{\ts\sp\prime}$ mentioned above implies the desired inequality.	
\end{proof}

The analogous statement for $\contoCts$ is obviously also valid, with a yet easier proof.

%%%%%%%%%%%%%%%%%%%%%%%%%%%%%%%%%%%%%%%%%%%%%%%%%%%%%%%%%%%%%%%%%%
 \section{Commuting idempotents and spectral measures}\label{3_sec:commuting_idempotents_and_spectral_measures}

%%%%%%%%%%%%%%%%%%%%%%%%%%%%%%%%%%%%%%%%%%%%%%%%%%%%%%%%%%%%%%%%%%
\noindent As indicated in the introduction, our approach to obtain a representing spectral measure for a positive algebra homomorphism $\posmap$ from $\contcts$ or $\contots$ into a partially ordered algebra consists of two steps. The first step is to invoke a Riesz representation theorem for positive operators from \cite{de_jeu_jiang:2022b}, showing that there is a representing measure for $\posmap$. The second step is to use the multiplicativity of $\posmap$ to show that this representing measure is, in fact, a spectral measure. This second step follows from the key result \cref{3_res:representing_measure_is_spectral} in this section. Its proof is easier than one might perhaps expect.

We need the following preparatory result for the proof of \cref{3_res:representing_measure_is_spectral}. The, perhaps, somewhat surprising equivalence of its parts~\ref{3_part:idempotents_form_boolean_algebra_1} and~\ref{3_part:idempotents_form_boolean_algebra_2} and the conditional validity of part~\ref{3_part:idempotents_form_boolean_algebra_a} also hold in the absence of a partial ordering. Although we shall apply it in the context of associative algebras, neither the associativity of the multiplication nor the vector space structure is needed for its proof. In view of its rather basic nature, we have formulated it under minimal hypotheses. It applies to what could be called  (not necessarily associative) partially ordered $\ZZ$-algebras that have no additive 2-torsion. Its parts~\ref{3_part:idempotents_form_boolean_algebra_b} and~\ref{3_part:idempotents_form_boolean_algebra_c} already hint at the moduli preserving properties of algebra homomorphisms in \cref{3_sec:relations_between_measures_algebra_homomorphisms_and_vector_lattice_homomorphisms}.

\begin{proposition}\label{3_res:idempotents_form_boolean_algebra}
	Let $\pset$ be a non-empty set, and let $\alg$ be an algebra of subsets of $\pset$. Let $\oa$ be an abelian group that has no elements of order 2. Suppose that $\oa$ is supplied with a translation invariant partial ordering such that $\posoa+\posoa\subseteq\posoa$, and with a bi-additive map $(a_1,a_2)\mapsto a_1a_2$ from $\oa\times\oa$ into $\oa$ such that $\posoa\posoa\subseteq\posoa$. Let $\psm\colon \alg\to\posoa$ be a map such that
	\begin{itemize}
		\item $\psm(\emptyset)=0$;
		\item $\psm\left(\bigcup_{i=1}^n\mss_i\right)=\sum_{i=1}^n \psm(\mss_i)$ whenever $\mss_1,\dotsc,\mss_n\in\alg$ are pairwise disjoint.
	\end{itemize}
	Then the following are equivalent:
	\begin{enumerate_arabic}
		\item\label{3_part:idempotents_form_boolean_algebra_1}   $\psm(\mss)^2=\psm(\mss)$ for $\mss\in\alg$ and $\psm(\mss_1)\psm(\mss_2)=\psm(\mss_2)\psm(\mss_1)$ for $\mss_1,\mss_2\in\alg$;
		\item\label{3_part:idempotents_form_boolean_algebra_2}
		$\psm(\mss_1\cap\mss_2)=\psm(\mss_1)\psm(\mss_2)$ for $\mss_1,\mss_2\in\alg$.
	\end{enumerate_arabic}

	Suppose that this is the case. Set $\leftidalgnonassociative{\psm(\pset)}{\oa}\coloneqq\{a\in\oa: \psm(X) a=a\}$ and $\rightidalgnonassociative{\psm(\pset)}{\oa}\coloneqq\{a\in\oa:  a\psm(X)=a\}$. Then:
	\begin{enumerate_alpha}
		\item\label{3_part:idempotents_form_boolean_algebra_a}
		$\psm(\alg)\subseteq\tposskip\leftidalgnonassociative{\psm(\pset)}{\oa}\cap\rightidalgnonassociative{\psm(\pset)}{\oa}$;
		\item\label{3_part:idempotents_form_boolean_algebra_b}
		for $\mss_1,\mss_2\in\alg$, $\psm(\mss_1)\vee\psm(\mss_2)$ exists in 	
		$\leftidalgnonassociative{\psm(\pset)}{\oa}$,  $\rightidalgnonassociative{\psm(\pset)}{\oa}$, $\leftidalgnonassociative{\psm(\pset)}{\oa}\cap\rightidalgnonassociative{\psm(\pset)}{\oa}$; it equals $\psm(\mss_1\cup\mss_2)$ in all cases;
		\item\label{3_part:idempotents_form_boolean_algebra_c}
		for $\mss_1,\mss_2\in\alg$, $\psm(\mss_1)\wedge\psm(\mss_2)$ exists in $\leftidalgnonassociative{\psm(\pset)}{\oa}$,  $\rightidalgnonassociative{\psm(\pset)}{\oa}$, $\leftidalgnonassociative{\psm(\pset)}{\oa}\cap\rightidalgnonassociative{\psm(\pset)}{\oa}$, and $\npm(\alg)$; it equals $\psm(\mss_1\cap\mss_1)$ in all cases;
		\item\label{3_part:idempotents_form_boolean_algebra_d}  		
		when supplied with the partial ordering inherited from $\oa$, the set $\psm(\alg)$ is a Boolean algebra with largest element $\psm(X)$ and smallest element 0; the complement of $a\in\psm(\alg)$ is $\psm(\pset)-a$. The map $\psm\colon \alg\to\psm(\alg)$ is a homomorphism of Boolean algebras.
	\end{enumerate_alpha}
\end{proposition}

\begin{proof} We prove that part~\ref{3_part:idempotents_form_boolean_algebra_1} implies part~\ref{3_part:idempotents_form_boolean_algebra_2}.
	First, take $\mss_1,\mss_2\in\alg$ such that $\mss_1\cap\mss_2=\emptyset$. Then
	\begin{align*}
		\psm(\mss_1)+\psm(\mss_2)&=\psm(\mss_1\cup\mss_2)\\
		&=\big(\psm(\mss_1\cup\mss_2)\big)^2\\
		&=\big(\psm(\mss_1)+\psm(\mss_2)\big)^2\\
		&=\psm(\mss_1)+2\psm(\mss_1)\psm(\mss_2)+\psm(\mss_2).
	\end{align*}
	Hence $\psm(\mss_1)\psm(\mss_2)=0$ whenever $\mss_1,\mss_2\in\alg$ are disjoint. Next, take $\mss_1,\mss_2\in\alg$ arbitrary. Using what we have just established, we see that
	\begin{align*}
		\psm(\mss_1)\psm(\mss_2)&=\big(\psm(\mss_1\setminus(\mss_1\cap\mss_2))+\psm(\mss_1\cap\mss_2)\big)\, \big(\psm(\mss_2\setminus(\mss_1\cap\mss_2))\\
		&\phantom{=}\quad+\psm(\mss_1\cap\mss_2)\big)\\
		&=\psm(\mss_1\setminus(\mss_1\cap\mss_2)) \psm(\mss_2\setminus(\mss_1\cap\mss_2))\\
		&\phantom{=}\quad+\psm(\mss_1\setminus(\mss_1\cap\mss_2)) \psm(\mss_1\cap\mss_2)\\
		&\phantom{=}\quad + \psm(\mss_1\cap\mss_2) \psm(\mss_2\setminus(\mss_1\cap\mss_2))+\psm(\mss_1\cap\mss_2)\,\psm(\mss_1\cap\mss_2)\\
		&=0+0+0+\psm(\mss_1\cap\mss_2)^2\\
		&	=\psm(\mss_1\cap\mss_2),
	\end{align*}
	as required. It is clear that part~\ref{3_part:idempotents_form_boolean_algebra_2} implies part~\ref{3_part:idempotents_form_boolean_algebra_1}.
	
	Suppose that the properties in the parts~\ref{3_part:idempotents_form_boolean_algebra_1} and~\ref{3_part:idempotents_form_boolean_algebra_2} are valid.
	
	Then part~\ref
	{3_part:idempotents_form_boolean_algebra_a} follows from part~\ref{3_part:idempotents_form_boolean_algebra_2}
	
	We turn to part~\ref{3_part:idempotents_form_boolean_algebra_b}. We consider the supremum in $\{a\in\oa: a\psm(X)=a\}$; the other case is handled similarly, and each of these implies the statements in the remaining two cases. Take $\mss_1,\mss_2\in\alg$.  Then $\psm(\mss_1\cup\mss_2)\geq\psm(\mss_1)$ and $\psm(\mss_1\cup\mss_2)\geq\psm(\mss_2)$. Suppose that $a\in\oa$ is such that $a\geq\psm(\mss_1)$, $a\geq\psm(\mss_2)$, and $a\psm(\pset)=a$. For any $\mss_3\in\alg$ such that $\mss_3\subseteq\mss_1$, we then have $a\geq\psm(\mss_3)$, so that $a\psm(\mss_3)\geq \psm(\mss_3)^2=\psm(\mss_3)$; and similarly for any $\mss_3\subseteq\mss_2$. Using this, we see that
	\begin{align*}
		a&=a\psm(\pset)\\
		&=a\big(\psm(\mss_1\setminus(\mss_1\cap\mss_2))+\psm(\mss_1\cap\mss_2)+\psm(\mss_2\setminus(\mss_1\cap\mss_2))\\
		&\phantom{=}\quad +\psm(\pset\setminus(\mss_1\cup\mss_2))\big)\\
		&\geq \psm(\mss_1\setminus(\mss_1\cap\mss_2))+\psm(\mss_1\cap\mss_2)+\psm(\mss_2\setminus(\mss_1\cap\mss_2))\\
		&=\psm(\mss_1\cup\mss_2).
	\end{align*}
	This concludes the proof of part~\ref{3_part:idempotents_form_boolean_algebra_b}.
	
	For part~\ref{3_part:idempotents_form_boolean_algebra_c}, we consider only the infimum in $\{a\in\oa: a\psm(X)=a\}$. The other case is handled similarly, and each of these implies the statements in the remaining two cases. Take $\mss_1,\mss_2\in\alg$.  It is clear that $\psm(\mss_1\cap\mss_2)\leq\psm(\mss_1)$ and $\psm(\mss_1\cap\mss_2)\leq\psm(\mss_2)$. Suppose that $a\in\oa$ is such that $a\leq\psm(\mss_1)$, $a\leq\psm(\mss_2)$, and $a\psm(\pset)=a$. Then $\psm(\pset)-a\geq\psm(\pset)-\psm(\mss_1)=\psm(\pset\setminus\mss_1)$; likewise, $\psm(\pset)-a\geq\psm(\pset\setminus\mss_2)$.  Part~\ref{3_part:idempotents_form_boolean_algebra_b} then implies that $\psm(\pset)-a\geq\psm(\pset\setminus(\mss_1\cap\mss_2))$, showing that $a\leq\psm(\mss_1\cap\mss_2)$. This concludes the proof of part~\ref{3_part:idempotents_form_boolean_algebra_c}.
	
	Now that the parts~\ref{3_part:idempotents_form_boolean_algebra_b} and~\ref{3_part:idempotents_form_boolean_algebra_c} have been established, it is clear that $\psm(\alg)$ is a lattice when supplied with the partial ordering inherited from $\oa$. It is routine to verify the  statements in part~\ref{3_part:idempotents_form_boolean_algebra_d}.
\end{proof}

\begin{remark}\label{3_rem:idempotents_form_boolean_algebra_associative_case}
If the multiplication in $\oa$ in \cref{3_res:idempotents_form_boolean_algebra} is associative, then clearly	
	$\leftidalgnonassociative{\psm(\pset)}{\oa}=\leftidalg{\psm(\pset)}{\oa}$,
	$\rightidalgnonassociative{\psm(\pset)}{\oa}=\rightidalg {\psm(\pset)}{\oa}$, and $\leftidalgnonassociative{\psm(\pset)}{\oa}\cap \rightidalgnonassociative{\psm(\pset)}{\oa}=\leftrightidalg{\psm(\pset)}{\oa}$.
\end{remark}

Although we shall not use it in our proofs, before we proceed, we still want to mention the following strengthening of what we believe to be a noteworthy result of Alekhno's in \cite{alekhno:2012}. Compared to \cref{3_res:idempotents_form_boolean_algebra}, its premises are of a different nature, but its conclusions are in the same vein. In \cref{3_res:idempotents_form_boolean_algebra}, the commutativity of the multiplication on the idempotents is supposed in  part~\ref{3_part:idempotents_form_boolean_algebra_1}, and its associativity follows from part~\ref{3_part:idempotents_form_boolean_algebra_2}. In \cref{3_res:alekhno_improved}, it is just the other way around. As in \cite{alekhno:2012}, the ordering is already necessary to establish the purely algebraic statement in part~\ref{3_part:alekhno_improved_1} of \cref{3_res:alekhno_improved} which, just as  \cref{3_res:idempotents_form_boolean_algebra}, applies to (not necessarily associative) partially ordered $\ZZ$-algebras, but now also when there is additive 2-torsion.

\begin{proposition}[Alekhno]\label{3_res:alekhno_improved}
	Let $\oa$ be an abelian group that is supplied with a translation invariant partial ordering such that $\posoa+\posoa\subseteq\posoa$, and with a bi-additive  map $(a_1,a_2)\mapsto a_1a_2$ from $\oa\times\oa$ into $\oa$ such that $\posoa\posoa\subseteq\posoa$. Take a positive idempotent $e$ in $\oa$, and let
	\[
	\mathrm{OI}(\oa;e)\coloneqq\{p\in\oa: 0\leq p\leq e,\,p^2=p,\,ep=pe=p\}
	\]
	denote the order idempotents relative to $e$. Suppose that $(p_1p_2)p_3=p_1(p_2p_3)$ whenever $p_1,p_2,p_3\in\mathrm{OI}(\oa;e)$.
	Set $\leftidalgnonassociative{e}{\oa}\coloneqq\{a\in\oa: ea=a\}$ and $\rightidalgnonassociative{e}{\oa}\coloneqq\{a\in\oa: ae=a\}$.
	Then ${\mathrm{OI}}(\oa;e)\subseteq {\leftidalgnonassociative{e}{\oa}}\cap{\rightidalgnonassociative{e}{\oa}}$, and:
	\begin{enumerate}
		\item\label{3_part:alekhno_improved_1}
		$p_1p_2=p_2p_1$ for all $p_1,p_2\in\mathrm{OI}(\oa;e)$;
		\item\label{3_part:alekhno_improved_2}
		for $p_1,p_2\in\mathrm{OI}(\oa;e)$, $p_1\vee p_2$ exists in $\leftidalgnonassociative{e}{\oa}$, $\rightidalgnonassociative{e}{\oa}$,  $\leftidalgnonassociative{e}{\oa}\cap\rightidalgnonassociative{e}{\oa}$, and $\mathrm{OI}(\oa;e)$. In all cases, it equals $p_1+p_2-p_1p_2\in \mathrm{OI}(\oa;e)$;
		\item\label{3_part:alekhno_improved_3}
		for $p_1,p_2\in\mathrm{OI}(\oa;e)$, $p_1\wedge p_2$ exists in  $\leftidalgnonassociative{e}{\oa}$, $\rightidalgnonassociative{e}{\oa}$,  $\leftidalgnonassociative{e}{\oa}\cap\rightidalgnonassociative{e}{\oa}$, and $\mathrm{OI}(\oa;e)$. In all cases, it equals $p_1p_2\in \mathrm{OI}(\oa;e)$.
	\end{enumerate}
	Supplied with the partial ordering inherited from $\oa$, $\mathrm{OI}(\oa;e)$ is a Boolean algebra with largest element $e$ and smallest element 0; the complement of $p\in{\mathrm {OI}}(\oa;e)$ is $e-p$.
	
	Suppose that every non-empty upward directed subset of ${\mathrm{OI}}(\oa;e)$ has a supremum in $\leftidalgnonassociative{e}{\oa}\cap\rightidalgnonassociative{e}{\oa}$. Take a non-empty \uppars{not necessarily directed} subset of ${\mathrm{OI}}(\oa;e)$. Then its supremum and infimum exist in $\leftidalgnonassociative{e}{\oa}\cap\rightidalgnonassociative{e}{\oa}$, and they are both elements of ${\mathrm{OI}(\oa;e)}$. Hence $\mathrm{OI}(\oa;e)$ is a complete Boolean algebra in this case.
\end{proposition}

\begin{remark}\label{3_rem:alekhno}\quad
	\begin{enumerate}
		\item\label{3_part:alekhno_1}
		In the algebra of all operators from a vector lattice to itself, the order idempotents relative to the identity operator are precisely the order projections; see \cite[Theorem~1.44]{aliprantis_burkinshaw_POSITIVE_OPERATORS_SPRINGER_REPRINT:2006}, for example. This motivates the terminology.
		\item\label{3_part:alekhno_2}
		Alekhno's actual statements in \cite[Lemma~2.1 and Corollary~2.2]{alekhno:2012} and their proofs are given in the context of an ordered Banach algebra $\oa$ with a positive identity element $e$. An inspection of his arguments shows that, with a few adaptations, they also suffice to establish \cref{3_res:alekhno_improved}.
	\end{enumerate}
\end{remark}

Returning to the main line, we have the following consequence of \cref{3_res:idempotents_form_boolean_algebra}.

\begin{proposition}\label{3_res:idempotents_for_subsets_of_finite_measure_give_spectral_measure}
	Let $\ms$ be a measurable space, let $\oa$ be a \smc\ partially ordered algebra, and let $\npm\colon \alg\to\posoaext$ be a measure. Then the following are equivalent:
	\begin{enumerate}
		\item\label{3_part:idempotents_for_subsets_of_finite_measure_give_spectral_measure_1}
		$\npm(\mss)^2=\npm(\mss)$ for all $\mss\in\alg$ with finite measure, and $\npm(\mss_1)\npm(\mss_2)=\npm(\mss_2)\npm(\mss_1)$ for all $\mss_1,\mss_2\in\alg$ with finite measure;
		\item\label{3_part:idempotents_for_subsets_of_finite_measure_give_spectral_measure_2}
		$\npm$ is a spectral measure.
		\end{enumerate}
\end{proposition}

\begin{proof}
	We prove that part~\ref{3_part:idempotents_for_subsets_of_finite_measure_give_spectral_measure_1} implies part~\ref{3_part:idempotents_for_subsets_of_finite_measure_give_spectral_measure_2}; the converse is trivial. Take $\mss_1,\mss_2\in\alg$ with finite measure. Set $\alg^\prime\coloneqq\{\mss\cap(\mss_1\cup\mss_2):\mss\in\alg\}$, and define $\npm^\prime:\alg^\prime\to\posoa$ by setting $\npm^\prime(\mss^\prime)\coloneqq\npm(\mss^\prime)$ for $\mss^\prime\in\alg^\prime$. Then \cref{3_res:idempotents_form_boolean_algebra} applies to the finite measure $\npm^\prime$ on $\alg^\prime$ and shows that, in particular, $\npm(\mss_1\cap\mss_2)=\npm(\mss_1)\npm(\mss_2)$.
\end{proof}

We now come to one of our key results.

\begin{theorem}\label{3_res:representing_measure_is_spectral}
	Let $\ts$ be a locally compact Hausdorff space, let $\oa$ be a monotone complete partially ordered algebra with a monotone continuous multiplication, and let $\psm\colon \borel\to\posoaext$ be a regular Borel measure. Let $\posmap\colon \contcts\to\oa$ be a positive algebra homomorphism.
	\begin{enumerate}
		\item\label{3_part:representing_measure_is_spectral_1}
		Suppose that $\psm(V)=\sup\{\posmap(f): f\prec V\}$ for every open subset $V$ of $\pset$ with finite measure. Then $\psm$ is a spectral measure.
		\item\label{3_part:representing_measure_is_spectral_2}
		Suppose that $\psm(K)=\inf\{\posmap(f): K\prec f\}$ for every compact subset $K$ of $\pset$, and that $\npm$ is inner regular at all Borel subsets with finite measure. Then $\psm$ is a spectral measure.
	\end{enumerate} 	
\end{theorem}

\begin{proof}
	We prove part~\ref{3_part:representing_measure_is_spectral_1}. The first step is to show that every $\npm(\mss)$ for $\mss\in\borel$ with finite measure is an idempotent. We start with a special case. Take an open subset $V$ of $\ts$ with finite measure.  If $f_1,f_2\prec V$, then also $f_1f_2\prec V$. If $f\prec V$, then also $\sqrt{f}\prec V$. Using the first statement in \cref{3_res:multiplication_in_two_variables_is_monotone_order_continuous}, we thus see that
	\begin{align*}
		\psm(V)&=\sup\{\posmap(f): f\prec V\}\\
		&=\sup\{\posmap(f_1f_2): f_1,f_2\prec V\}\\
		&=\sup\{\posmap(f_1)\posmap(f_2): f_1,f_2\prec V\}\\
		&=\sup\{\posmap(f_1): f_1\prec V\}\cdot\sup\{\posmap(f_2): f_2\prec V\}\\
		&=\psm(V)^2.
		\intertext{For the general case, take a Borel subset $\mss$ of $\ts$ with finite measure. Using the outer regularity of $\npm$, what he have just established, and the second statement in \cref{3_res:multiplication_in_two_variables_is_monotone_order_continuous}, we have}
		\psm(\mss)&=\inf\{\psm(V): V\text{ is open, }\ \mss\subseteq V\text{, and }\npm(V)<\infty\}\\
		&=\inf\{\psm(V)^2:V\text{ is open, }\ \mss\subseteq V\text{, and }\npm(V)<\infty \}\\
		&=\inf\{\psm(V): V\text{ is open, }\ \mss\subseteq V\text{, and }\npm(V)<\infty\}^2\\
		&=\psm(\mss)^2.
		\intertext{In the second step, we proceed to show that the $\psm(\mss)$ for $\mss\in\borel$ with finite measure all commute. Again, we start with a special case. Take open subsets $V_1, V_2$ of $\ts$ with finite measure. Using the first statement in \cref{3_res:multiplication_in_two_variables_is_monotone_order_continuous}, we have}
		\psm(V_1)\psm(V_2)&=\sup\{\posmap(f_1): f_1\prec V_1\}\cdot\sup\{\posmap(f_2): f_2\prec V_2\}\\
		&=\sup\{\posmap(f_1)\posmap(f_2): f_1\prec V_1\,,f_2\prec V_2\}\\	
		&=\sup\{\posmap(f_2)\posmap(f_1): f_2\prec V_2,\,f_1\prec V_1\}\\	
		&=\sup\{\posmap(f_2): f_2\prec V_2\}\cdot\sup\{\posmap(f_1): f_1\prec V_1\}\\
		&=\psm(V_1)\psm(V_2).
	\end{align*}
	Using this, the outer regularity of $\psm$ at all Borel subsets of $\ts$ and the decreasing analogue of the first statement in \cref{3_res:multiplication_in_two_variables_is_monotone_order_continuous} show  that $\psm(\mss_1)\psm(\mss_2)=\psm(\mss_1)\psm(\mss_2)$ for $\mss_1,\mss_2\in\borel$ with finite measure. We can now invoke \cref{3_res:idempotents_for_subsets_of_finite_measure_give_spectral_measure} to conclude that $\npm$ is a spectral measure.
	
	The proof of part~\ref{3_part:representing_measure_is_spectral_2} is similar. %One now exploits the fact that, since $\oa$ is supposed to be normal and $\npm$ is finite, $\psm$ is inner regular at all Borel subsets of $\ts$; see \cite[Proposition~3.6]{de_jeu_jiang:2022b}.
\end{proof}

%%%%%%%%%%%%%%%%%%%%%%%%%%%%%%%%%%%%%%%%%%%%%%%%%%%%%%%%%%%%%%%%%%%%%%%%%%%%%%%%%%%%%%%%%%%%%%%%%%%%%%%%%%%%%%%%%%%%%%%%%%%%%%%%%%%%

\section{Relations between measures, algebra homomorphisms, and vector lattice homomorphisms}\label{3_sec:relations_between_measures_algebra_homomorphisms_and_vector_lattice_homomorphisms}

%%%%%%%%%%%%%%%%%%%%%%%%%%%%%%%%%%%%%%%%%%%%%%%%%%%%%%%%%%%%%%%%%%%%%%%%%%%%%%%%%%%%%%%%%%%%%%%%%%%%%%%%%%%%%%%%%%%%%%%%%%%%%%%%%%%%

\noindent Let $\msm$ \ be a measure space, and consider the associated integral operator $\opintm\colon \integrablefun\to\os$.  If $\os$ is a vector lattice, for which measures $\npm$ is $\opintm$ a vector lattice homomorphism? If $\os$ is a partially ordered algebra, for which measures $\npm$ is $\opintm$ an algebra homomorphism? In this section, we shall consider these and other relations between measures, algebra homomorphisms, and vector lattice homomorphisms. The notion of moduli preserving operators from \cref{3_subsec:moduli_preserving_operators} is a convenient one to formulate some of the results in this section with.

The two questions just raised are easily answered; see \cref{3_res:integral_is_vector_lattice_homomorphism} and \cref{3_res:integral_is_algebra_homomorphism}. Somewhat surprisingly, when $\os$ is a partially ordered algebra, $\npm$ is finite, and $\opintm\colon \integrablefun\to\os$ is an algebra homomorphism, there is always a vector lattice homomorphism associated with $\opintm$. Indeed, as \cref{3_res:integral_preserves_moduli} shows, the image of $\integrablefun$ under $\opintm$ is then a vector lattice, and $\opintm\colon \integrablefun\to\opintm(\integrablefun)$ is a vector lattice algebra homomorphism; see also \cref{3_res:eight_properties}. The results are particularly nice when $\os$ is a vector lattice algebra to begin with; see \cref{3_res:integral_into_riesz_algebra_is_riesz_algebra_homomorphism}. Various equivalences and (conditional) implications between the properties of finite measures with values in vector lattice algebras and those of the associated integral operators are collected in \cref{3_res:eight_properties}.

\begin{proposition}\label{3_res:integral_is_vector_lattice_homomorphism}
	Let $\ms$ be a measurable space, let $\rs$ be a \sDc\ vector lattice, and let $\npm\colon \alg\to\pososext$ be a measure. Then the following are equivalent:
	\begin{enumerate}
		\item\label{3_part:integral_act_as_a_riesz_hom_1}
		$\opintm\colon \integrablefun\to\rs$ is a vector lattice homomorphism;
		\item \label{3_part:integral_act_as_a_riesz_hom_2}
		$\npm(\mss_1\cap \mss_2)=\npm(\mss_1)\wedge\npm(\mss_2)$ in $\os$ for all $\mss_1,\mss_2\in\alg$ with finite measure;
		\item \label{3_part:integral_act_as_a_riesz_hom_2_extra}
		$\npm(\mss_1\cup\mss_2)=\npm(\mss_1)\vee\npm(\mss_2)$ in $\os$ for all $\mss_1,\mss_2\in\alg$ with finite measure.
	\end{enumerate}
When $\npm$ is finite, these are also equivalent to:
	\begin{enumerate}[resume]
		\item\label{3_part:integral_act_as_a_riesz_hom_3}
		$\opintm\colon \boundedmeasfun\to\rs$ is a vector lattice homomorphism.
	\end{enumerate}
If $\opintm\colon \integrablefun\to\rs$ is a vector lattice homomorphism, then the kernel of $\opintm$ is $\aezerofun$, so that the induced map $\opintP\colon \ellonespectral\to\rs$ is an injective vector lattice homomorphism. The space $\ellonespectral$ is then \sDc, and $\opintP\colon \ellonespectral\to\rs$ is $\sigma$-order continuous. If, in addition, $\rs$ is monotone complete and has the \csp, then $\ellonespectral$ is \Dc, it has the \csp, and $\opintP\colon \ellonespectral\to\os$ is order continuous.
\end{proposition}

\begin{proof}
	It is clear that part~\ref{3_part:integral_act_as_a_riesz_hom_1} implies part~\ref{3_part:integral_act_as_a_riesz_hom_2}. We show that part~\ref{3_part:integral_act_as_a_riesz_hom_2} implies part~\ref{3_part:integral_act_as_a_riesz_hom_1}. Take integrable elementary functions $\varphi_1$ and $\varphi_2$. There exist mutually disjoint elements $\mss_1,\dotsc,\mss_n$ of $\alg$ with finite measure, $\alpha_1,\dotsc,\alpha_n\geq 0$, and $\beta_1,\dotsc,\beta_n\geq 0$ such that $\varphi_1=\sum_{i=1}^n\alpha_i\chi_{\mss_i}$ and $\varphi_2=\sum_{i=1}^n\beta_i\chi_{\mss_i}$. As a consequence of the assumption, $\npm(\mss_k)\wedge\npm(\mss_l)=0$ for all $k,l=1,\dotsc,n$ such that $k\neq l$. Using \cite[Theorem~1.7(4)]{aliprantis_burkinshaw_LOCALLY_SOLID_RIESZ_SPACES_WITH_APPLICATIONS_TO_ECONOMICS_SECOND_EDITION:2003}, this implies that $\opintm(\varphi_1)\wedge\opintm(\varphi_2)=\sum_{i=1}^n\alpha_i\npm(\mss_i)\wedge\beta_i\npm(\mss_i)=\sum_{i=1}^n(\alpha_i\wedge\beta_i)\npm(\mss_i)=\opintm(\varphi_1\wedge\varphi_2)$. It then follows from the definition of the order integral and the (sequential) order continuity of the lattice operations in a vector lattice that $\opintm(f_1\wedge f_2)=\opintm(f_1)\wedge\opintm(f_2)$ for $f_1,f_2\in\posintegrablefun$. This implies that $\opintm$ is a vector lattice homomorphism from $\integrablefun$ into $\rs$.
	
	The equivalence of the parts \ref{3_part:integral_act_as_a_riesz_hom_2} and \ref{3_part:integral_act_as_a_riesz_hom_2_extra} follows from the fact that $\npm(\mss_1)+\npm(\mss_2)=\npm(\mss_1\cup\mss_1)+\npm(\mss_1\cap\mss_2)$ for all $\mss_1,\mss_2\in\alg$ with finite measure.
	
	Suppose that $\npm$ is finite. Then it is clear that part~\ref{3_part:integral_act_as_a_riesz_hom_3} implies part~\ref{3_part:integral_act_as_a_riesz_hom_2}. Evidently, part~\ref{3_part:integral_act_as_a_riesz_hom_1} then implies part~\ref{3_part:integral_act_as_a_riesz_hom_3}.
	
	The final statements follow from \cite[Theorem~6.17]{de_jeu_jiang:2022a}.
\end{proof}

\begin{proposition}\label{3_res:integral_is_algebra_homomorphism}
	Let $\ms$ be a measurable space, let $\oa$ be a \smc\ partially ordered algebra with a $\sigma$-monotone continuous multiplication, and let $\npm\colon \alg\to\posoaext$ be a measure. Then the following are equivalent:
	\begin{enumerate}
		\item\label{3_part:integral_preserves_multiplication_1}
		$\integrablefun$ is a commutative algebra, and $\opintm\colon  \integrablefun\to\oa$ is an algebra homomorphism;
		\item\label{3_part:integral_preserves_multiplication_2}
		$\npm$ is a spectral measure.
	\end{enumerate}
	When $\npm$ is finite, these are also equivalent to each of:
	\begin{enumerate}[resume]
		\item\label{3_part:integral_preserves_multiplication_3}
		$\opintm\colon  \boundedmeasfun\to\oa$ is an algebra homomorphism;
	\item\label{3_part:integral_preserves_multiplication_4}  $\psm(\mss)^2=\psm(\mss)$ for $\mss\in\alg$ and $\psm(\mss_1)\psm(\mss_2)=\psm(\mss_2)\psm(\mss_1)$ for $\mss_1,\mss_2\in\alg$.
	\end{enumerate}
\end{proposition}

\begin{proof}
	It is obvious that part~\ref{3_part:integral_preserves_multiplication_1} implies part~\ref{3_part:integral_preserves_multiplication_2}. We show that part~\ref{3_part:integral_preserves_multiplication_2} implies part~\ref{3_part:integral_preserves_multiplication_1}. For this, it is sufficient to show that $fg\in\integrablefun$ whenever $f,g\in\posintegrablefun$, and that then $\opintm(fg)=\opintm(f)\opintm(g)$.
	These are both easily seen to hold when $f$ and $g$ are integrable elementary functions. For the general case, we choose integrable elementary functions $\varphi_1,\varphi_2,\dotsc$ and $\psi_1,\psi_2,\dotsc$ such that $\varphi_m\uparrow f$ and $\psi_n\uparrow g$ pointwise. Then $\opintm(\varphi_m\psi_n)=\opintm(\varphi_m)\opintm(\psi_n)$ for $m,n=1,2,\dotsc$. For fixed $m$, we let $n$ tend to infinity. The definition of the order integral and the $\sigma$-monotone continuity of the multiplication in $\oa$ then show that $\varphi_m g\in\posintegrablefun$, and that $\opintm(\varphi_m g)=\opintm(\varphi_m)\opintm(g)$ for $m=1,2,\dotsc$. Now we let $m$ tend to infinity, and use the monotone convergence theorem (see \cite[Theorem~6.9]{de_jeu_jiang:2022a}) and the $\sigma$-monotone continuity of the multiplication to conclude that $fg\in\integrablefun$, and that $\opintm(fg)=\opintm(f)\opintm(g)$.
	
	Suppose that $\npm$ is finite. Then it is clear that part~\ref{3_part:integral_preserves_multiplication_3} implies
	part~\ref{3_part:integral_preserves_multiplication_2}. Evidently,  part~\ref{3_part:integral_preserves_multiplication_1} then implies part~\ref{3_part:integral_preserves_multiplication_3}. \cref{3_res:idempotents_form_boolean_algebra} shows that the parts~\ref{3_part:integral_preserves_multiplication_2}\ and~\ref{3_part:integral_preserves_multiplication_4} are equivalent.
\end{proof}

\begin{remark}\label{3_rem:algebra_for_ricker_spectral_measures}
In the context of spectral measures which take their values in the bounded operators on a Banach space in the sense of \cite[Definition~III.2]{ricker_OPERATOR_ALGEBRAS_GENERATE_BY_COMMUTING_PROJECTIONS:1999}, it is also true that the integrable functions (see their `weak' definition in \cite[Definition~III.5]{ricker_OPERATOR_ALGEBRAS_GENERATE_BY_COMMUTING_PROJECTIONS:1999}) form an algebra, and that the pertinent integral operator is an algebra homomorphism; see \cite[Proposition~V.3]{ricker_OPERATOR_ALGEBRAS_GENERATE_BY_COMMUTING_PROJECTIONS:1999}.  Since it is established a little later that the integrable functions are precisely the essentially bounded ones (see \cite[Proposition~~V.4]{ricker_OPERATOR_ALGEBRAS_GENERATE_BY_COMMUTING_PROJECTIONS:1999}), the fact that they form an algebra becomes somewhat less surprising. \cref{3_ex:integrable_function_for_spectral_measure_need_not_be_bounded} shows, however,  that, for the spectral measures in our sense, integrable functions need not be essentially bounded. The fact that they still always form an algebra is, therefore, a longer lasting surprise than  the corresponding statement in \cite[Proposition~V.3]{ricker_OPERATOR_ALGEBRAS_GENERATE_BY_COMMUTING_PROJECTIONS:1999}.
\end{remark}

In the context of \cref{3_res:integral_is_algebra_homomorphism}, suppose that $\npm$ is finite, and that the four then equivalent statements in it hold. Then the parts~\ref{3_part:idempotents_form_boolean_algebra_1} and~\ref{3_part:idempotents_form_boolean_algebra_2}  of \cref{3_res:idempotents_form_boolean_algebra} apply.  Its parts~\ref{3_part:idempotents_form_boolean_algebra_b}  and~\ref{3_part:idempotents_form_boolean_algebra_c} then show that, when appropriately interpreted, $\opintm$ preserves the supremum and infimum of the characteristic functions of two measurable subsets of $\pset$. This is actually true for two arbitrary elements of $\integrablefun$, as is shown by part~\ref{3_part:integral_preserves_moduli_1_inserted} of the following result.

\begin{theorem}\label{3_res:integral_preserves_moduli}
	Let $\ms$ be a measurable space, let $\oa$ be a \smc\ partially ordered algebra with a $\sigma$-monotone continuous multiplication, and let $\psm\colon \alg\to\posoa$ be a finite spectral measure.
	Then:
	\begin{enumerate}
		\item\label{3_part:integral_preserves_moduli_1} $\opintP(f)\in\psm(\pset)\oa\psm(\pset)$ for $f\in\integrablefun$;
		\item\label{3_part:integral_preserves_moduli_1_inserted} the maps $\opintP\colon \integrablefun\to\leftidalg{\psm(\pset)}{\oa}$, $\opintP\colon \integrablefun\to\rightidalg{\psm(\pset)}{\oa}$, and $\opintP\colon \integrablefun\to\leftrightidalg{\psm(\pset)}{\oa}$ preserve moduli;
		\item \label{3_part:integral_preserves_moduli_2}
		$\integrablefunspectral$ is a commutative unital vector lattice algebra;
		\item \label{3_part:integral_preserves_moduli_3}
		when $\opintP(\integrablefunspectral)$ is supplied with the partial ordering inherited from $\oa$, it is a commutative unital vector lattice algebra with $\psm(\pset)$ as its positive multiplicative identity element. The map $\opintP\colon \integrablefunspectral\to\opintP(\integrablefunspectral)$ is then a surjective unital vector lattice algebra homomorphism.
		\item\label{3_part:integral_preserves_moduli_4}
		the kernel of $\opintP\colon \integrablefunspectral\to\oa$ is $\aezerofunspectral$, so that the induced map $\opintP\colon \ellonespectral\to\opintP(\ellonespectral)$ is an isomorphism of vector lattice algebras. The space $\ellonespectral$ is \sDc, and $\opintP\colon \ellonespectral\to\oa$ is $\sigma$-order continuous.  If, in addition, $\oa$ is monotone complete and has the \csp, then $\ellonespectral$ is \Dc, it has the \csp, and $\opintP\colon \ellonespectral\to\oa$ is order continuous.
	\end{enumerate}
\end{theorem}

\begin{proof}
Part~\ref{3_part:integral_preserves_moduli_1} follows from \cref{3_res:integral_is_algebra_homomorphism} and the fact that the constant 1 function is integrable.
	 For the moduli preserving properties of the two maps in part~\ref{3_part:integral_preserves_moduli_1_inserted}, we consider the case with codomain $\leftidalg{\psm(\pset)}{\oa}$. The other can be treated similarly, and clearly each of these implies the statement for their subset $\leftrightidalg{\psm(\pset)}{\oa}$.  Take $f,g\in\posintegrablefun$. According to \cref{3_res:lattice_preserving_on_positive_cone_is_sufficient}, the proof will be complete once we show that $\opintP(f)\vee\opintP(g)$ exists in $\leftidalg{\psm(\pset)}{\oa}$, and that it equals $\opintP(f\vee g)$. Certainly, $\opintP(f\vee g)$ is an upper bound of $\{\opintP(f),\opintP(g)\}$. Let $a\in\leftidalg{\psm(\pset)}{\oa}$ also be an upper bound of this set.
	
	As a preparation, suppose that $\varphi$ and $\psi$ are integrable elementary functions such that $a\geq\opintP(\varphi)$ and $a\geq\opintP(\psi)$. There exist mutually disjoint $\mss_1,\dotsc,\mss_n\in\alg$ such that $\pset=\bigcup_{i=1}^n\mss_i$, $\alpha_1,\dotsc,\alpha_n\geq 0$, and $\beta_1,\dotsc,\beta_n\geq 0$, such that $\varphi=\sum_{i=1}^n\alpha_i\chi_{\mss_i}$ and $\psi=\sum_{i=1}^n\beta_i\chi_{\mss_i}$. We have $a\geq\opintP(\varphi)=\sum_{i=1}^n \alpha_i\psm(\mss_i)\geq \alpha_j\psm(\mss_j)$ for $j=1,\dotsc,n$. Hence $\psm(\mss_j)a\geq\alpha_j\psm(\mss_j)^2=\alpha_j\psm(\mss_j)$ for $j=1,\dotsc,n$. A similar argument applies to $\psi$, and we conclude that $\psm(\mss_j)a\geq(\alpha_j\vee\beta_j)\psm(\mss_j)$ for $j=1,\dotsc,n$. Then
	\begin{align*}
		a&=\psm(\pset)a\\
		&=\sum_{i=1}^n\psm(\mss_i)a\\
		&\geq \sum_{i=1}^n(\alpha_i\vee\beta_i)\psm(\mss_i)\\
		&=\opintP(\varphi\vee\psi).
	\end{align*}
	
After this preparation, we choose integrable elementary functions $\varphi_1,\varphi_2,\dotsc$ and $\psi_1,\psi_2,\dotsc$ such that $\varphi_m\uparrow f$ and $\psi_n\uparrow g$ pointwise. Then $a\geq\opintP(f)\geq\opintP(\varphi_m)$ and $a\geq\opintP(g)\geq\opintP(\psi_n)$ for $m,n=1,2\dotsc$, so that $a\geq\opintP(\varphi_m\vee\psi_n)$ for $m,n=1,2,\dotsc$. On letting $n$ tend to infinity, we see from the definition of the order integral that $a\geq\opintP(f\vee \psi_n)$ for $n=1,2,\dotsc$. Now we let $n$ tend to infinity, and invoke the monotone convergence theorem (see \cite[Theorem~6.9]{de_jeu_jiang:2022a}) to conclude that $a\geq \opintP(f\vee g)$, as required.

The parts~\ref{3_part:integral_preserves_moduli_2} and~\ref{3_part:integral_preserves_moduli_3} follow from part~\ref{3_part:integral_preserves_moduli_1}, \cref{3_res:lattice_preserving_map_give_vector_lattice_homomorphism}, and \cref{3_res:integral_is_algebra_homomorphism}. Part~\ref{3_part:integral_preserves_moduli_4} follows from \cite[Theorem~6.17]{de_jeu_jiang:2022a}.
\end{proof}

When $\oa$ is a vector lattice algebra to begin with, \cref{3_res:integral_preserves_moduli} can sometimes be improved.

\begin{theorem}\label{3_res:integral_into_riesz_algebra_is_riesz_algebra_homomorphism}
		Let $\ms$ be a measurable space, and let $\oa$ be a \sDc\ vector lattice algebra with a positive identity element $e$ and $\sigma$-monotone continuous multiplication. Let $\psm\colon \alg\to\posoa$ be a finite spectral measure. Suppose that $\psm(\pset)\leq e$. Then:
		\begin{enumerate}
			\item\label{3_part:integral_into_riesz_algebra_is_riesz_algebra_homomorphism_1}
			$\psm(\pset)\oa\psm(\pset)$ is a vector lattice subalgebra of $\oa$ that is also a projection band in $\oa$;
			\item\label{3_part:integral_into_riesz_algebra_is_riesz_algebra_homomorphism_2}
			$\integrablefunspectral$ is  a unital vector lattice algebra,\! and  $\opintP\!\colon\! \integrablefunspectral\to\oa$ is a vector lattice algebra homomorphism;
			\item\label{3_part:integral_into_riesz_algebra_is_riesz_algebra_homomorphism_inserted}
			\begin{enumerate}
				\item the image of \!$\boundedmeasfun$ under $\opintm$ is contained in the order ideal of $\psm(\pset)\oa\psm(\pset)$ that is generated by $\npm(\pset)$;
				\item the image of $\integrablefun$ under $\opintm$ is contained in the $\sigma$-order ideal of $\psm(\pset)\oa\psm(\pset)$ that is generated by $\npm(\pset)$;
			\end{enumerate}
			\item\label{3_part:integral_into_riesz_algebra_is_riesz_algebra_homomorphism_3}
			the kernel of $\opintP\colon \integrablefunspectral\to\oa$ is $\aezerofunspectral$, so that the quotient $\ellonespectral$ is a commutative unital vector lattice algebra. The induced map $\opintP\colon \ellonespectral\to\oa$ is an injective vector lattice algebra isomorphism. The space $\ellonespectral$ is \sDc, and $\opintP\colon \ellonespectral\to\oa$ and $\opintP\colon \ellonespectral\to\psm(\pset)\oa\psm(\pset)$ are both $\sigma$-order continuous. If, in addition, $\oa$ is monotone complete and has the \csp, then $\ellonespectral$ is \Dc, it has the \csp, and $\opintP\colon \ellonespectral\to\oa$ and $\opintP\colon \ellonespectral\to\psm(\pset)\oa\psm(\pset)$ are both order continuous.
		\end{enumerate}
\end{theorem}

\begin{proof}
We prove part~\ref{3_part:integral_into_riesz_algebra_is_riesz_algebra_homomorphism_1}. It is clear that $\psm(\pset)\oa\psm(\pset)$ is a subalgebra of $\oa$. Define the idempotent map $P\colon \oa\to\oa$ by setting $Pa\coloneqq\psm(\pset)a\psm(\pset)$ for $a\in\oa$. Since $\psm(\pset)\leq e$, $P$ lies between the zero map and the identity map on $\oa$. By \cite[Theorem~1.44]{aliprantis_burkinshaw_POSITIVE_OPERATORS_SPRINGER_REPRINT:2006}, $P$ is an order projection on $\oa$. Hence its range $\psm(\pset)\oa\psm(\pset)$ is a projection band in $\oa$.

For part~\ref{3_part:integral_into_riesz_algebra_is_riesz_algebra_homomorphism_2}, we apply
\cref{3_res:integral_preserves_moduli}. It shows that $\integrablefun$ is a commutative unital vector lattice algebra, and also that $\opintP\colon \integrablefunspectral\to\psm(\pset)\oa\psm(\pset)$ preserves moduli. Since we know here that $\psm(\pset)\oa\psm(\pset)$ is a vector sublattice of $\oa$, $\opintP\colon \integrablefunspectral\to\oa$ is a vector lattice homomorphism.

The first statement of part~\ref{3_part:integral_into_riesz_algebra_is_riesz_algebra_homomorphism_inserted} follows from the triangle inequality for the order integral; see \cite[Lemma~6.7]{de_jeu_jiang:2022a}. The second is clear from the definition of that integral.

Part~\ref{3_part:integral_into_riesz_algebra_is_riesz_algebra_homomorphism_3} follows from \cite[Theorem~6.17]{de_jeu_jiang:2022a}, except the ($\sigma$)-order continuity of  $\opintP\colon \ellonespectral\to\psm(\pset)\oa\psm(\pset)$. Since $\psm(\pset)\oa\psm(\pset)$ is a projection band in $\oa$, this follows from the ($\sigma$)-order continuity of $\opintP\colon \ellonespectral\to\oa$.
\end{proof}

In the following overview result, we collect various equivalences and conditional implications that are valid for finite measures when $\oa$ is a vector lattice algebra. As \cref{3_res:integral_preserves_moduli} shows, vector lattice algebras can enter the picture as a codomain even when the initial codomain is only a partially ordered algebra. Regarding the condition in part~\ref{3_part:eight_properties_are_equivalent_b} of \cref{3_res:eight_properties}, we recall that a vector lattice algebra $\oa$ is called an \emph{$\!f\!$-algebra} when the left an right multiplications preserve disjointness. For this, it is necessary and sufficient that, for $x,y,z\in\posoa$, $(zx)\wedge y=(xz)\wedge y=0$ whenever $x\wedge y=0$. A vector lattice algebra with an identity element is an $\!f\!$-algebra if and only if its squares are positive. We refer to \cite[Corollary~1]{steinberg:1976} for this; additional equivalent characterisations of $\!f\!$-algebras among the vector lattice algebras with a positive identity element can be found in \cite[Theorem~2.3]{huijsmans:1990}.  An $\!f\!$-algebra is commutative; see \cite[Theorem~2.56]{aliprantis_burkinshaw_POSITIVE_OPERATORS_SPRINGER_REPRINT:2006}.
An $\!f\!$-algebra is called \emph{semiprime} when 0 is its only nilpotent element.
Every $\!f\!$-algebra with an identity element is semiprime; see \cite[Theorem~10.4]{de_pagter_THESIS:1981}. The orthomorphisms on a vector lattice form an $\!f\!$-algebra with an identity element (see \cite[Theorem~2.59]{aliprantis_burkinshaw_POSITIVE_OPERATORS_SPRINGER_REPRINT:2006}) which is then semiprime. When they are algebras, our spaces $\integrablefun$ and $\ellone$ are semiprime $\!f\!$-algebras. When $\oa$ is a commutative complex \Calgebra, its self-adjoint part is a semiprime $\!f\!$-algebra; this is clear from its realisation as a $\contots$-space. % This final observation will be important later on when considering representations of   commutative complex  $\mathrm{C}^\ast\!$-algebras on Hilbert spaces.

\begin{theorem}\label{3_res:eight_properties}
	Let $\ms$ be a measurable space, let $\oa$ be a \sDc\ vector lattice algebra with a $\sigma$-monotone continuous multiplication, and let $\psm\colon \alg\to\posoa$ be a finite measure.
	
	The following are equivalent:
	\begin{enumerate}
		\item\label{3_part:eight_properties_are_equivalent_5}
		$\opintP\colon \integrablefunspectral\to\oa $ is a vector lattice homomorphism;
		\item\label{3_part:eight_properties_are_equivalent_6}
		$\opintP\colon \boundedmeasfun\to\oa $ is a vector lattice homomorphism;
		\item\label{3_part:eight_properties_are_equivalent_7}
		$\psm(\mss_1\cap\mss_2)=\psm(\mss_1)\wedge\psm(\mss_2)$ in $\oa$ for $\mss_1,\mss_2\in\alg$;
		\item \label{3_part:eight_properties_are_equivalent_7_extra}
		$\npm(\mss_1\cup\mss_2)=\npm(\mss_1)\vee\npm(\mss_2)$ in $\os$ for $\mss_1,\mss_2\in\alg$.	
	\end{enumerate}

The following are equivalent:
\begin{enumerate}[resume]
	\item\label{3_part:eight_properties_are_equivalent_1}
	$\integrablefunspectral$ is a commutative algebra, and $\opintP\colon \integrablefunspectral\to\oa $ is an algebra homomorphism;
	\item\label{3_part:eight_properties_are_equivalent_2}
	$\opintP\colon \boundedmeasfun\to\oa $ is an algebra homomorphism;			
	\item\label{3_part:eight_properties_are_equivalent_3}
	$\npm$ is a spectral measure;
	\item\label{3_part:eight_properties_are_equivalent_4} $\psm(\mss)^2=\psm(\mss)$ for $\mss\in\alg$ and $\psm(\mss_1)\psm(\mss_2)=\psm(\mss_2)\psm(\mss_1)$ for $\mss_1,\mss_2\in\alg$.
\end{enumerate}
	
	Suppose that at least one of the following is satisfied:
\begin{enumerate_alpha}
	\item\label{3_part:eight_properties_are_equivalent_a}
	$\oa$ has a positive identity element $e$ and $\psm(\pset)\leq e$;
	\item\label{3_part:eight_properties_are_equivalent_b}
	$\oa$ is a semiprime $\!f\!$-algebra.
\end{enumerate_alpha}
Then each of the equivalent parts~\ref{3_part:eight_properties_are_equivalent_1}--\ref{3_part:eight_properties_are_equivalent_4} implies each of the equivalent parts~\ref{3_part:eight_properties_are_equivalent_5}--\ref{3_part:eight_properties_are_equivalent_7_extra}.
\end{theorem}

\begin{proof}
		The equivalence of the parts~\ref{3_part:eight_properties_are_equivalent_5}--\ref{3_part:eight_properties_are_equivalent_7_extra} follows from \cref{3_res:integral_is_vector_lattice_homomorphism}; that of the parts~\ref{3_part:eight_properties_are_equivalent_1}--\ref{3_part:eight_properties_are_equivalent_4} follows from \cref{3_res:integral_is_algebra_homomorphism}.
	
	Suppose that $\oa$ has a positive identity element $e$, that $\psm(\pset)\leq e$, and that part~\ref{3_part:eight_properties_are_equivalent_3} holds. Then \cref{3_res:integral_into_riesz_algebra_is_riesz_algebra_homomorphism} shows that  part~\ref{3_part:eight_properties_are_equivalent_5} holds.
	
	Suppose that $\oa$ is a semiprime $\!f\!$-algebra and that part~\ref{3_part:eight_properties_are_equivalent_2} holds. Since $\boundedmeasfun$ is a semiprime $\!f\!$-algebra, and since a positive algebra homomorphism be\-tween two semi\-prime $\!f\!$-algebras is automatically a vector lattice homomorphism (see \cite[p.~96]{de_pagter_THESIS:1981}), we see that part~\ref{3_part:eight_properties_are_equivalent_6} holds.
\end{proof}

%%%%%%%%%%%%%%%%%%%%%%%%%%%%%%%%%%%%%%%%%%%%%%%%%%%%%%%%%%%%%%%%%%%%%%%%%%%%%%%%%%%%%%%%%%%%%%%%%%%%%%%%%%%%%%%%%%%%%%%%%%%%%%%%%%%%

\section{Ups and downs}\label{3_sec:ups_and_downs}

%%%%%%%%%%%%%%%%%%%%%%%%%%%%%%%%%%%%%%%%%%%%%%%%%%%%%%%%%%%%%%%%%%%%%%%%%%%%%%%%%%%%%%%%%%%%%%%%%%%%%%%%%%%%%%%%%%%%%%%%%%%%%%%%%%%%

\noindent Suppose that $\os$ is a partially ordered vector space, that $\ts$ is a locally compact Hausdorff space, and that $\posmap\colon \contcts\to\os$ is a positive operator. The earlier paper \cite{de_jeu_jiang:2022b} contains a number of results to the extent that, under appropriate conditions, there is unique regular $\pososext$-valued Borel measure $\npm$ on $\ts$ such that
\begin{equation*}\label{3_eq:operator_and_measure}
\posmap(f)=\ointm{f}
\end{equation*}
for all $f\in\contcts$. Moreover, if $V$ is a non-empty open subset of $\ts$, then
\begin{equation*}\label{3_eq:measure_of_open_subset}
	\npm(V)=\bigvee\{\posmap(f) : f\prec V\}
\end{equation*}
in $\pososext$; and if $K$ is a compact subset of $\ts$, then
\begin{equation*}\label{3_eq:measure_of_compact_subset}
	\npm(K)=\bigwedge\{\posmap(f) : K\prec f\}
\end{equation*}
in $\pososext$.
The original operator $\posmap$ can be extended to $\opintm\colon\integrablefun\to\os$, and to $\opintm\colon \boundedmeasfun\to\os$ when $\npm$ is finite.  Can we then describe the images $\opintm(\integrablefun)$ and $\opintm(\boundedmeasfun)$ of these functional calculi more directly in terms of $\posmap(\contcts)$? It turns out that this can often be done. The underlying reason is that (as \cref{3_sec:relations_between_measures_algebra_homomorphisms_and_vector_lattice_homomorphisms} shows) $\opintm$ often preserves moduli; if $\os$ has the countable sup property, this fact is then sufficient to make such a description possible. The present section is devoted to this.

We shall actually work in a more general context, where a Borel measure and a positive operator are \emph{supposed} to be related in a certain way. This allows us to obtain our results without unnecessary restrictions on, in particular, $\os$. The Riesz representation theorems in \cite{de_jeu_jiang:2022b} for positive operators $\posmap\colon \contcts\to\os$ then guarantee that, under appropriate conditions on, in particular, $\os$, these hypotheses are indeed satisfied, enabling us to apply the results in the present section to \cref{3_sec:special_positive_representations} where\textemdash this is not needed in the present section\textemdash$\npm$ will even be a spectral measure.

The description of the images of $\opintm$ will be in terms of ups and downs. We start with some preparations.

Let $\os$ be a partially ordered vector space, and let $S$ be a non-empty subset of $\os$. We define
\begin{align*}
	S^\up&\coloneqq\left\{ x\in\os: \text{ there exists a net }\net{x} \text{ in }S\text{ such that } x_\lambda\uparrow x \text{ in }\os\right\}\\
	\intertext{and its sequential version}
	S^\ups&\coloneqq\left\{ x\in\os: \text{ there exists a sequence }\seq{x} \text{ in }S\text{ such that } x_n\uparrow x \text{ in }\os\right\},
\end{align*}
and define $S^\down$ and $S^\downs$ similarly.\footnote{In \cite{aliprantis_burkinshaw_POSITIVE_OPERATORS_SPRINGER_REPRINT:2006}, our $S^\up$, $S^\down$, $S^\ups$, and $S^\downs$ are denoted by $S^\uparrow$, $S^\downarrow$, $S^\upharpoonleft$, and $S^\downharpoonleft$, respectively; in \cite{de_pagter:1983}, they are $S^\uparrow$, $S^\downarrow$, $S^{\uparrow_\omega}$, and $S^{\downarrow_\omega}$, respectively. Our notation may be a little clearer in smaller font.} We shall use self-evident notations such as $S^{\ups\down}\coloneqq\left(S^\ups\right)^{\raisebox{-2.5pt}{$\scriptstyle\down$}}$, etc.

\begin{remark}\label{3_rem:ups_and_downs_and_supersets}
When $\ostwo$ is a linear subspace of $\os$ and $S$ is a non-empty subset of  $\ostwo$, the ups and downs of $S$ in $\ostwo$ and in $\os$ need not be the same. With \cref{3_res:integral_preserves_moduli} in mind, we note that, for a monotone complete partially ordered algebra $\oa$ with monotone continuous multiplication and an idempotent $p\in\posoa$, the ups and downs of a non-empty subset of $p\oa$ in $p \oa$ coincide with their counterparts in $\oa$. This follows from \cref{3_res:properties_inherited_by_algebras_associated_to_idempotent}. Similar statements hold for $\oa p$ and $p\oa p$.
\end{remark}

We collect a few basic facts in the next three results.

\begin{lemma}\label{3_res:properties_of_ups_and_downs}
	Let $S$ be a non-empty subset of a partially ordered vector space $\os$, and let $\lambda\geq 0$.
	\begin{enumerate}
		\item\label{3_part:properties_of_ups_and_downs_1a}
		If $S+S\subseteq S$, or if $\lambda S\subseteq S$, then $S^\up$, $S^\down$, $S^\ups$, and $S^\downs$ all have the same respective property.
		\item\label{3_part:properties_of_ups_and_downs_2}		
		If $S^\vee$ exists in $\os$ and $S^\vee=S$, then $S^{\up\up}=S^\up$ and $S^{\ups\ups}=S^\ups$.
		\item\label{3_part:properties_of_ups_and_downs_3}
		If $S^\wedge$ exists in $\os$ and $S^\wedge=S$, then  $S^{\down\down}=S^\down$ and $S^{\downs\downs}=S^\downs$.
		\item\label{3_part:properties_of_ups_and_downs_1b}
		Suppose that $\os$ is a vector lattice. If $S^\vee=S$, or if $S^\wedge=S$, then each of $S^\up$, $S^\down$, $S^\ups$, and $S^\downs$ has the same respective property.
	\end{enumerate}
\end{lemma}

\begin{proof}The parts~\ref{3_part:properties_of_ups_and_downs_1a} and~\ref{3_part:properties_of_ups_and_downs_1b} are routine to establish.

We prove part~\ref{3_part:properties_of_ups_and_downs_2}; the proof of part~\ref{3_part:properties_of_ups_and_downs_3} is similar.
	
	Take an $x\in S^{\up\up}$. There exists a net $\{x_\lambda\}_{\lambda\in\Lambda}$ in $S^\up$ such that $x_\lambda\uparrow x$. For each $\lambda\in\Lambda$, there exists a net $\{x_{i_\lambda}\}_{i_\lambda\in I_\lambda}$ in $S$ such that $x_{i_\lambda}\uparrow x_\lambda$. Then $x$ is the supremum of the subset $S_0\coloneqq\{x_{i_\lambda} : \lambda\in \Lambda,\,i_\lambda\in I_\lambda\}$ of $S$. This subset is not obviously (the image of) an increasing net in $S$. However, the subset $S_0^\vee$ of $S$ still  provides a net in $S$ that increases to $x$.
	
	Take an $x\in S^{\ups\ups}$. There exists a sequence $\{x_n\}_{n=1}^\infty$ in $S^\ups$ such that $x_n \uparrow x$. For each $n\geq 1$, there exists a sequence $\{x_n^m\}_{m=1}^\infty$ in $S$ such that $x_n^m\uparrow_m x_n$. Then $x$ is the supremum of the countable subset $S_0\coloneqq\{x_n^m : n,m\geq 1\}$ of $S$. Choose an enumeration $z_1,\,z_2,\,z_3,\,\dotsc$ of $S_0$. Then the sequence $z_1$, $\sup\{z_1, z_2\}$, $\sup\{z_1,z_2,z_3\}$, $\dotsc$ is a sequence in $S$ that increases to $x$.
\end{proof}

\begin{remark}\label{3_rem:mistake_in_positive_operators}
For a vector lattices $\os$, it is stated on \cite[p.~83]{aliprantis_burkinshaw_POSITIVE_OPERATORS_SPRINGER_REPRINT:2006} that it is clear that $S^{\up\up}=S^\up$ and that $S^{\down\down}=S^\down$, without including any condition on $S$. This appears to be a mistake.
\end{remark}

\begin{lemma}\label{3_res:intersection_of_ups_and_downs_is_linear_subspace}
	Let $L$ be a linear subspace of a partially ordered vector space $\os$. Then $L^{\ups}\cap L^\downs$ and $L^{\up\downs}\cap L^{\down\ups}$ are linear subspaces of $\os$. In general, intersections such as $L^{\downs\up\up\ups}\cap L^{\ups\down\down\downs}$ of a finite number of consecutive  ups or downs of $L$, an arbitrary number of which may  be sequential, and the `mirrored'  consecutive  downs or ups  of $L$, with the directions of the arrows reversed, is a linear subspace of $\os$.
\end{lemma}

\begin{proof}
As an example, we prove that $L^{\up\downs}\cap L^{\down\ups}$  is a linear subspace of $\os$. In view of a double application of part~\ref{3_part:properties_of_ups_and_downs_1a} of \cref{3_res:properties_of_ups_and_downs}, it is sufficient to show that $-x\in L^{\up\downs}\cap L^{\down\ups}$ whenever $x\in L^{\up\downs}\cap L^{\down\ups}$. The fact that $x\in L^{\up\downs}$ implies that $-x\in -(L^{\up\ups})=(-L^\up)^\downs=(-L)^{\down\ups}=L^{\down\ups}$. Likewise, it follows from the fact that $x\in L^{\down\ups}$ that $-x\in L^{\up\downs}$.
\end{proof}

Suppose  that $S$ is a non-empty subset of the partially ordered vector space $\os$. Then we let
\[
\Wed{S}\coloneqq\left\{\sum_{i=1}^n{\alpha_is_i: n=1,2,\dotsc,\ \alpha_i\geq 0 \text{ for }i=1,\dotsc,n}\right\}
\]
denote the wedge in $\os$ that is generated by $S$. The following is clear from \cref{3_res:properties_of_ups_and_downs}.

\begin{lemma}\label{3_res:inclusion_between_wedges}
	Let $\os$ be a partially ordered vector space, and let $W$ be a wedge in $\os$. Then each of $W^\up$, $W^\down$, $W^\ups$, and $W^\downs$ is a wedge in $\os$. Suppose that $S$ is a non-empty subset of  $\os$. Then $\Wed{S^\up}\subseteq\left[\Wed{S}\right]^\up$. Similar statements hold for $S^\down$, $S^\ups$, and $S^\downs$.
	
	Suppose that $\os$ is a vector lattice. If $W^\vee=W$ or $W^\wedge=W$, then each of $W^\up$, $W^\down$, $W^\ups$, and $W^\downs$ has the same respective property.
	\end{lemma}

We can now show that the images of canonical positive cones are contained in cones that are built from the image of $\pos{\contcts}$ by using ups and downs.

\begin{proposition}\label{3_res:image_of_posmap_is_contained_in_ups_and_downs}
	Let $\ts$ be a locally compact Hausdorff space, let $\os$ be a \mc\  partially ordered vector space, and let $\posmap: \contcts\to\os$ be a positive operator. Suppose that $\npm\colon \borel\to\pososext$ is a regular Borel measure such that
	
	\begin{equation*}
		\posmap(f)=\ointm{f}
	\end{equation*}
for $f\in\contcts$;
	\begin{equation}\label{3_eq:image_of_posmap_is_contained_in_ups_and_downs_open_subsets}
		\npm(V)=\psup\{\posmap(f) : f\prec V\}
	\end{equation}
	in $\posos$ for every open subset $V$ of $\ts$ with finite measure; and
	\begin{equation}\label{3_eq:image_of_posmap_is_contained_in_ups_and_downs_compact_subsets}		
		\npm(K)=\pinf\{\posmap(f) : K\prec f\}
	\end{equation}
	in $\posos$ for every compact subset $K$ of $\ts$. Then
	\begin{align}
		\label{3_eq:image_of_posmap_is_contained_in_ups_and_downs_1}
		&\opintm\left(\posintegrablefunts\right)\subseteq\left[ \posmap(\pos{\contcts})\right]^{\up\down\ups};\\
		\intertext{if $\npm$ is inner regular at all Borel subsets of $\ts$ with finite measure, then also}
		\label{3_eq:image_of_posmap_is_contained_in_ups_and_downs_2}
		&\opintm\left(\posintegrablefunts\right)\subseteq \left[\posmap(\pos{\contcts})\right]^{\down\up\ups}.
		\intertext{Suppose that $\npm$ is finite. Then}
		\label{3_eq:image_of_posmap_is_contained_in_ups_and_downs_finite_1}
		&\opintm\left(\posboundedmeasfunts\right)\subseteq\left[ \posmap(\pos{\contcts})\right]^{\up\down\downs};\\
		\intertext{if $\npm$ is inner regular at all Borel subsets of $\ts$ , then also}
		\label{3_eq:image_of_posmap_is_contained_in_ups_and_downs_finite_2}
		&\opintm\left(\posboundedmeasfunts\right)\subseteq \left[\posmap(\pos{\contcts})\right]^{\down\up\downs}.
	\end{align}
\end{proposition}

\begin{proof} We establish \cref{3_eq:image_of_posmap_is_contained_in_ups_and_downs_1}.	It is clear that
	\[
	\opintm(\integrableelemfunts)=\Wed{\{\npm(\mss) : \mss\in\borel,\, \npm(\mss)<\infty\}}.
	\]

	The outer regularity of $\npm$ and \cref{3_eq:image_of_posmap_is_contained_in_ups_and_downs_open_subsets} imply that
	\[
	\{\npm(\mss) : \mss\in\borel,\, \npm(\mss)<\infty\}\subseteq\{\posmap(f): f\prec\ts\}^{\up\down}.
	\]
	Using a twofold application of \cref{3_res:inclusion_between_wedges} in the second step, we therefore see that
	\begin{equation}\label{3_eq:image_of_posmap_is_contained_in_ups_and_downs_finite_proof}
		\begin{split}
			\opintm\left(\integrableelemfunts\right)&\subseteq\Wed{\{\posmap(f): f \prec\ts\}^{\up\down}}\\
			&\subseteq\big[\Wed{\{\posmap(f): f \prec\ts\}}\big]^{\up\down}\\
			&=\big[\posmap(\pos{\contcts})\big]^{\up\down}.
		\end{split}
	\end{equation}
	The definition of the order integral now shows that \cref{3_eq:image_of_posmap_is_contained_in_ups_and_downs_1} holds.
	
	The proof of \cref{3_eq:image_of_posmap_is_contained_in_ups_and_downs_2} is similar, but now combines \cref{3_eq:image_of_posmap_is_contained_in_ups_and_downs_compact_subsets} with the inner regularity of $\npm$ at all Borel sets.
	
	Suppose that $\npm$ is finite. Take an $f\in\posboundedmeasfunts$. As for general elements of $\integrablefunts$, we have that $\opintm(f)\in\big[\opintm(\integrableelemfunts)\big]^\ups$, but we claim that now also $\opintm(f)\in\big[\opintm(\integrableelemfunts)\big]^\downs$. To see this, take an $M\geq 0$ such that $f(x)\leq M$ for all $x\in\pset$. Then $M\onefunction-f\geq 0$, so there exists a sequence $\{s_n\}_{n=1}^\infty$ of elementary function such that $s_n\uparrow M\onefunction - f$. Hence $\opintm(s_n)\uparrow \opintm(M\onefunction-f)$, which implies that $\opintm(M\onefunction-s_n)\downarrow\opintm(f)$. Since  $M\onefunction-s_n\in\integrableelemfunts$ for all $n$, this establishes our claim. An appeal to \cref{3_eq:image_of_posmap_is_contained_in_ups_and_downs_finite_proof} then concludes the proof of \cref{3_eq:image_of_posmap_is_contained_in_ups_and_downs_finite_1}. Similarly, the proof of \cref{3_eq:image_of_posmap_is_contained_in_ups_and_downs_finite_2} uses \cref{3_eq:image_of_posmap_is_contained_in_ups_and_downs_compact_subsets}, the inner regularity of $\npm$ at all Borel sets, and our claim.
\end{proof}

To improve\textemdash under extra conditions\textemdash inclusions such as in \cref{3_res:image_of_posmap_is_contained_in_ups_and_downs} to equalities, we need the following preparatory result. It is based on the monotone convergence theorem.

\begin{proposition}\label{3_res:pulling_back_ups_and_downs}
	Let $\msm$ be a measure space. Suppose that $\os$ has the \csp, and that $\opintm\colon \integrablefun\to\os$ preserves moduli. Let $S$ be a non-empty subset of $\posellone$.
	\begin{enumerate}
		\item\label{3_part:pulling_back_ups_and_downs_1}
		Suppose that $S^\vee=S$. Then $\opintm(S^\ups)=[\opintm(S)]^\ups=[\opintm(S)]^\up$.
		\item\label{3_part:pulling_back_ups_and_downs_2}
		Suppose that $S^\wedge=S$. Then $\opintm(S^\downs)=[\opintm(S)]^\downs=[\opintm(S)]^\down$.
		\end{enumerate}
		In fact:
		\begin{enumerate}[resume]
			\item\label{3_part:pulling_back_ups_and_downs_3}
			if $S\sp\vee=S$, then the subset of $\os$ that is obtained by taking at least one consecutive  ups of $\opintm(S)$, an arbitrary number of which may be sequential, is always equal to $\opintm(S^\ups)$;
			\item\label{3_part:pulling_back_ups_and_downs_4}
			if $S\sp\wedge=S$, then the subset of $\os$ that is obtained by taking at least one consecutive  downs of $\opintm(S)$, an arbitrary number of which may be sequential, is always equal to $\opintm(S^\downs)$.
		\end{enumerate}
\end{proposition}

\begin{proof}
	We establish part~\ref{3_part:pulling_back_ups_and_downs_1}. Already without any further conditions on $\os$ or $\opintm$, it is a consequence of the monotone convergence theorem that $\opintm\colon\ellone\to\os$ is $\sigma$-order continuous; see  \cite[Theorem~6.17]{de_jeu_jiang:2022a}. Hence certainly $\opintm(S^\ups)\subseteq[\opintm(S)]^\ups\subseteq[\opintm(S)]^\up$. We shall now use the extra hypotheses to show that $[\opintm(S)]^\up\subseteq \opintm(S^\ups)$. Suppose that $e\in\os$ and that $\opintm(s_\lambda)\uparrow e$ for some net $\{s_\lambda\}_{\lambda\in\Lambda}$ in $S$. Since $\os$ has the \csp, there exists a sequence $\{f_n\}_{n=1}^\infty$ in $\posintegrablefun$ such that $\{\eclass{f_n]}: n=1,2,\dotsc\}\subseteq\{s_\lambda:\lambda\in\Lambda\}\subseteq S$ and $e=\sup_{n\geq 1}\opintm(f_n)$. For $n=1,2\dotsc$, set $g_n\coloneqq f_1\vee\dotsb\vee f_n$. Since $S^\vee=S$, we have  $\eclass{g_n}=\eclass{f_1}\vee\dotsb\vee\eclass{f_n}\in S$. Furthermore, since $\opintm\colon \integrablefun\to\os$ preserves moduli, we have (see \cref{3_rem:binary_lattice_preserving_operators}) that $\opintm(g_n)=\opintm(f_1)\vee\dotsb\vee\opintm(f_n)\uparrow e$. We define the measurable function $g\colon \pset\to\posRext$ by setting $g(x)\coloneqq\sup_{n\geq 1}g_n(x)\in\posRext$ for  $x\in\pset$. Then $g_n(x)\uparrow g(x)$ in $\posRext$ for every $x\in\pset$. According to the monotone convergence theorem (see \cite[Theorem~6.9]{de_jeu_jiang:2022a}), we have that $\opintm(g_n)\uparrow\opintm(g)$ in $\pososext$. Hence $\opintm(g)=e$. Since this is finite, it follows from \cite[Lemma~6.4]{de_jeu_jiang:2022a} that $g$ is almost everywhere finite-valued. When necessary, we can, therefore, redefine $g$ and the $g_n$ for $n\geq 1$ to be zero on a subset of measure zero, and arrange that $g$ is finite-valued and that $g_n(x)\uparrow g(x)$ in $\RR$ for all $x\in\pset$. Then $g_n\uparrow g$ in $\integrablefun$. Because the quotient map from $\integrablefun$ to $\ellone$ is $\sigma$-order continuous (see \cite[Theorem~6.17]{de_jeu_jiang:2022a}, this implies that $\eclass{g_n}\uparrow \eclass{g}$ in $\ellone$. Since $\opintm(\eclass{g})=e$, we can now conclude that $e\in\opintm(S^\ups)$, as desired.
	
	The proof of part~\ref{3_part:pulling_back_ups_and_downs_2} is similar, using the monotone  convergence theorem for decreasing sequences; see \cite[Corollary~6.10]{de_jeu_jiang:2022a}. It is even slightly easier because the then occurring pointwise limit function is already automatically finite-valued.
	
	We prove part~\ref{3_part:pulling_back_ups_and_downs_3}. Suppose that $S\sp\vee=S$. Take $n\geq 1$ and take $n$ consecutive ups of $\opintm(S)$, an arbitrary number of which may be sequential, and let $\Sigma$ denote the resulting subset of $\os$. For $n=1$, part~\ref{3_part:pulling_back_ups_and_downs_3} coincides with part~\ref{3_part:pulling_back_ups_and_downs_1}. For $n\geq 2$, we let $\widetilde\Sigma$ denote the subset of $\os$ that is obtained by $n$ consecutive net ups $\up$  of $\opintm(S)$. We note that $\opintm(S\sp\ups)=[\opintm(S)]\sp\ups\subseteq\Sigma\subseteq\widetilde\Sigma$. Since $S\sp\vee=S$ and $\opintm$ preserves moduli, we see that $[\opintm(S)]\sp\vee$ exists in $\os$ and that $[\opintm(S)]\sp\vee=\opintm(S)$; see \cref{3_rem:binary_lattice_preserving_operators}.  An $(n-1)$-fold application of part~\ref{3_part:properties_of_ups_and_downs_2} of \cref{3_res:properties_of_ups_and_downs} shows that $\widetilde\Sigma=[\opintm(S)]\sp\up$.  By part~\ref{3_part:pulling_back_ups_and_downs_1}, $[\opintm(S)]\sp\up$ equals $\opintm(S\sp\ups)$. The proof of  part~\ref{3_part:pulling_back_ups_and_downs_3} is now complete.
	
	Similarly, part~\ref{3_part:pulling_back_ups_and_downs_4} follows from part~\ref{3_part:pulling_back_ups_and_downs_2} and part~\ref{3_part:properties_of_ups_and_downs_3} of \cref{3_res:properties_of_ups_and_downs}.
\end{proof}

\begin{corollary}\label{3_res:pulling_back_ups_and_downs_for_vector_sublattices}
	Let $\msm$ be a measure space. Suppose that $\os$ has the \csp, and that $\opintm\colon \integrablefun\to\os$ preserves moduli. Let $S$ be a vector sublattice of $\ellone$. Then:
	\begin{enumerate}
		\item\label{3_part:pulling_back_ups_and_downs_for_vector_sublattices_1}
		$\opintm(S^\ups)=[\opintm(S)]^\ups=[\opintm(S)]^\up$;
		\item\label{3_part:pulling_back_ups_and_downs_for_vector_sublattices_2}
		$\opintm(S^\downs)=[\opintm(S)]^\downs=[\opintm(S)]^\down$.
		\end{enumerate}
		In fact:
		\begin{enumerate}[resume]
		\item\label{3_part:pulling_back_ups_and_downs_for_vector_sublattices_3}
		the subset of $\os$ that is obtained by taking at least one consecutive  ups of $\opintm(S)$, an arbitrary number of which  may be sequential, is always equal to $\opintm(S^\ups)$;		\item\label{3_part:pulling_back_ups_and_downs_for_vector_sublattices_4}
		the subset of $\os$ that is obtained by taking at least one consecutive  downs of $\opintm(S)$, an arbitrary number of which  may be sequential, is always equal to $\opintm(S^\downs)$.
	\end{enumerate}	
\end{corollary}

\begin{proof}
	We prove part~\ref{3_part:pulling_back_ups_and_downs_for_vector_sublattices_1}; the proof of part~\ref{3_part:pulling_back_ups_and_downs_for_vector_sublattices_2} is similar.
	It follows from the monotone convergence theorem that 	$\opintm(S^\ups)\subseteq[\opintm(S)]^\ups$, and trivially 	$[\opintm(S)]^\ups\subseteq[\opintm(S)]^\up$. We show that $[\opintm(S)]^\up\subseteq\opintm(S^\ups)$.  Suppose that $e\in\os$ and that $\opintm(s_\lambda)\uparrow e$ for some net $\{s_\lambda\}_{\lambda\in\Lambda}$ in $S$. Since $\os$ has the \csp, there exists a sequence $\{f_n\}_{n=1}^\infty$ in $\posintegrablefun$ such that $\{\eclass{f_n]}: n=1,2,\dotsc\}\subseteq\{s_\lambda:\lambda\in\Lambda\}\subseteq S$ and $e=\sup_{n\geq 1}\opintm(f_n)$. For $n=1,2\dotsc$, set $g_n\coloneqq f_1\vee\dotsb\vee f_n-f_1$. Then $\eclass{g_n]}\in\pos{S}$, and $\opintm(\eclass{g_n})\uparrow e-\opintm(\eclass{f_1})$. Part~\ref{3_part:pulling_back_ups_and_downs_1} of \cref{3_res:pulling_back_ups_and_downs} shows that $e-\opintm(\eclass{f_1})\in\opintm((\pos{S})^\ups)\subseteq\opintm(S^\ups)$. This implies that $e\in\opintm(S^\ups)$.
	
	The parts~\ref{3_part:pulling_back_ups_and_downs_for_vector_sublattices_3} and~\ref{3_part:pulling_back_ups_and_downs_for_vector_sublattices_4} follow from the  parts~\ref{3_part:pulling_back_ups_and_downs_for_vector_sublattices_1} resp.~\ref{3_part:pulling_back_ups_and_downs_for_vector_sublattices_2} as in the proof of  \cref{3_res:pulling_back_ups_and_downs}.
\end{proof}

We can now establish the following result, where inclusions as in \cref{3_res:image_of_posmap_is_contained_in_ups_and_downs} are replaced with equalities. As \cref{3_subsec:the_countable_sup_property} indicates, the condition in it that $\os$ have the \csp\ is often met. Regarding its final part we recall that, if $\os$ is a \smc\ normed partially ordered algebra with a monotone norm, and $\npm$ is a finite spectral measure, then \cref{3_res:integrable_function_for_spectral_measure_is_bounded} shows that it is automatic that $\ellonets=\boundedmeasfunaets$; here we have used our convention as in \cref{3_rem:embedding_of_ellone}.

%For typographical reasons, when $S\subseteq\integrablefunts$, we shall write $\llbracket S\rrbracket$ for the image of $S$ in $\ellonets$ under the quotient map, rather than $[S]$ as we have done for individual elements.  Notations such as $\llbracket S\rrbracket\sp\ups$ then refer to the sequential up of $\llbracket S\rrbracket$ in $\ellonets$, etc.

\begin{theorem}[Ups and downs for positive cones]\label{3_res:ups_and_downs}
	
		Let $\ts$ be a locally compact Hausdorff space, let $\os$ be a \mc\  partially ordered vector space, and let $\posmap: \contcts\to\os$ be a positive operator. Suppose that $\npm\colon \borel\to\pososext$ is a regular Borel measure such that
	
	\begin{equation*}
		\posmap(f)=\ointm{f}
	\end{equation*}
	for $f\in\contcts$;
	\begin{equation*}%\label{3_eq:ups_and_downs_open_subsets}
		\npm(V)=\psup\{\posmap(f) : f\prec V\}
	\end{equation*}
	in $\posos$ for every open subset $V$ of $\ts$ with finite measure; and
	\begin{equation*}%\label{3_eq:ups_and_downs_compact_subsets}	
		\npm(K)=\pinf\{\posmap(f) : K\prec f\}
	\end{equation*}
	in $\posos$ for every compact subset $K$ of $\ts$.
	Suppose, furthermore, that $\os$ has the \csp, and that $\opintm\colon \integrablefun\to\os$ preserves moduli. Then:
	\begin{enumerate}
		\item\label{3_part:ups_and_downs_1}
		\begin{align}
			\label{3_eq:ups_and_downs_3}
			\opintm(\posintegrablefunts)&=\opintm(\posintegrablefunts)^\up=\opintm(\posintegrablefunts)^\down
			\intertext{and}
			\label{3_eq:ups_and_downs_1}
			\opintm(\posintegrablefunts)&=\opintm(\eclass{\pos{\contcts}}^{\ups\downs\ups})=\big[\posmap(\pos{\contcts})\big]^{\ups\downs\ups};
		\end{align}
		\item\label{3_part:ups_and_downs_2}
		If $\npm$ is inner regular at all Borel subsets of $\ts$ with finite measure, then
		\begin{equation}
		\label{3_eq:ups_and_downs_2}
		\opintm(\posintegrablefunts)=\opintm(\eclass{\pos{\contcts}}^{\downs\ups})=\big[\posmap(\pos{\contcts})\big]^{\downs\ups};
		\end{equation}
		\item\label{3_part:ups_and_downs_3}
		If $\npm$ is finite and $\ellonets=\boundedmeasfunaets$, then
		\end{enumerate}
		\begin{equation}\label{3_eq:ups_and_downs_finite}
		\opintm(\posintegrablefunts)=\opintm(\eclass{\pos{\contcts}}^{\ups\downs})=\big[\posmap(\pos{\contcts})\big]^{\ups\downs}.
		\end{equation}
\end{theorem}

\begin{proof}
	The equalities in \cref{3_eq:ups_and_downs_3} are immediate from \cref{3_res:pulling_back_ups_and_downs}.
	
	We establish \cref{3_eq:ups_and_downs_1}. We know from \cref{3_eq:image_of_posmap_is_contained_in_ups_and_downs_1} that
	\[
	\opintm(\posintegrablefunts)\subseteq\big[\posmap(\pos{\contcts})\big]^{\up\down\ups}=\big[\opintm(\eclass{\pos{\contcts}})\big]^{\up\down\ups}.
	\]
	Since $\eclass{\pos{\contcts}}$ is closed under the taking of finite suprema (and of finite infima), \cref{3_res:pulling_back_ups_and_downs} shows that
	\[
	\big[\opintm(\eclass{\pos{\contcts}})\big]^{\up\down\ups}=\big[\opintm(\eclass{\pos{\contcts}}\sp\ups)\big]^{\down\ups}.
	\]
	According to \cref{3_res:properties_of_ups_and_downs}, $\eclass{\pos{\contcts}}^\ups$ is still closed under the taking of finite infima (and of finite suprema). Another appeal to \cref{3_res:pulling_back_ups_and_downs}, followed by one further repetition of the argument, therefore yields that
	\[
	\big[\opintm(\eclass{\pos{\contcts}}^\ups)\big]^{\down\ups}=\opintm(\eclass{\pos{\contcts}}^{\ups\downs\ups}).
	\]
	Since, trivially,
	\[
	\opintm(\eclass{\pos{\contcts}}^{\ups\downs\ups})\subseteq\opintm(\posellonets)=\opintm(\posintegrablefunts),
	\]
	we can now conclude that
	\[
	\opintm(\posintegrablefunts)=\opintm(\eclass{\pos{\contcts}}^{\ups\downs\ups}).
	\]
	A threefold application of \cref{3_res:properties_of_ups_and_downs} and \cref{3_res:pulling_back_ups_and_downs} shows that
	\[
	\opintm(\eclass{\pos{\contcts}}^{\ups\downs\ups})=\big[\opintm(\eclass{\pos{\contcts}})\big]^{\ups\downs\ups}=\big[\posmap(\pos{\contcts})\big]^{\ups\downs\ups}.
	\]
	This completes the proof of \cref{3_eq:ups_and_downs_1}.
	
	When $\npm$ is inner regular at all Borel subsets of finite measure, we use \cref{3_eq:image_of_posmap_is_contained_in_ups_and_downs_2} as a starting point, and show similarly that
	\[
	\opintm(\posintegrablefunts)=\opintm(\eclass{\pos{\contcts}}^{\downs\ups\ups}).
	\]
	Since $\eclass{\contcts}^\downs$ is closed under the taking of finite suprema by \cref{3_res:properties_of_ups_and_downs}, the same \cref{3_res:properties_of_ups_and_downs} shows that $\eclass{\pos{\contcts}}^{\downs\ups\ups}=\eclass{\pos{\contcts}}^{\downs\ups}$. A twofold application of \cref{3_res:properties_of_ups_and_downs} and \cref{3_res:pulling_back_ups_and_downs}  then completes the proof of \cref{3_eq:ups_and_downs_2}.
	
	When $\npm$ is finite and $\ellonets=\boundedmeasfunaets$, \cref{3_eq:image_of_posmap_is_contained_in_ups_and_downs_finite_1} shows that
	\[
	\opintm(\posintegrablefunts)\!=\!\tnegskip\opintm(\posboundedmeasfunts)\tnegskip\subseteq\tnegskip\big[\!\posmap(\pos{\contcts})\!\big]^{\up\down\downs}\!=\big[\tnegskip\opintm(\eclass{\pos{\contcts}})\big]^{\up\down\downs}.
	\]
	Arguing as before, \cref{3_eq:ups_and_downs_finite} follows from this.	
\end{proof}

%\begin{remark}\label{3_rem:no_extra_information}
%	In \cref{3_res:ups_and_downs}, if $\npm$ is finite,  $\ellonets=\boundedmeasfunaets$, and $\npm$ is inner regular at all Borel subsets of $\ts$ with finite measure, then one can also use \cref{3_eq:image_of_posmap_is_contained_in_ups_and_downs_finite_2} instead of \cref{3_eq:image_of_posmap_is_contained_in_ups_and_downs_finite_1} in the final part of the above proof.  The outcome is then the validity of \cref{3_eq:ups_and_downs_inserted}, which we already know to hold without the condition that $\npm$ be finite.
%\end{remark}

\begin{remark}\label{3_rem:components_of_a_positive_operators}
	It is known from \cite[Theorem~3.10]{de_pagter:1983} that the set of components of a positive operator between two \Dc\ vector lattices $\os$ and $\ostwo$, where $\ocdual{\ostwo}$ separates the points of $\ostwo$, can be obtained as the $\ups\down\up$ and as the $\downs\up\down$ of the set of its simple components.\footnote{According to \cite[Theorem~2.6]{aliprantis_burkinshaw_POSITIVE_OPERATORS_SPRINGER_REPRINT:2006}, $\os$ need merely have the principal projection property.} When the operator is order continuous, it suffices to take the $\ups\downs$ or the $\downs\ups$ of the set of its simple components; see \cite[Theorem~3.11]{de_pagter:1983}. Given the similarities, it is an intriguing question whether (special cases of) these results in \cite{de_pagter:1983} may be related to (special cases of) those in \cref{3_res:ups_and_downs}. 	
\end{remark}	

\cref{3_res:image_of_posmap_is_contained_in_ups_and_downs} and \cref{3_res:ups_and_downs} are concerned with the images of the positive cones of the integrable and bounded measurable functions. It is also possible to establish versions for the image of the full vector lattice of bounded measurable functions. These are taken together in the following result. Its final statement will be particularly relevant in the sequel.

\begin{theorem}[Ups and downs for vector lattices]\label{3_res:full_ups_and_downs}

	Let $\ts$ be a locally compact Hausdorff space, let $\os$ be a \mc\  partially ordered vector space, and let $\posmap: \contcts\to\os$ be a positive operator. Suppose that $\npm\colon \borel\to\posos$ is a finite regular Borel measure that is inner regular at all Borel subsets of $\ts$, and such that

\begin{equation*}
	\posmap(f)=\ointm{f}
\end{equation*}
for $f\in\contcts$;
\begin{equation*}%\label{3_eq:ups_and_downs_open_subsets}
	\npm(V)=\psup\{\posmap(f) : f\prec V\}
\end{equation*}
in $\posos$ for every open subset $V$ of $\ts$;
\begin{equation*}%\label{3_eq:ups_and_downs_compact_subsets}	
	\npm(K)=\pinf\{\posmap(f) : K\prec f\}
\end{equation*}
in $\posos$ for every compact subset $K$ of $\ts$. Then
	\begin{align}
		\label{3_eq:full_image_of_posmap_is_contained_in_ups_and_downs_1}
		&\opintm\left(\boundedmeasfunts\right)\subseteq\left[ \posmap(\contcts)\right]^{\up\down\ups}\cap\left[\posmap(\contcts)\right]^{\down\up\downs}\\
		\intertext{and}
		\label{3_eq:full_image_of_posmap_is_contained_in_ups_and_downs_3}
		&\opintm\left(\boundedmeasfunts\right)\subseteq \left[\posmap(\contcts)\right]^{\up\down\downs}\cap \left[ \posmap(\contcts)\right]^{\down\up\ups}.
		\intertext{If $\os$ has the \csp\ and $\opintm\colon\integrablefunts\to\os$ preserves moduli, then}
		\label{3_eq:full_image_of_posmap_is_contained_in_ups_and_downs_6}
		&\opintm\left(\boundedmeasfunts\right)\subseteq  \left[ \posmap(\contcts)\right]^{\ups\downs}\cap\left[ \posmap(\contcts)\right]^{\downs\ups}.
	\end{align}

	If $\os$ has the \csp, $\opintm \colon\integrablefunts\to\os$ preserves moduli, and $\ellonets=\boundedmeasfunaets$, then		
	
	\begin{equation}\label{3_eq:full_image_of_posmap_is_contained_in_ups_and_downs_7}
		\opintm\left(\boundedmeasfunts\right)=\left[ \posmap(\contcts)\right]^{\ups\downs}=\left[ \posmap(\contcts)\right]^{\downs\ups}.
	\end{equation}
\end{theorem}

\begin{proof}
	It follows from \cref{3_eq:image_of_posmap_is_contained_in_ups_and_downs_1,3_eq:image_of_posmap_is_contained_in_ups_and_downs_finite_2} that
\begin{align*}
\opintm(\posboundedmeasfunts)&\subseteq\left[ \posmap(\pos{\contcts})\right]^{\up\down\ups}\cap \left[\posmap(\pos{\contcts})\right]^{\down\up\ups}\\
&\subseteq \left[ \posmap(\contcts)\right]^{\up\down\ups}\cap \left[\posmap(\contcts)\right]^{\down\up\ups}.
\end{align*}
Since $\left[ \posmap(\contcts)\right]^{\up\down\ups}\cap \left[\posmap(\contcts)\right]^{\down\up\ups}$ is a linear subspace of $\os$ by \cref{3_res:intersection_of_ups_and_downs_is_linear_subspace}, the validity of
\cref{3_eq:full_image_of_posmap_is_contained_in_ups_and_downs_1} is now clear. Similarly, \cref{3_eq:image_of_posmap_is_contained_in_ups_and_downs_2,3_eq:image_of_posmap_is_contained_in_ups_and_downs_finite_1} can be used to establish \cref{3_eq:full_image_of_posmap_is_contained_in_ups_and_downs_3}.

Suppose that $\os$ has the \csp, and that $\opintm\tnegskip\colon  \integrablefunts\to\os$ preserves moduli.
Using \cref{3_res:pulling_back_ups_and_downs_for_vector_sublattices},
\cref{3_eq:full_image_of_posmap_is_contained_in_ups_and_downs_3} implies that
\[\opintm\left(\boundedmeasfunts\right)\subseteq\left[ \posmap(\contcts)\right]^{\up\down\downs}=\left[ \opintm(\eclass{\contcts})\right]^{\up\down\downs}=\left[\opintm(\eclass{\contcts^\ups})\right]^{\down\downs}.
 \]
As $\os$ has the \csp, we have
\[
\left[\opintm(\eclass{\contcts^\ups})\right]^{\down\downs}=\left[\opintm(\eclass{\contcts^\ups})\right]^{\downs\downs}.
\]

Since $\eclass{\contcts}^\ups$ is still closed under the taking of finite infima by part~\ref{3_part:properties_of_ups_and_downs_1b} of \cref{3_res:properties_of_ups_and_downs}, the fact that $\opintm$ preserves moduli implies that $[\opintm(\eclass{\contcts}^\ups)]\sp{\wedge}$ exists in $\os$ and that $[\opintm(\eclass{\contcts}^\ups)]\sp{\wedge}=\opintm(\eclass{\contcts}^\ups)$; see \cref{3_rem:binary_lattice_preserving_operators}. Part~\ref{3_part:properties_of_ups_and_downs_3} of \cref{3_res:properties_of_ups_and_downs} therefore yields that

\[
\left[\opintm(\eclass{\contcts^\ups})\right]^{\downs\downs}=\left[\opintm(\eclass{\contcts^\ups})\right]^{\downs},
\]
Finally, \cref{3_res:pulling_back_ups_and_downs_for_vector_sublattices} shows that
\[
\left[\opintm(\eclass{\contcts^\ups})\right]^{\downs}=\left[\opintm(\eclass{\contcts})\right]^{\ups\downs}=\left[\posmap(\contcts)\right]^{\ups\downs}.
\]
One similarly shows that
\[\opintm\left(\boundedmeasfunts\right)\subseteq\left[ \posmap(\contcts)\right]^{\down\up\ups}=\left[\posmap(\contcts)\right]^{\downs\ups},
\]
which concludes the proof of \cref{3_eq:full_image_of_posmap_is_contained_in_ups_and_downs_6}.

Suppose that $\os$ has the \csp, that $I_{\npm}\colon\integrablefunts\to\os$ preserves moduli, and that $\ellonets=\boundedmeasfunaets$.  We establish the first equality in \cref{3_eq:full_image_of_posmap_is_contained_in_ups_and_downs_7}. Starting from \cref{3_eq:full_image_of_posmap_is_contained_in_ups_and_downs_6}, and using \cref{3_res:pulling_back_ups_and_downs_for_vector_sublattices} in the final step, we have that
\[
\opintm(\boundedmeasfunts)\subseteq\left[\posmap(\contcts)\right]^{\ups\downs}=\left[\opintm(\eclass{\contcts})\right]^{\ups\downs}=\left[\opintm(\eclass{\contcts}\sp\ups)\right]^{\downs}.
\]
Using \cref{3_res:pulling_back_ups_and_downs_for_vector_sublattices} in the second step, we see that
\begin{align*}
	\left[\opintm(\eclass{\contcts}\sp{\ups}\right]^{\downs}&\subseteq\left[\opintm(\ellonets)\right]^{\downs}=\opintm(\ellonets\sp\downs)\\
		&=\opintm(\ellonets)=\opintm(\boundedmeasfunts).
	\end{align*}
We conclude that  $\opintm(\boundedmeasfunts)=\left[\posmap(\contcts)\right]^{\ups\downs}$.
 It is similarly proved that $\opintm(\boundedmeasfunts)=\left[\posmap(\contcts)\right]^{\downs\ups}$.
\end{proof}

The following should not go unnoticed.

\begin{proposition}\label{3_res:should_not_remain_unnoticed}
 Let $\ts$ be a locally compact Hausdorff space, let $\os$ be a \smc\ partially ordered vector space, and let $\npm\colon \borel\to\posos$ be a finite Borel measure. Then $\opintm(\pos{\contots})\subseteq[\opintm(\pos{\contcts})]\sp\ups$.
 %When $\npm$ is finite, $\opintm(\contots)\subseteq\opintm(\contcts)\sp\ups\cap\opintm(\contcts)\sp\downs$.
	\end{proposition}

\begin{proof}
Take an $f\in\contots$. For $n=1,2,\ldots$, there exists a $\varphi_n\in\contcts$ such that $\zerofunction\leq \varphi_n\leq\onefunction$ and $\varphi_n (x)=1$ when $\abs{f(x)}\leq 1/n$. Then $\varphi_1f$, $(\varphi_1\vee\varphi_2)f$,  $(\varphi_1\vee\varphi_2\vee\varphi_3)f$ is a sequence in $\pos{\contcts}$ that increases pointwise to $f$, so that the image sequence increases to $\opintm(f)$ by the monotone convergence theorem.
%This established the first statement. The second statement follows easily from the first by using the fact that then $\onefunction$ is integrable.
\end{proof}

%%%%%%%%%%%%%%%%%%%%%%%%%%%%%%%%%%%%%%%%%%%%%%%%%%%%%%%%%%%%%%%%%%%%%%%%%%%%%%%%%%%%%%%%%%%%%%%%%%%%%%%%%%%%%%%%%%%%%%%%%%%%%%%%%%%%

\section{Spectral theorems for positive algebra homomorphisms}
\label{3_sec:riesz_representation_theorems_for_positive_algebra_homomorphisms}

%%%%%%%%%%%%%%%%%%%%%%%%%%%%%%%%%%%%%%%%%%%%%%%%%%%%%%%%%%%%%%%%%%%%%%%%%%%%%%%%%%%%%%%%%%%%%%%%%%%%%%%%%%%%%%%%%%%%%%%%%%%%%%%%%%%%

\noindent In this section, a number of results from \cite{de_jeu_jiang:2022a} and  \cite{de_jeu_jiang:2022b} are combined with those from the present paper to yield two spectral theorems for positive algebra homomorphisms. The statements of the theorems are long, and some parts of them are identical, but we thought it worthwhile to collect all major results from these three papers that are applicable in a particular context in one place. With an eye towards possible further extensions and applications, we mention that the monotone convergence theorem, Fatou's lemma, and the dominated convergence theorem hold for the order integral that occurs in the results below; see \cite[Section~6.2]{de_jeu_jiang:2022a}.

It will be a recurring theme to know that a spectral measure for a positive algebra homomorphism from $\contcts$ into a partially ordered algebra is finite, as a consequence of the fact that it is the restriction of a positive algebra homomorphism that is defined on $\contots$. This will be possible when the algebra is a quasi-perfect partially ordered vector space. As \cref{3_res:examples_of_quasi_perfect_spaces} shows, this class of spaces contains a good number of spaces of practical interest.

The results on ups and downs below are established under the hypothesis that the codomain have the \csp. As indicated in \cref{3_subsec:the_countable_sup_property}, this condition is also satisfied for a variety of spaces of practical interest.

Let $\msm$ be a measure space. We recall for the convenience of the reader that  $\opintm\colon\integrablefun\to\os$ denotes the map $f\mapsto\ointm{f}$, and that we use the same notation for its restriction to subspaces of $\integrablefun$ and to quotients of such subspaces. In the results below, it will typically denote an extension of a positive operator $\posmap\colon\contcts\to\os$, or an operator that is compatible with a canonical map from $\contots$ into the domain of $\opintm$.

Our first result is for Banach lattice algebras with order continuous norms and monotone continuous multiplications. Banach lattice algebras of operators on infinite dimensional spaces will not often fall into this category\textemdash the order continuity of the norm is problematic\textemdash but on finite dimensional spaces these requirements are met. Function algebras such as $\ell^p$ for $1\leq p<\infty$ provide another class of examples to which \cref{3_res:positive_homomorphisms_into_banach_lattice_algebras} can be applied.

\begin{theorem}[Positive algebra homomorphisms from $\contcts$ into Banach lattice algebras with order continuous norms]\label{3_res:positive_homomorphisms_into_banach_lattice_algebras}
Let $\ts$ be a locally compact Hausdorff space, let $\oa$ be a Banach lattice algebra with an order continuous norm and a monotone continuous multiplication, and let $\posmap\colon \contcts\to\oa$ be a positive algebra homomorphism.

\begin{enumerate}
\item\label{3_part:positive_homomorphisms_into_banach_lattice_algebras_1}
There exists a unique regular Borel measure $\npm\colon \borel\to\posoaext$ on the Borel $\sigma$-algebra $\borel$ of $\ts$ such that
\begin{equation}\label{3_eq:positive_homomorphisms_into_banach_lattice_algebras_1}
\posmap(f)=\ointm{f}
\end{equation}
for all $f\in\contcts$. If $V$ is a non-empty open subset of $\ts$, then
\begin{equation}\label{3_eq:positive_homomorphisms_into_banach_lattice_algebras_inserted_number}
	\npm(V)=\psup\{\posmap(f) : f\prec V\}
\end{equation}
in $\posoaext$. If $K$ is a compact subset of $\ts$, then
\begin{equation*}%\label{3_eq:positive_homomorphisms_into_banach_lattice_algebras_1}
	\npm(K)=\pinf\{\posmap(f) : K\prec f\}
\end{equation*}
in $\posoa$.
\item\label{3_part:positive_homomorphisms_into_banach_lattice_algebras_1_extra}
The measure $\npm$ is a spectral measure which is inner regular at all Borel sets of finite measure. It is finite  if and only if
$\{\posmap(f) : f\in\pos{\contcts},\,\norm{f}\leq 1\}$ is bounded above in $\oa$. This is automatically the case when $\ts$ is compact, and also when $\oa$ is quasi-perfect and $\posmap$ is the restriction of a positive algebra homomorphism $\posmap\colon \contots\to\oa$. In the latter case, \cref{3_eq:positive_homomorphisms_into_banach_lattice_algebras_1} also holds for $f\in\contots$.

\item\label{3_part:positive_homomorphisms_into_banach_lattice_algebras_inserted}
$\ellonets\subseteq\boundedmeasfunaets$. If $\npm$ is finite, then $\ellonets=\boundedmeasfunaets$.

\item\label{3_part:positive_homomorphisms_into_banach_lattice_algebras_2}
The spaces $\integrablefunts$ and $\ellonets$ are \sDc\ vector lattices, and the naturally defined operators $\opintm$ from these spaces into $\oa$ are both $\sigma$-order continuous. The operator $\opintm\colon \ellonets\to\os$ is strictly positive. When $\oa$ has the \csp, $\ellonets$ is a \Dc\ vector lattice with the \csp, and $\opintm\colon \ellonets\to\oa$ is order continuous.

\item\label{3_part:positive_homomorphisms_into_banach_lattice_algebras_5}
Suppose that $\npm$ is finite. Then:
\begin{enumerate}
	\item $\opintm(\boundedmeasfunaets)\subseteq\npm(\ts)\oa\npm(\ts)$;
	\item  $\opintm\colon\boundedmeasfunaets\to\npm(\ts)\oa$,  $\opintm\colon\boundedmeasfunaets\to\oa\npm(\ts)$, and\newline $\opintm\colon\boundedmeasfunaets\to\npm(\ts)\oa\npm(\ts)$ preserve moduli;
	\item $\opintm\colon\boundedmeasfunaets\to\oa$ is a topological embedding of the Banach algebra $\boundedmeasfunaets$ as a Banach subalgebra of $\oa$;
	\item when $\opintm(\boundedmeasfunaets)$ is supplied with the partial ordering inherited from $\oa$, it is a unital Banach lattice algebra with identity element $\npm(\ts)$, and $\opintm\colon\boundedmeasfunaets\to\opintm(\boundedmeasfunaets)$ is an isomorphism of Banach lattice algebras.
	\item if $\oa$ has an identity $e$ and $\npm(\ts)\leq e$, then  $\opintm(\boundedmeasfunaets)$ is a Banach lattice subalgebra of $\oa$.
\end{enumerate}

\item\label{3_part:positive_homomorphisms_into_banach_lattice_algebras_inserted_once_more}
Suppose that $\npm$ is finite. For $x^\prime\in\pos{(\ocdualoa)}=\pos{(\ndualoa)}$ and $\mss\in\borel$, set $\npm_{x^\prime}(\mss)\coloneqq(\npm(\mss),x^\prime)$. Then $\npm_{x^\prime}\colon \borel\to\posR$ is a regular Borel measure, and we have $f\in{\lebfont L}^1(\pset,\alg,\npm_{x^\prime};\RR)$ for $f\in\integrablefunts$. For $f\in\integrablefunts$, $\opintm(f)=\ointm{f}$ is the unique element of $\oa$ such that
\begin{equation*}
	\left(\opintm(f), x^\prime\right)=\int_\pset\!f\di{\npm_{x^\prime}}
\end{equation*}
for all $x^\prime\in\pos{(\ocdualoa)}$.

\item\label{3_part:positive_homomorphisms_into_banach_lattice_algebras_3}
For $a\in\oa$, the following are equivalent:
\begin{enumerate_alpha}
	\item\label{3_part:positive_homomorphisms_into_banach_lattice_algebras_commuting_1}
	$a\posmap(f)=\posmap(f)a$ for all $f\in\contcts$;
	\item\label{3_part:positive_homomorphisms_into_banach_lattice_algebras_commuting_2}
	$a\npm(\mss)=\npm(\mss)a$ for all $\mss\in\borel$ with finite measure;	
	\item\label{3_part:positive_homomorphisms_into_banach_lattice_algebras_commuting_3}
	$a\opintm(f)=\opintm(f)a$ for all $f\in\integrablefun$.
\end{enumerate_alpha}

When $\contots\subseteq\integrablefun$, these are also equivalent to:
\begin{enumerate_alpha}[resume]
	\item\label{3_part:positive_homomorphisms_into_banach_lattice_algebras_commuting_4}
	$a\opintm(f)=\opintm(f)a$ for all $f\in\contots$.
	\end{enumerate_alpha}

\item\label{3_part:positive_homomorphisms_into_banach_lattice_algebras_6}
Suppose that $\npm$ is finite and that $\oa$ has the \csp. Then:
\begin{align}
		\label{3_eq:positive_homomorphisms_into_banach_lattice_algebras_2}
		&\opintm(\posboundedmeasfunaets)=\opintm(\posboundedmeasfunaets)^\up=\opintm(\posboundedmeasfunaets)^\down;\\	
		\label{3_eq:positive_homomorphisms_into_banach_lattice_algebras_3}
		&\opintm(\posboundedmeasfunaets)=\opintm(\eclass{\pos{\contcts}}^{\downs\ups})=\big[\posmap(\pos{\contcts})\big]^{\downs\ups};\\
		\label{3_eq:positive_homomorphisms_into_banach_lattice_algebras_4}
		&\opintm(\posboundedmeasfunaets)=\opintm(\eclass{\pos{\contcts}}^{\ups\downs})=\big[\posmap(\pos{\contcts})\big]^{\ups\downs};\\
		\label{3_eq:positive_homomorphisms_into_banach_lattice_algebras_5}
		&\opintm(\boundedmeasfunaets)=\opintm(\boundedmeasfunaets)^\up=\opintm(\boundedmeasfunaets)^\down;\\
		\label{3_eq:positive_homomorphisms_into_banach_lattice_algebras_6}
		&\opintm(\boundedmeasfunaets)=\big[\posmap(\contcts)\big]^{\ups\downs}=\big[\posmap(\contcts)\big]^{\downs\ups}.
\end{align}
Here the ups and downs of the images of $\opintm$ and $\posmap$ can be taken in $\oa$, $\npm(\ts)\oa$, $\oa\npm(\ts)$, $\npm(\ts)\oa\npm(\ts)$, or $\opintm(\boundedmeasfunaets)$ with equal outcomes.
\end{enumerate}
\end{theorem}

\begin{proof}

\begin{enumerate}
The parts~\ref{3_part:positive_homomorphisms_into_banach_lattice_algebras_1} and~\ref{3_part:positive_homomorphisms_into_banach_lattice_algebras_1_extra} are a consequence of \cite[Theorems~4.2 and~6.8]{de_jeu_jiang:2022b} for positive operators, and of \cite[Proposition~3.6]{de_jeu_jiang:2022b}, except the fact that $\npm$ is a spectral measure. Since $\posmap$ is now an algebra homomorphism, this follows from \cref{3_res:representing_measure_is_spectral}.

Part~\ref{3_part:positive_homomorphisms_into_banach_lattice_algebras_inserted} follows from \cref{3_res:integrable_function_for_spectral_measure_is_bounded}.

Part~\ref{3_part:positive_homomorphisms_into_banach_lattice_algebras_2} follows from  \cite[Proposition~6.14 and Theorem~6.17]{de_jeu_jiang:2022a}.

Part~\ref{3_part:positive_homomorphisms_into_banach_lattice_algebras_5} follows from  part~\ref{3_part:positive_homomorphisms_into_banach_lattice_algebras_inserted}, \cref{3_res:topological_embedding_with_closed_image} (which applies as the positive operator $\opintm$ between two Banach lattices is automatically continuous), \cref{3_res:integral_preserves_moduli}, and \cref{3_res:integral_into_riesz_algebra_is_riesz_algebra_homomorphism}.

Part~\ref{3_part:positive_homomorphisms_into_banach_lattice_algebras_inserted_once_more} follows from \cite[Proposition~6.8]{de_jeu_jiang:2022a}.

We turn to part~\ref{3_part:positive_homomorphisms_into_banach_lattice_algebras_3}, and
prove that part~\ref{3_part:positive_homomorphisms_into_banach_lattice_algebras_commuting_1} implies part~\ref{3_part:positive_homomorphisms_into_banach_lattice_algebras_commuting_2}. Suppose that $a=\pos{a}-\negt{a}$ in $\oa$ commutes with $\posmap(f)$ for all $f\in\contcts$. Then, in particular,
	\begin{equation}\label{3_eq:positive_homomorphisms_into_banach_lattice_algebras_commuting_1}
		\pos{a}\posmap(f)+\posmap(f)\negt{a}=\posmap(f)\pos{a}+\negt{a}\posmap(f)
	\end{equation}
	for all $f\in\pos{\contcts}$. Take an non-empty open subset $V$ of $\ts$ with finite measure. It follows from \cref{3_eq:positive_homomorphisms_into_banach_lattice_algebras_commuting_1}, \cref{3_eq:positive_homomorphisms_into_banach_lattice_algebras_inserted_number}, and the monotone continuity of the multiplication in $\oa$ that
	\begin{equation*}%\label{3_eq:positive_homomorphisms_into_banach_lattice_algebras_commuting_1}
		\pos{a}\npm(V)+\psm(V)\negt{a}=\npm(V)\pos{a}+\negt{a}\npm(V).
	\end{equation*}
	The inner regularity of $\npm$ can now be used to establish the same equality where $\npm(V)$ is replaced with $\npm(\mss)$ for a Borel subset $\mss$ of $\ts$ with finite measure. This shows that  part~\ref{3_part:positive_homomorphisms_into_banach_lattice_algebras_commuting_2} holds. When part~\ref{3_part:positive_homomorphisms_into_banach_lattice_algebras_commuting_2} holds, the definition of the order integral and the \mbox{($\sigma$-)}monotone continuity of the multiplication in $\oa$ imply that part~\ref{3_part:positive_homomorphisms_into_banach_lattice_algebras_commuting_3} holds. It is clear that part~\ref{3_part:positive_homomorphisms_into_banach_lattice_algebras_commuting_3} implies part~\ref{3_part:positive_homomorphisms_into_banach_lattice_algebras_commuting_1}. The statement regarding part~\ref{3_part:positive_homomorphisms_into_banach_lattice_algebras_commuting_4} is clear.

	We turn to part~\ref{3_part:positive_homomorphisms_into_banach_lattice_algebras_6}. Part~\ref{3_part:positive_homomorphisms_into_banach_lattice_algebras_5} shows that $\opintm$ maps $\integrablefun$ into $\npm(\ts)\oa$, and that $\opintm:\integrablefun\to\npm(\ts)\oa$ preserves moduli. We can view $\npm$ as a $\npm(\ts)\oa$-valued measure, and $\posmap$ as a positive operator from $\contcts$ into $\npm(\ts)\oa$. It then follows from \cref{3_res:properties_inherited_by_algebras_associated_to_idempotent} that \cref{3_res:ups_and_downs} can be applied to this context. Thus the equations \eqref{3_eq:positive_homomorphisms_into_banach_lattice_algebras_2}, \eqref{3_eq:positive_homomorphisms_into_banach_lattice_algebras_3}, and \eqref{3_eq:positive_homomorphisms_into_banach_lattice_algebras_4} (with the ups and downs of the images taken in $\npm(\ts)\oa$) follow from the equations \eqref{3_eq:ups_and_downs_3}, \eqref{3_eq:ups_and_downs_2}, and \eqref{3_eq:ups_and_downs_finite}, respectively. Similarly, the equations \eqref{3_eq:positive_homomorphisms_into_banach_lattice_algebras_5} and \eqref{3_eq:positive_homomorphisms_into_banach_lattice_algebras_6} (again with the ups and downs of the images taken in $\npm(\ts)\oa$) follow from \cref{3_res:pulling_back_ups_and_downs_for_vector_sublattices} and
	\cref{3_eq:full_image_of_posmap_is_contained_in_ups_and_downs_7}, respectively. Likewise, one establishes the validity of the equations \eqref{3_eq:positive_homomorphisms_into_banach_lattice_algebras_2}--\eqref{3_eq:positive_homomorphisms_into_banach_lattice_algebras_6} with the ups and downs of the images taken in $\npm(\ts)\oa$, and in $\npm(\ts)\oa\npm(\ts)$.
	It follows from \cref{3_rem:ups_and_downs_and_supersets} that the outcomes in all three cases agree with those in $\oa$.
	
	The monotone completeness of $\oa$ and the validity of
	\cref{3_eq:positive_homomorphisms_into_banach_lattice_algebras_5} with the up an down taken in $\oa$ imply that $\oa$ and $\opintm(\boundedmeasfunaets)$ as supersets for ups and downs also give the same result.
\end{enumerate}
\end{proof}

\begin{remark}\label{3_rem:banach_lattice_algebras_of_operators}
Suppose that, in \cref{3_res:positive_homomorphisms_into_banach_lattice_algebras}, $\oa$ consists of order continuous operators on a directed normal partially ordered vector space, and that $\npm$ is finite. For $x\in\posos$ and $x^\prime\in\pos{(\ocdualos)}$, the functional $a\mapsto (ax,x^\prime)$ is a positive order continuous functional on $\oa$, and setting $\npm_{x,x^\prime}(\mss)\coloneqq(\npm(\mss)x,x^\prime)$ for $\mss\in\borel$ yields a finite regular Borel measure $\npm_{x,x^\prime}\colon \borel\to\posR$. Using \cite[Proposition~6.8]{de_jeu_jiang:2022a}, it is easy to see that
$f\in{\lebfont L}^1(\pset,\alg,\npm_{x,x^\prime};\RR)$ when $f\in\integrablefunts$ and that,  for $f\in\integrablefunts$, $\opintm(f)=\ointm{f}$ is the unique element of $\oa$ such that
\begin{equation}\label{3_eq:weak_characterisation_of_order_integrals}
	\left(\opintm(f)x, x^\prime\right)=\int_\pset\!f\di{\npm_{x,x^\prime}}
\end{equation}
for all $x\in\posos$ and $x^\prime\in\pos{(\ocdualos)}$.
\end{remark}

\begin{remark}\label{3_rem:infinite_spectral_measure}
	It is not true that the spectral measure $\npm$ in \cref{3_res:positive_homomorphisms_into_banach_lattice_algebras} is always finite, even when $\posmap$ can be extended to a positive algebra homomorphism from $\contots$ into $\oa$. By way of example, let $S$ be an infinite set, supplied with the discrete topology. Then $\conto{S}$ is a Banach lattice algebra with an order continuous norm and monotone continuous multiplication. Consider the identity map $\posmap\colon \conto{S}\to\conto{S}$. Its representing spectral measure $\npm$ is given by $\npm(\mss)=\indicator{\mss}$ when $\mss$ is a finite subset of $S$, and by  $\npm(\mss)=\infty$ when $\mss$ is an infinite subset of $S$. Hence $\npm$ is not finite.
	
	Apparently, $\conto{S}$ is not quasi-perfect. It is normal, as is any Banach lattice with an order continuous norm, but it does not satisfy condition~\ref{3_part:quasi_perfect_spaces_2} in \cref{3_def:quasi_perfect_spaces}. This is easy to see. Take an infinite countable subset $\{s_1,s_2,\dotsc\}$ of $S$, and consider the sequence $\{\indicator{\{s_1,\dotsc,s_n\}}\}_{n=1}^\infty$ of indicator functions in $\contcts$. Since the dual of $\conto{S}$ can be identified with $\ell^1(S)$, it is clear that $\sup_{n\geq 1}\f{\indicator{\{s_1,\dotsc,s_n\}},x^\prime}<\infty$ for each $x^\prime\in\pos{(\odual{\conto{S}})}$. Yet the sequence has no supremum in $\conto{S}$.	
\end{remark}	

Whereas \cref{3_res:positive_homomorphisms_into_banach_lattice_algebras} does typically not apply to algebras of operators on infinite dimensional spaces, our next result, \cref{3_res:positive_homomorphisms_into_partially_ordered_algebras}, often \emph{does}; see \cref{3_res:order_continuous_operators_are_suitable_algebra}, \cref{3_res:order_continuous_operators_on_Banach_lattices_are_suitable_algebra}, and \cref{3_res:riesz_algebra_for_hilbert_spaces}. The order continuous operators on quasi-perfect spaces are an example of quasi-perfect partially ordered algebras with a monotone continuous multiplication to which it applies; another important example is in the context of operators on Hilbert spaces. In \cref{3_sec:special_positive_representations}, we shall see its consequences for positive representations of $\contots$ on Banach lattices and Hilbert spaces. As \cref{3_rem:infinite_spectral_measure} makes clear, \cref{3_res:positive_homomorphisms_into_banach_lattice_algebras} and \cref{3_res:positive_homomorphisms_into_partially_ordered_algebras} have an independent value of their own.

\begin{theorem}[Positive algebra homomorphisms from $\contcts$ into normal partially ordered algebras with a monotone continuous multiplication]\label{3_res:positive_homomorphisms_into_partially_ordered_algebras}
Let $\ts$ be a locally compact Hausdorff space, let $\oa$ be a monotone complete and normal partially ordered algebra  with a monotone continuous multiplication,  and let $\posmap\colon \contcts\to\oa$ be a positive algebra homomorphism such that $\{\posmap(f)\colon f\in\pos{\contcts},\ \norm{f}\leq 1\}$ is bounded above in $\oa$; the latter condition is automatically satisfied when $\ts$ is compact, and also when $\oa$ is quasi-perfect and $\posmap$ is the restriction of a positive algebra homomorphism $\posmap\colon \contots\to\oa$.
\begin{enumerate}
\item\label{3_part:positive_homomorphisms_into_partially_ordered_algebras_1}
There exists a unique regular Borel measure $\npm\colon \borel\to\posoaext$ on the Borel $\sigma$-algebra $\borel$ of $\ts$ such that
\begin{equation}\label{3_eq:positive_homomorphisms_into_partially_ordered_algebras_1}
\posmap(f)=\ointm{f}
\end{equation}
for all $f\in\contcts$. If $V$ is a non-empty open subset of $\ts$, then
\begin{equation*}\label{3_eq:positive_homomorpisms_into_bpartially_ordered_algebras_finites_inserted_number}
	\npm(V)=\psup\{\posmap(f) : f\prec V\}
\end{equation*}
in $\posoa$. If $K$ is a compact subset of $\ts$, then
\begin{equation*}%\label{3_eq:positive_homomorphisms_into_banach_lattice_algebras_1}
	\npm(K)=\pinf\{\posmap(f) : K\prec f\}
\end{equation*}
in $\posoa$.  When  $\posmap$ is the restriction of a positive algebra homomorphism $\posmap\!\colon \contots\to\oa$,  \cref{3_eq:positive_homomorphisms_into_banach_lattice_algebras_1} also holds for $f\in\contots$.

\item\label{3_part:positive_homomorphisms_into_partially_ordered_algebras_4}
The measure $\npm$ is a finite spectral measure which is inner regular at all Borel sets.

\item\label{3_part:positive_homomorphisms_into_partially_ordered_algebras_2}
The spaces $\integrablefunts$ and $\ellonets$ are both \sDc\ vector lattices, and the naturally defined operators $\opintm$ from these spaces into $\oa$ are both $\sigma$-order continuous.  When $\oa$ has the \csp, $\ellonets$ is a \Dc\ vector lattice with the \csp, and $\opintm\colon \ellonets\to\oa$ is order continuous.

\item\label{3_part:positive_homomorphisms_into_partially_ordered_algebras_5}

\begin{enumerate}
	\item\label{3_part:positive_homomorphisms_into_partially_ordered_algebras_5a} $\opintm(\ellonets)\subseteq\leftrightidalg{\psm(\pset)}{\oa}$;
	\item\label{3_part:positive_homomorphisms_into_partially_ordered_algebras_5b} The operators $\opintm\colon \ellonets\to\leftidalg{\psm(\pset)}{\oa}$, $\opintm\colon \ellonets\to\rightidalg{\psm(\pset)}{\oa}$, and  $\opintm\colon \ellonets\to\leftrightidalg{\psm(\pset)}{\oa}$ preserve moduli.
	\item \label{3_part:positive_homomorphisms_into_partially_ordered_algebras_5c}
	$\ellonets$ is a unital vector lattice algebra.
	\item \label{3_part:positive_homomorphisms_into_partially_ordered_algebras_5d}
	When supplied with the partial ordering inherited from $\oa$, the image $\opintm(\ellonets)$ is a unital vector lattice algebra with $\psm(\pset)$ as identity element, and $\opintm\colon \ellonets\to\opintm(\ellonets)$ is an isomorphism of vector lattice algebras.
	\item If $\oa$ is a vector lattice algebra with an identity element $e$ and $\npm(\ts)\leq e$, then $\opintm(\ellonets)$ is a vector lattice subalgebra of $\oa$.
\end{enumerate}

\item\label{3_part:positive_homomorphisms_into_partially_ordered_algebras_6}
Suppose that $\oa$ is also a normed algebra such that $\norm{x}\leq\norm{y}$ for all $x,y\in\oa$ with $0\leq x\leq y$. Then $\ellonets=\boundedmeasfunaets$, and
the map $\opintm\colon\boundedmeasfunaets\to\oa$ is a topological embedding of the Banach algebra $\boundedmeasfunaets$ as a closed subalgebra of $\oa$;

\item\label{3_part:positive_homomorphisms_into_partially_ordered_algebras_inserted_once_more}
For $x^\prime\in\pos{(\ocdualoa)}$ and $\mss\in\borel$, set $\npm_{x^\prime}(\mss)\coloneqq(\npm(\mss),x^\prime)$. Then $\npm_{x^\prime}\colon \borel\to\posR$ is a regular Borel measure, and $f\in{\lebfont L}^1(\pset,\alg,\npm_{x^\prime};\RR)$ when $f\in\integrablefunts$. For $f\in\integrablefunts$, $\opintm(f)=\ointm{f}$ is the unique element of $\oa$ such that
\begin{equation*}
	\left(\opintm(f), x^\prime\right)=\int_\pset\!f\di{\npm_{x^\prime}}
\end{equation*}
for all $x^\prime\in\pos{(\ocdualoa)}$.

\item\label{3_part:positive_homomorphisms_into_partially_ordered_algebras_3}
For $a\in\oa$, the following are equivalent:
\begin{enumerate_alpha}
	\item\label{3_part:positive_homomorphisms_into_partially_ordered_algebras_commuting_1}
	$a\posmap(f)=\posmap(f)a$ for all $f\in\contcts$;
	\item\label{3_part:positive_homomorphisms_into_partially_ordered_algebras_commuting_2}
	$a\npm(\mss)=\npm(\mss)a$ for all $\mss\in\borel$ with finite measure;	
	\item\label{3_part:positive_homomorphisms_into_partially_ordered_algebras_commuting_3}
	$a\opintm(f)=\opintm(f)a$ for all $f\in\integrablefun$.
\end{enumerate_alpha}

When $\contots\subseteq\integrablefun$, these are also equivalent to:
\begin{enumerate_alpha}[resume]
	\item\label{3_part:positive_homomorphisms_into_partially_ordered_algebras_commuting_4}
	$a\opintm(f)=\opintm(f)a$ for all $f\in\contots$.
	\end{enumerate_alpha}

\item\label{3_part:positive_homomorphisms_into_partially_ordered_algebras_7}
Suppose that $\ellonets=\boundedmeasfunaets$, and that $\oa$ has the \csp. Then:
	\begin{align*}
		&\opintm(\posboundedmeasfunaets)=\opintm(\posboundedmeasfunaets)^\up=\opintm(\posboundedmeasfunaets)^\down;\\	
		&\opintm(\posboundedmeasfunaets)=\opintm(\eclass{\pos{\contcts}}^{\downs\ups})=\big[\posmap(\pos{\contcts})\big]^{\downs\ups};\\
		&\opintm(\posboundedmeasfunaets)=\opintm(\eclass{\pos{\contcts}}^{\ups\downs})=\big[\posmap(\pos{\contcts})\big]^{\ups\downs};\\
		&\opintm(\boundedmeasfunaets)=\opintm(\boundedmeasfunaets)^\up=\opintm(\boundedmeasfunaets)^\down;\\
		&\opintm(\boundedmeasfunaets)=\big[\posmap(\contcts)\big]^{\ups\downs}=\big[\posmap(\contcts)\big]^{\downs\ups}.
	\end{align*}
Here the ups and downs of the images of $\posmap$ and $\opintm$ can be taken in $\oa$, $\npm(\ts)\oa$, $\oa\npm(\ts)$, $\npm(\ts)\oa\npm(\ts)$, or $\opintm(\boundedmeasfunaets)$ with equal outcomes.
\end{enumerate}
\end{theorem}

\begin{proof}
	It is clear that $\{\posmap(f)\colon f\in\pos{\contcts},\ \norm{f}\leq 1\}$ is bounded above in $\oa$ when $\ts$ is compact. It follows from \cite[Theorem~6.8]{de_jeu_jiang:2022b} that this is also true when $\oa$ is quasi-perfect and $\posmap$ is the restriction of a positive algebra homomorphism $\posmap\colon\contots\to\oa$.
	
	 The statements in the parts~\ref{3_part:positive_homomorphisms_into_partially_ordered_algebras_1} and~\ref{3_part:positive_homomorphisms_into_partially_ordered_algebras_4} follow from \cite[Proposition~3.6, Lemma~6.3, and Theorem~5.4]{de_jeu_jiang:2022b} and \cref{3_res:representing_measure_is_spectral}.
	
	 Part~\ref{3_part:positive_homomorphisms_into_partially_ordered_algebras_2} follows from \cite[Proposition~6.14 and Theorem~6.17]{de_jeu_jiang:2022a}.
	
	 Part~\ref{3_part:positive_homomorphisms_into_partially_ordered_algebras_5} follows from \cref{3_res:integral_preserves_moduli,3_res:integral_into_riesz_algebra_is_riesz_algebra_homomorphism}.

	 The first statement in part~\ref{3_part:positive_homomorphisms_into_partially_ordered_algebras_6} follows from \cref{3_res:integrable_function_for_spectral_measure_is_bounded} and the finiteness of $\npm$. It is easy to see that $\opintm\colon \boundedmeasfunaets\to\oa$ is continuous, so that the statement on the embedding follows from \cref{3_res:topological_embedding_with_closed_image}.

	 The proofs of the parts~\ref{3_part:positive_homomorphisms_into_partially_ordered_algebras_inserted_once_more},~\ref{3_part:positive_homomorphisms_into_partially_ordered_algebras_3}, and~\ref{3_part:positive_homomorphisms_into_partially_ordered_algebras_7} are as in the proof of \cref{3_res:positive_homomorphisms_into_banach_lattice_algebras}.
\end{proof}

\begin{remark}\label{3_rem:partially_ordered_algebras_of_operators}
Except for the fact that the finiteness of $\npm$ need not be supposed, \cref{3_rem:banach_lattice_algebras_of_operators} applies verbatim to the algebra $\oa$ in \cref{3_res:positive_homomorphisms_into_partially_ordered_algebras}. It yields the same weak characterisation of order integrals of integrable functions as in \cref{3_eq:weak_characterisation_of_order_integrals} in the case where $\oa$ consists of order continuous operators on a directed normal partially ordered vector space.
\end{remark}

\section{Positive representations of $\contots$ on Banach lattices and Hilbert spaces}\label{3_sec:special_positive_representations}

\noindent We shall now apply \cref{3_res:positive_homomorphisms_into_partially_ordered_algebras} to positive representations of $\contots$ on Banach lattices, and to representations of $\contoCts$ on Hilbert spaces.

As a preparation, we recall the following. Suppose that $\oatwo$ is a normed algebra, and that $\posmap$ is a bounded representation of $\oatwo$ on a normed space $\os$. Then the representation $\posmap$ is  \emph{non-degenerate} when $\os$ is the closed linear span of the elements $\posmap(b)x$ for $b\in\oatwo$ and $x\in\os$. When $\oatwo$ has a bounded left approximate identity, then it is not difficult to see that there is as largest invariant linear subspace $\os_{\mathrm{nd}}$ of $\os$ such that the restricted representation of $\oatwo$ on it is non-degenerate. This subspace is closed; in fact, it is  the closed linear span of the elements $\posmap(b)x$ for $b\in\oatwo$ and $x\in\os$. In our case, we shall apply this with $\oatwo=\contots$, so that $\os_{\mathrm{nd}}=\overline{{\mathrm {Span}} \{\posmap(f)x : f\in\contots, \,x\in\os\}}$. When working with a positive representation $\posmap$ of $\contots$ on a \Dc\ Banach lattice, the automatic continuity of $\posmap$ with respect to the regular norm (and then also the operator norm) implies that then also  $\os_{\mathrm{nd}}=\overline{{\mathrm {Span}} \{\posmap(f)x : f\in\contcts, x\in\os\}}$. Similar remarks apply to representations of $\contoCts$ on complex Hilbert spaces.

\subsection{Banach lattices}\label{3_subsec:Banach_lattices}
According to \cref{3_res:order_continuous_operators_on_Banach_lattices_are_suitable_algebra}, the order continuous operators on a \Dc\ normal Banach lattice form a \Dc\ normal Banach lattice algebra with a monotone continuous multiplication. Hence \cref{3_res:positive_homomorphisms_into_partially_ordered_algebras} can be applied to positive algebra homomorphisms $\posmap\colon \contots\to\ocontop{\os}$.  If, in addition, the norm on $\os$ is a Levi norm, then,  according to \cref{3_res:order_continuous_operators_on_Banach_lattices_are_suitable_algebra}, $\ocontop{\os}$ is even a  quasi-perfect Banach lattice algebra with a monotone continuous multiplication. In this case, the measure $\npm$ in \cref{3_res:positive_homomorphisms_into_partially_ordered_algebras} is automatically finite.

On taking into account \cref{3_res:topological_embedding_with_closed_image}; the weak characterisation of $\opintm$ in \cref{3_rem:partially_ordered_algebras_of_operators};  the fact that every vector sublattice of $\regularop{\os}$ has the \csp\ whenever $\os$ is a \Dc\ separable Banach lattice (see \cref{3_subsec:the_countable_sup_property}); and the fact that $\ocontop{\os}$ is a band in $\regularop{\os}$, we obtain the following from \cref{3_res:positive_homomorphisms_into_partially_ordered_algebras}. In its part~\ref{3_part:positive_representations_on_Banach_lattices_5}, the regular operators on the Banach lattice $\os$ are supplied with the regular norm.

\begin{theorem}[Positive representations of $\contots$ on \Dc\ normal Banach lattices]\label{3_res:positive_representations_on_Banach_lattices}
	Let $\ts$ be a locally compact Hausdorff space, let $\os$ be a \Dc\ Banach lattice such that  $\ocdualos$ separates the points of $\os$, and let $\posmap\colon \contots\to\ocontop{E}$ be a positive algebra homomorphism. Suppose that at least one of the following is satisfied:
	\begin{enumerate_roman}
		\item\label{3_part:positive_representations_on_Banach_lattices_i}
		 $\ts$ is compact;
		\item\label{3_part:positive_representations_on_Banach_lattices_ii}
		 the norm on $\os$ is a Levi norm.
		\end{enumerate_roman}
	Then the following hold.
	\begin{enumerate}
		\item\label{3_part:positive_representations_on_Banach_lattices_1}
		There exists a unique regular Borel measure $\npm\colon \borel\to\overline{\pos{\ocontop{E}}}$ on the Borel $\sigma$-algebra $\borel$ of $\ts$ such that
		\begin{equation}\label{3_eq:positive_representations_on_Banach_lattices_1}
			\posmap(f)=\ointm{f}
		\end{equation}
		for all $f\in\contcts$. If $V$ is a non-empty open subset of $\ts$, then
		\begin{equation*}\label{3_eq:positive_representations_on_Banach_lattices_inserted_number}
			\npm(V)=\psup\{\posmap(f) : f\prec V\}
		\end{equation*}
			in $\ocontop{\os}$ and in $\regularop{\os}$. If $K$ is a compact subset of $\ts$, then
		\begin{equation*}%\label{3_eq:positive_homomorphisms_into_banach_lattice_algebras_1}
			\npm(K)=\pinf\{\posmap(f) : K\prec f\}
		\end{equation*}
		in $\ocontop{\os}$ and in $\regularop{\os}$.
		
		\item The measure $\npm$ is a finite spectral measure which is inner regular at all Borel sets.   \Cref{3_eq:positive_homomorphisms_into_banach_lattice_algebras_1} also holds for $f\in\contots$.
		
		\item\label{3_part:positive_representations_on_Banach_lattices_2}
		We have $\ellonets=\boundedmeasfunaets$.

		\item\label{3_part:positive_representations_on_Banach_lattices_4}
		The spaces $\boundedmeasfunts$ and $\boundedmeasfunaets$ are both \sDc\ vector lattices, and the naturally defined operators $\opintm$ from these spaces into $\ocontop{E}$ are both $\sigma$-order continuous. When $\ocontop{E}$ has the \csp\ \uppars{which is the case when $\os$ is separable},   $\boundedmeasfunaets$ is a \Dc\ vector lattice with the \csp, and $\opintm\colon \boundedmeasfunaets\to\ocontop{\os}$ is order continuous.

		\item\label{3_part:positive_representations_on_Banach_lattices_5}
		
		\begin{enumerate}
			\item\label{3_part:positive_representations_on_Banach_lattices_5a} $\opintm(\boundedmeasfunaets)\subseteq\psm(\pset)\ocontop{\os}\psm(\pset)$.
			\item\label{3_part:positive_representations_on_Banach_lattices_5b}  The operators $\opintm\!\colon \boundedmeasfunaets\to\leftidalg{\psm(\pset)}{\ocontop{E}}$, $\opintm\!\colon \boundedmeasfunaets\to\rightidalg{\psm(\pset)}{\ocontop{E}}$, and $\opintm\colon\boundedmeasfunaets\to\leftrightidalg{\psm(\pset)}{\ocontop{E}}$ preserve moduli.
			\item\label{3_part:positive_representations_on_Banach_lattices_5e} The map $\opintm\colon \boundedmeasfunaets\to\opintm(\boundedmeasfunaets)$ is a topological embedding of the Banach algebra $\boundedmeasfunaets$ as a Banach subalgebra of $\ocontop{\os}$.
			\item \label{3_part:positive_representations_on_Banach_lattices_5d}
			When supplied with the partial ordering inherited from $\regularop{\os}$, the image $\opintm(\boundedmeasfunaets)$ is a unital vector lattice algebra with $\psm(\pset)$ as identity element. The map $\opintm\colon \boundedmeasfunaets\to\opintm(\boundedmeasfunaets)$ is then an isomorphism of vector lattice algebras.
			
			\item If $\npm(\ts)\leq\idop$, then $\opintm(\boundedmeasfunaets)$ is a Banach lattice subalgebra of the center of $\os$, and $\opintm\colon \boundedmeasfunaets\to\opintm(\boundedmeasfunaets)$ is an isomorphism of Banach lattice algebras.
			\item The image $\posmap(\contots)$ is a closed subalgebra of $\regularop{\os}$.
		\end{enumerate}

		\item\label{3_part:positive_representations_on_Banach_lattices_3}
		For $x\in\posos$ and $x^\prime\in\pos{(\ocdualos)}$, set $\npm_{x,x^\prime}(\mss)\coloneqq(\npm(\mss)x,x^\prime)$ for $\mss\in\borel$. Then $\npm_{x,x^\prime}\colon \borel\to\posR$ is a finite regular Borel measure on $\ts$. For $f\in\boundedmeasfunts$, $\opintm(f)=\ointm{f}$ is the unique element of $\regularop{\os}$ such that
		\begin{equation}\label{3_eq:weak_characterisation_of_order_integrals_banach_lattices}
			\left(\opintm(f)x, x^\prime\right)=\int_\pset\!f\di{\npm_{x,x^\prime}}
		\end{equation}
		for all $x\in\posos$ and $x^\prime\in\pos{(\ocdualos)}$.

		\item\label{3_part:positive_representations_on_Banach_lattices_6}
		For $S\in\ocontop{E}$, the following are equivalent:
		\begin{enumerate_alpha}
			\item\label{3_part:positive_representations_on_Banach_lattices_commuting_1}
			$S\posmap(f)=\posmap(f)S$ for all $f\in\contcts$;
			\item\label{3_part:positive_representations_on_Banach_lattices_commuting_2}
			$S\npm(\mss)=\npm(\mss)S$ for all $\mss\in\borel$ with finite measure;	
			\item\label{3_part:positive_representations_on_Banach_lattices_commuting_3}
			$S\opintm(f)=\opintm(f)S$ for all $f\in\boundedmeasfun$;
			\item\label{3_part:positive_representations_on_Banach_lattices_commuting_4}
			$S\posmap(f)=\posmap(f)S$ for all $f\in\contots$.
		\end{enumerate_alpha}
		\item\label{3_part:positive_representations_on_Banach_lattices_7}
		If $\ocontop{E}$ has the \csp\ \uppars{which is the case when $\os$ is separable}, then:
		\begin{align*}
			%\label{3_eq:positive_representations_on_Banach_lattices_algebras_2}
			&\opintm(\posboundedmeasfunaets)=\opintm(\posboundedmeasfunaets)^\up=\opintm(\posboundedmeasfunaets)^\down;\\	
			%\label{3_eq:positive_representations_on_Banach_lattices_algebras_3}
			&\opintm(\posboundedmeasfunaets)=\opintm(\eclass{\pos{\contcts}}^{\downs\ups})=\big[\posmap(\pos{\contcts})\big]^{\downs\ups};\\
			%\label{3_eq:positive_representations_on_Banach_lattices_algebras_4}
			&\opintm(\posboundedmeasfunaets)=\opintm(\eclass{\pos{\contcts}}^{\ups\downs})=\big[\posmap(\pos{\contcts})\big]^{\ups\downs};\\
			%\label{3_eq:positive_representations_on_Banach_lattices_algebras_5}
			&\opintm(\boundedmeasfunaets)=\opintm(\boundedmeasfunaets)^\up=\opintm(\boundedmeasfunaets)^\down;\\
			%\label{3_eq:positive_representations_on_Banach_lattices_algebras_6}
			&\opintm(\boundedmeasfunaets)=\big[\posmap(\contcts)\big]^{\ups\downs}=\big[\posmap(\contcts)\big]^{\downs\ups}.
		\end{align*}
		Here the ups and downs of the images of $\opintm$ and $\posmap$ can be taken in $\ocontop{\os}$, $\npm(\ts)\ocontop{\os}$, $\ocontop{\os}\npm(\ts)$, $\npm(\ts)\ocontop{\os}\npm(\ts)$, $\opintm(\boundedmeasfunaets)$, or $\regularop{E}$ with equal outcomes.
	\end{enumerate}
\end{theorem}

It is easy to see that, for a monotone net of regular operators on a Banach lattice with an order continuous norm, the existence of its order limit and the existence of its strong operator limit are equivalent, and that,  when they exist, they are equal. It follows from this that  the regular operators are closed in the bounded operators under the taking of monotone strong operator limits. These two observations enable us to extend \cref{3_res:positive_representations_on_Banach_lattices} for Banach lattices with order continuous norms.  We recall that the Banach lattices with order continuous Levi norms are precisely the KB-spaces.

\begin{theorem}[Positive representations of $\contots$ on Banach lattices with order continuous norms]\label{3_res:positive_representations_on_Banach_lattices_with_order_continuous_norms}
Let $\ts$ be a locally compact Hausdorff space, let $\os$ be a Banach lattice with an order continuous norm, and let $\posmap\colon \contots\to\regularop{\os}$ be a positive algebra homomorphism. Suppose that at least one of the following is satisfied:
\begin{enumerate_roman}
	\item\label{3_part:positive_representations_on_Banach_lattices_with_order_continuous_norms_i}
	$\ts$ is compact;
	\item\label{3_part:positive_representations_on_Banach_lattices_with_order_continuous_norms_ii}
	$\os$ is a KB-space.
\end{enumerate_roman}
Then all statements in \cref{3_res:positive_representations_on_Banach_lattices} hold, with the additional observations that $\ocontop{\os}=\regularop{\os}$ and that $\ocdualos=\odualos=\ndualos$. Furthermore:
\begin{enumerate}
\item\label{3_part:positive_representations_on_Banach_lattices_with_order_continuous_norms_2}
If $V$ is a non-empty open subset of $\ts$, then
\begin{equation}\label{3_eq:positive_representations_on_Banach_lattices_with_order_continuous_norms_1}
	\npm(V)=\SOTlim_{f\prec V}\posmap(f).
\end{equation}
If $K$ is a compact subset of $\ts$, then
\begin{equation}\label{3_eq:positive_representations_on_Banach_lattices_with_order_continuous_norms_2}
	\npm(K)=\SOTlim_{K\prec f}\posmap(f).
\end{equation}
In particular,
\[
\npm(\ts)=\SOTlim_{f\prec \ts}\posmap(f).
\]
\item\label{3_part:positive_representations_on_Banach_lattices_with_order_continuous_norms_1}
If $\seq{\mss}$ is a pairwise disjoint sequence in $\borel$ then, for all $x\in\os$,
\begin{equation*}
	\npm\left(\bigcup_{n=1}^\infty\mss_n\right)x=\sum_{n=1}^\infty\npm(\mss_n)x
\end{equation*}
in the norm topology of $\os$;

\item\label{3_part:positive_representations_on_Banach_lattices_with_order_continuous_norms_5}
The projection $\npm(\ts)$ projects onto the non-degenerate part of $\os$, i.e.,
\begin{align*}
	\npm(\ts)\os=\os_{\mathrm{nd}}&=\overline{{\mathrm {Span}} \{\posmap(f)x : f\in\contots, x\in\os\}}\\
	&=\overline{{\mathrm {Span}} \{\posmap(f)x : f\in\contcts, x\in\os\}}.
\end{align*}

\item\label{3_part:positive_representations_on_Banach_lattices_with_order_continuous_norms_3}
The \SOT-closed linear subspaces of the bounded \uppars{not necessarily regular} linear operators on $\os$ that are generated by the following sets are equal:
\begin{enumerate}
	\item%\label{2_part:riesz_representation_theorem_for_contots_operators_on_kb_space_1}
	$\{\posmap(f): f\in\contcts\}$;
	\item%\label{2_part:riesz_representation_theorem_for_contots_operators_on_kb_space_2}
	$\{\posmap(f): f\in\contots\}$;
	\item%\label{2_part:riesz_representation_theorem_for_contots_operators_on_kb_space_3}
	$\{\opintm(f): f\in\boundedmeasfunts\}$;
	%\item%label{2_part:riesz_representation_theorem_for_contots_operators_on_kb_space_4}
	%$\{\posmap(f): f\in\integrablefunts\}$	
	\item%\label{2_part:riesz_representation_theorem_for_contots_operators_on_kb_space_5}
	$\{\npm(V): V\text{ is an open subset of }\ts\}$;
	\item%\label{2_part:riesz_representation_theorem_for_contots_operators_on_kb_space_6}
	$\{\npm(K): K\text{ is a compact subset of }\ts\}$;
	\item%\label{2_part:riesz_representation_theorem_for_contots_operators_on_kb_space_7}
	$\{\npm(\mss): \mss\text{ is a Borel subset of }\ts\}$.
\end{enumerate}
\item\label{3_part:positive_representations_on_Banach_lattices_with_order_continuous_norms_4} If $\regularop{E}$ has the \csp\ \uppars{which is the case when $\os$ is separable}, then the equalities in part~\ref{3_part:positive_representations_on_Banach_lattices_7} of \cref{3_res:positive_representations_on_Banach_lattices} hold when the ups and downs of the images of $\opintm$ and $\posmap$ in $\regularop{\os}$ are defined in the strong operator topology;
\end{enumerate}
\end{theorem}

\begin{proof}
The first observation preceding the theorem implies that the parts~\ref{3_part:positive_representations_on_Banach_lattices_with_order_continuous_norms_2} and~\ref{3_part:positive_representations_on_Banach_lattices_with_order_continuous_norms_1} hold. It also the basis for \cite[Theorem~6.10]{de_jeu_jiang:2022b}, from which part~\ref{3_part:positive_representations_on_Banach_lattices_with_order_continuous_norms_3} is taken. For part~\ref{3_part:positive_representations_on_Banach_lattices_with_order_continuous_norms_4}, one uses both observations preceding the theorem.

We turn to part~\ref{3_part:positive_representations_on_Banach_lattices_with_order_continuous_norms_5}. The final two equalities were already observed in the beginning of this section. Since $\npm(\ts)\posmap(f)=\posmap(f)$ for $f\in\contots$, it is clear that $\os_{\mathrm{nd}}\subseteq\npm(\ts)\os$. The reverse inclusion follows from part~\ref{3_part:positive_representations_on_Banach_lattices_with_order_continuous_norms_2}.
\end{proof}

\begin{remark}\label{3_rem:positive_representations_on_Banach_lattices_with_order_continuous_norms}
	\quad
	\begin{enumerate}
		\item\label{3_rem:positive_representations_on_Banach_lattices_with_order_continuous_norms_1}
	It follows from part~\ref{3_part:positive_representations_on_Banach_lattices_with_order_continuous_norms_3} of \cref{3_res:positive_representations_on_Banach_lattices_with_order_continuous_norms} that the commutants and then also the bicommutants (both in the bounded, not necessarily regular, operators on $\os$) of the seven sets in part~\ref{3_part:positive_representations_on_Banach_lattices_with_order_continuous_norms_3} are equal. Thus part~\ref{3_part:positive_representations_on_Banach_lattices_6} of \cref{3_res:positive_representations_on_Banach_lattices} has be improved. Consequently, in the context of \cref{3_res:positive_representations_on_Banach_lattices_with_order_continuous_norms},  $\npm$ takes its values in the coinciding bicommutants of these sets. This statement, which is familiar from the representation theory of $\contCts$ on Hilbert spaces is, however, less precise than  part~\ref{3_part:positive_representations_on_Banach_lattices_with_order_continuous_norms_3}.
	\item\label{3_rem:positive_representations_on_Banach_lattices_with_order_continuous_norms_2}
	The positive algebra homomorphisms $\posmap\colon \contots\to\regularop{\os}$ for a KB-space $\os$ were studied in \cite{de_jeu_ruoff:2016}. \cref{3_res:positive_representations_on_Banach_lattices_with_order_continuous_norms} provides a substantial improvement of the main results in \cite{de_jeu_ruoff:2016}.
	%\item\label{3_rem:positive_representations_on_Banach_lattices_with_order_cont%For the sake of clarity we mention that, when $\ts$ is compact, it is not supposed in \cref{3_res:positive_representations_on_Banach_lattices} or  \cref{3_res:positive_representations_on_Banach_lattices_with_order_continuous_norms} that $\posmap$ be unital.
	\end{enumerate}
\end{remark}

We conclude this subsection with a discussion of possibly degenerate positive representations of $\contots$ on KB-spaces.

Let $\os$ be a Banach lattice, and let $P\colon\os\to\os$ be a positive projection. Following ideas of Schaefer's (see \cite[p.~214]{schaefer_BANACH_LATTICES_AND_POSITIVE_OPERATORS:1974}) it can be shown that the closed subspace $P\os$ is a vector lattice when supplied with the partial ordering inherited from $\os$. Its modulus is given by $\lvert x \rvert _{P\os}=P\abs{x}$ for $x\in P\os$. The definition $\lVert x\rVert_{P\os}\coloneq \norm{\abs{x}_{P\os}}=\norm{P\abs{x}}$ yields a lattice norm on $P\os$ that is equivalent to the restriction of the original norm to $P\os$. Thus $P\os$ is a Banach lattice in the inherited partial ordering and with the norm $\lVert \,\cdot\,\rVert_{P\os}$. We refer to  \cite[Theorem~5.59]{abramovich_aliprantis_INVITATION_TO_OPERATOR_THEORY:2002} for details and additional information.

If $\os$ is a KB-space, then the equivalence of $\lVert \,\cdot\,\rVert_{P\os}$ and the original norm on $P\os$ implies that $P\os$ is also a KB-space.

Suppose that $\posmap\colon\contots\to\regularop{\os}$ is a positive representation on the KB-space $\os$, so that $\npm(\ts)$ projects onto the non-degenerate part $\os_{\mathrm{nd}}$ of $\os$. Since $\posmap(f)\npm(\ts)=\posmap(f)$ for $f\in\contots$, $\posmap(\contots)$ vanishes on the kernel of $\npm(\ts)$. We may, therefore, just as well restrict our attention to the representation of $\contots$ on the complementing KB-space $\os_{\mathrm{nd}}$. We claim that this representation on $\os_{\mathrm{nd}}$ is non-degenerate. Its non-degenerate part is $\overline{{\mathrm {Span}} \{\posmap(f)x : f\in\contcts, x\in\npm(\ts)\os\}}$, where the closure is in the lattice norm on $\os_{\mathrm{nd}}$. Since this lattice norm is equivalent to the restricted original norm, and $\os_{\mathrm{nd}}$ is closed, this closure is also the closure in $\os$. It is not difficult to see that it is  $\os_{\mathrm{nd}}$, establishing our claim.

\cref{3_res:positive_representations_on_Banach_lattices,3_res:positive_representations_on_Banach_lattices_with_order_continuous_norms} now apply to the positive representation of $\contots$ on $\os_{\mathrm{nd}}$. In this case, however, we know from the above that the projection $\npm(\ts)$ in these theorems is the identity operator (on the Banach lattice-subspace $\os_{\mathrm{nd}}$ of the original $\os$). In particular, part~\ref{3_part:positive_representations_on_Banach_lattices_5} of \cref{3_res:positive_representations_on_Banach_lattices} shows that, for a positive representation of $\contots$ on a KB-space $\os$, $\boundedmeasfunaets$ embeds as a Banach lattice subalgebra of the ideal centre of $\os_{\mathrm{nd}}$.\footnote{Some care is in order here: in this statement, $\npm$ is the spectral measure corresponding to the restricted representation, which is\textemdash in the obvious sense\textemdash the restriction of the original one. However, as these two spectral measures have the same zero sets, the corresponding spaces $\boundedmeasfunaets$ are, in the end, still equal.}

%It is clear from \cref{3_res:positive_representations_on_Banach_lattices,3_res:positive_representations_on_Banach_lattices_with_order_continuous_norms} that, from a practical point of view, the most convenient contexts to study positive algebra homomorphisms $\posmap\colon \contots\to\regularop{\os}$ in are those where $\os$ is separable and has an order continuous norm, and where $\ts$ is compact or/and $\os$ is (even) a KB-space.

\subsection{Hilbert spaces}\label{3_subsec:Hilbert_spaces}
We turn to representations on Hilbert spaces. Suppose that $\posmap\colon \contoCts\to\boundedh$ is a representation of the complex \Calgebra\ $\contoCts$ on a complex Hilbert space $\hilbert$. It follows from Kaplansky's density theorem (see  \cite[Theorem~5.3.5]{kadison_ringrose_FUNDAMENTALS_OF_THE_THEORY_OF_OPERATOR_ALGEBRAS_VOLUME_I:1983}), that $\oa\coloneqq\overline{\posmap(\contoCts)}\sp\SOT$ is a commutative \SOT-closed \Csubalgebra\ of  $\boundedh$, and that the set $\oa_\sa$ of its self-adjoint elements is $\overline{\posmap(\contots)}\sp\SOT$.  Since $\oa_\sa$ is a quasi-perfect vector lattice algebra by \cref{3_res:riesz_algebra_for_hilbert_spaces}, \cref{3_res:positive_homomorphisms_into_partially_ordered_algebras} applies to the positive algebra homomorphism $\posmap\colon \contots\to\oa_\sa$ and yields a representing finite regular spectral Borel measure $\npm\colon \borel\to\oa_\sa$. It takes its values in the idempotents in $\oa_\sa$, which are orthogonal projections on $\hilbert$.  Before giving the full statement, let us identify in operator algebraic terms the condition that  $\oa_\sa$ have the \csp, which is instrumental to the results on ups and downs. As a preparation for this, we note that $\npm(\ts)\posmap(f)=\posmap(f)\npm(\ts)=\posmap(f)$ for $f\in\contots$ because $\opintm\colon \boundedmeasfunts\to\oa_\sa$ is an algebra homomorphism. Consequently, $\oa$ is a unital algebra, so that the following result applies to it.

\begin{lemma}\label{3_res:sigma_finite_von_Neumann_algebra}
Let $\poalgfont{M}$ be a 	commutative unital \SOT-closed \Csubalgebra\ of $\boundedh$, not necessarily containing the identity operator. Then the vector lattice $\poalgfont{M}_\sa$ has the \csp\ if and only if every subset of $\poalgfont{M}$ that consists of non-zero mutually orthogonal self-adjoint  projections is countable.
\end{lemma}

\begin{proof}
	The unital \Calgebra\ $\poalgfont{M}$ is isomorphic to $\cont{K;\CC}$ for a Hausdorff space $K$. Since the vector lattice $\cont{K}\simeq\poalgfont{M}_\sa$ is Dedekind complete as a consequence of the fact that $\poalgfont{M}$ is strongly closed, $K$ is extremally disconnected; see \cite[Theorem~43.11]{luxemburg_zaanen_RIESZ_SPACES_VOLUME_I:1971}, for example. We now recall that a vector lattice has the \csp\ if and only if every disjoint system of strictly positive elements that is bounded from above is countable; see \cite[Theorem~29.3]{luxemburg_zaanen_RIESZ_SPACES_VOLUME_I:1971}. This makes it clear that every subset of $\cont{K}$  that consists of non-zero disjoint idempotents is countable when $\cont{K}$ has the \csp. Conversely, suppose that every subset of $\cont{K}$ that consists of non-zero mutually disjoint idempotents is countable, and let $\{f_i:i\in I\}$ be a system of mutually disjoint non-zero positive elements of $\cont{K}$ that is bounded above. Using \cite[Theorem~2.7]{rudin_REAL_AND_COMPLEX_ANALYSIS_SECOND_EDITION:1974}, the fact that $K$ is extremally disconnected implies that the interior of the support of each $f_i$ contains a non-empty clopen subset of $K$. As a consequence of the assumption, the characteristic functions of these sets must form an a most countable subset of $\cont{K}$. Hence the index set $I$ is countable.
\end{proof}

Hence the \csp\ of $\oa_\sa$ is equivalent to $\oa$ being a $\sigma$-finite von Neumann algebra in the sense of  \cite[p.~62]{pedersen_C-STAR-ALGEBRAS_AND_THEIR_AUTOMORPHISM_GROUPS:1979}.\footnote{Recall that, in \cite{pedersen_C-STAR-ALGEBRAS_AND_THEIR_AUTOMORPHISM_GROUPS:1979}, a von Neumann algebra need not contain the identity operator.} On separable Hilbert spaces, this is obviously always satisfied.

%We make a few additional comments on the statements in \cref{3_res:positive_representations_on_Hilbert_spaces} and its proof, which consists largely of an application of \cref{3_res:positive_homomorphisms_into_partially_ordered_algebras}.

We recall that the existence of the extremum of a monotone net of self-adjoint operators in $\boundedh$ and the existence  of its strong operator limit are equivalent and that, when they exists, they are equal. Thus part~\ref{3_part:positive_representations_on_Hilbert_spaces_2} of \cref{3_res:positive_representations_on_Hilbert_spaces} follows from the corresponding formulas in \cref{3_res:positive_homomorphisms_into_partially_ordered_algebras}. Its part~\ref{3_part:positive_representations_on_Hilbert_spaces_8} then follows just as in the proof of \cref{3_res:positive_representations_on_Banach_lattices_with_order_continuous_norms}. Since $\oa$ is isomorphic to its restriction to $\npm(\ts)\hilbert=\hilbert_{\mathrm {nd}}$, we see that $\oa_\sa$ certainly has the \csp\ when $\hilbert_{\mathrm{nd}}$ is separable.

This all being said, \cref{3_res:positive_homomorphisms_into_partially_ordered_algebras} implies the largest part of the following result. The equality of the \SOT-closed linear subspaces that are generated by the sets in part~\ref{3_part:positive_representations_on_Hilbert_spaces_4} follows
from \cite[Theorem~6.12]{de_jeu_jiang:2022b}.
Furthermore, for $x\in\hilbert$,  the functional $S\mapsto\lrinp{Sx,x}$ is a positive order continuous functional on $\oa_\sa$, and part~\ref{3_part:positive_representations_on_Hilbert_spaces_inserted_again} then follows from
\cite[Proposition~6.8]{de_jeu_jiang:2022a}.

\begin{theorem}[Representations of $\contoCts$ on Hilbert spaces]\label{3_res:positive_representations_on_Hilbert_spaces}
Let $\ts$ be a locally compact Hausdorff space, let $\hilbert$ be a complex Hilbert space, and let $\posmap\colon \contoCts\to\boundedh$ be a $^\ast$-homomorphism. Set $\oa\coloneqq\overline{\posmap(\contoCts)}\sp\SOT$. Then $\oa$ is a unital commutative \Calgebra, and its algebra $\oa_\sa$ of self-adjoint elements is $\overline{\posmap{(\contots})}\sp\SOT$. We supply $\oa_\sa$ with the partial ordering that is inherited from the usual partial ordering on $\boundedh_\sa$, so that $\oa_\sa$ becomes a quasi-perfect vector lattice algebra.
\begin{enumerate}
\item\label{3_part:positive_representations_on_Hilbert_spaces_1}
There exists a unique regular Borel measure on $\ts$ with values in the extended positive cone of $\oa_\sa$ such that
\begin{equation}\label{3_eq:positive_representations_on_hilbert_space_1}
\posmap(f)=\ointm{f}
\end{equation}
for all $f\in\contcts$. The measure $\npm$ is a finite spectral measure which is inner regular at all Borel sets, and \cref{3_eq:positive_representations_on_hilbert_space_1} also holds for $f\in\contots$. If $V$ is a non-empty open subset of $\ts$, then
\begin{equation}\label{3_eq:positive_representations_on_hilbert_spaces_inserted_1}
	\npm(V)=\SOTlim_{f\prec V}\posmap(f).
\end{equation}
If $K$ is a compact subset of $\ts$, then
\begin{equation}\label{3_eq:positive_representations_on_hilbert_spaces_inserted_2}
	\npm(K)=\SOTlim_{K\prec f}\posmap(f).
\end{equation}
In particular,
\[
\npm(\ts)=\SOTlim_{f\prec \ts}\posmap(f).
\]

\item\label{3_part:positive_representations_on_Hilbert_spaces_inserted}
If $\seq{\mss}$ is a pairwise disjoint sequence in $\borel$ then, for all $x\in\os$,
\begin{equation*}
	\npm\left(\bigcup_{n=1}^\infty\mss_n\right)x=\sum_{n=1}^\infty\npm(\mss_n)x
\end{equation*}
in the norm topology of $\hilbert$.

\item\label{3_part:positive_representations_on_Hilbert_spaces_2}
We have $\ellonets=\boundedmeasfunaets$.

\item\label{3_part:positive_representations_on_Hilbert_spaces_inserted_again}
For $x\in\hilbert$, set $\npm_{x,x}(\mss)\coloneqq\lrinp{\npm(\mss)x,x}$ for $\mss\in\borel$. Then $\npm_{x,x}\colon \borel\to\posR$ is a finite regular Borel measure on $\ts$. For $f\in\boundedmeasfunts$, $\opintm(f)=\ointm{f}$ is the unique element of $\boundedh$ such that
\begin{equation}\label{3_eq:weak_characterisation_of_order_integrals_hilbert spaces}
	\lrinp{\opintm(f)x, x}=\int_\pset\!f\di{\npm_{x,x}}
\end{equation}
for all $x\in\hilbert$.

\item\label{3_part:positive_representation_on_Hilbert_spaces_3}
The spaces $\boundedmeasfunts$ and $\boundedmeasfunaets$ are both \sDc\ vector lattices, and the naturally defined operators $\opintm$ from these spaces into $\oa_\sa$ are both $\sigma$-order continuous. If $\oa_\sa$ has the \csp\ \uppars{equivalently: when $\oa$ is $\sigma$-finite; this is certainly the case when $\hilbert_{\mathrm{nd}}$ is separable},  then $\boundedmeasfunaets$ is a \Dc\ vector lattice with the \csp, and $\opintm\colon \boundedmeasfunaets\to\oa_\sa$ is order continuous.

\item\label{3_part:positive_representations_on_Hilbert_spaces_5} The map $\opintm\colon \boundedmeasfunaets\to\oa_\sa$ is a topological embedding of the Banach lattice algebra $\boundedmeasfunae$ as a Banach lattice subalgebra of $\oa_\sa$. The map $\posmap\colon \contots\to\oa_\sa$ is a Banach lattice algebra homomorphism with closed range.

\item\label{3_part:positive_representations_on_Hilbert_spaces_4}
The  \SOT-closed complex linear subspaces of $\boundedh$ that are generated by the following sets are all equal to the algebra $\oa$:
	\begin{enumerate}
		\item\label{3_part:positive_representations_on_Hilbert_spaces_4-1}
		$\{\posmap(f): f\in\contcts\}$;
		\item\label{3_part:positive_representations_on_Hilbert_spaces_4-2}
		$\{\posmap(f): f\in\contots\}$;
		\item\label{3_part:positive_representations_on_Hilbert_spaces_3-3}
		$\{\opintm(f): f\in\boundedmeasfunts\}$;
		\item\label{3_part:positive_representations_on_Hilbert_spaces_3-5}
		$\{\npm(V): V\text{ is an open subset of }\ts\}$;
		\item\label{3_part:positive_representations_on_Hilbert_spaces_3-6}
		$\{\npm(K): K\text{ is a compact subset of }\ts\}$;
		\item\label{3_part:positive_representations_on_Hilbert_spaces_3-7}
		$\{\npm(\mss): \mss\text{ is a Borel subset of }\ts\}$.
	\end{enumerate}
	
\item\label{3_part:positive_representations_on_Hilbert_spaces_8}
The projection $\npm(\ts)$ projects onto the non-degenerate part of $\hilbert$, i.e.,
\begin{align*}
	\npm(\ts)\hilbert=\hilbert_{\mathrm{nd}}&=\overline{{\mathrm {Span}} \{\posmap(f)x : f\in\contoCts, x\in\hilbert\}}\\
	&=\overline{{\mathrm {Span}} \{\posmap(f)x : f\in\contcCts, x\in\hilbert\}}.
\end{align*}

\end{enumerate}
\end{theorem}

\begin{remark}\label{3_rem:automatic_continuity}
		After extending the order integral in the natural way to complex-valued functions, one obtains the (possibly non-unital) representations of \Calgebras\ $\opintm\colon \boundedmeasfunCts\to\boundedh$ and $\opintm\colon \boundedmeasfunCaets\to\boundedh$. Although their definition via the order integral did not involve any topology, the automatic continuity of these representations implies, for example,  that, for $f\in\boundedmeasfunCts$, $\opintm(f)$ is the uniform limit of natural linear combinations of the $\npm(\mss)$ for $\mss\in\borel$.
\end{remark}

\begin{remark}\label{3_rem:principled_order_theoretical_approach_1}
In the context of \cref{3_res:positive_representations_on_Hilbert_spaces}  take $x,y\in\hilbert$, and set  $\npm_{x,y}(\mss)\coloneqq\lrinp{\npm(\mss)x,y}$ for $\mss\in\borel$. Then $\npm_{x,y}\colon \borel\to\CC$ is a finite regular Borel measure on $\ts$, and  \cref{3_eq:weak_characterisation_of_order_integrals_hilbert spaces} implies that, for $f\in\contoCts$, $\posmap(f)$ is the unique element of  $\boundedh$ such that
\begin{equation}\label{3_eq:weak_characterisation_of_order_integrals_hilbert spaces_again}
	\lrinp{\posmap(f)x, y}=\int_\pset\!f\di{\npm_{x,y}}
\end{equation}
for all $x,y\in\hilbert$.

It follows from this that,  for compact $\ts$ and a unital representation of $\contCts$), our measure $\npm$ coincides with the spectral measure from the literature. As in \cite[Chapter IX.1]{conway_A_COURSE_IN_FUNCTIONAL_ANALYSIS_SECOND_EDITION:1990}, for example,   this is usually constructed by first using \cref{3_eq:weak_characterisation_of_order_integrals_hilbert spaces_again} for $f\in\contCts$ to \emph{define} $\npm_{x,y}$, then using \cref{3_eq:weak_characterisation_of_order_integrals_hilbert spaces_again} again to define $\posmap(f)$ for $f\in\boundedmeasfunCts$, and finally setting $\npm(\mss)\coloneqq \posmap(\mss)$ for $\mss\in\borel$. Our approach is from the opposite direction. The spectral measure is found first, and then the relation with the measures $\npm_{x,y}$ is an immediate consequence.

 These spectral measures from the literature are such that $\npm(\ts)=I$. For our measures, which exist in the most general setting of locally compact spaces and possibly non-degenerate representations, this need not be the case.
 \end{remark}

\begin{remark}\label{3_rem:principled_order_theoretical_approach_2}
Our principled order-theoretical approach to the spectral theorem for representations of commutative \Calgebras\ on Hilbert spaces  appears to be new. For unital representations of $\contCts$ with $\ts$ compact, traits of it appear in the statement of  \cite[Theorem~5.2.6]{kadison_ringrose_FUNDAMENTALS_OF_THE_THEORY_OF_OPERATOR_ALGEBRAS_VOLUME_I:1983} and its proof, where \cref{3_eq:positive_representations_on_hilbert_spaces_inserted_1} and the outer regularity of $\npm$ can be found. The full underlying picture with $\npm$ as a regular Borel measure that also satisfies \cref{3_eq:positive_representations_on_hilbert_spaces_inserted_2}, and with the functional calculi as special cases of generalised Lebesgue integrals is, however, not visible there.
 \end{remark}

 Part~\ref{3_part:positive_homomorphisms_into_partially_ordered_algebras_7} of \cref{3_res:positive_homomorphisms_into_partially_ordered_algebras}
does not make an appearance in \cref{3_res:positive_representations_on_Hilbert_spaces} because, in this particular context, it is an ingredient to the proof of the following strengthening of it.

\begin{theorem}\label{3_res:Borel_functional_calculus}
Let $\ts$ be a locally compact Hausdorff space, let $\hilbert$ be a complex Hilbert space, and let $\posmap\colon \contoCts\to\boundedh$ be a $^\ast$-homomorphism. Set $\oa\coloneqq\overline{\posmap(\contoCts)}\sp\SOT$.  Suppose that $\oa_\sa=\overline{\posmap(\contots)}\sp\SOT$ has the \csp\ or, equivalently, that $\oa$ is $\sigma$-finite; this is certainly the case when $\hilbert_{\mathrm{nd}}$ is separable.

Then:
	\begin{align*}
	\oa&=\opintm(\boundedmeasfunCaets);\\
	\oa_\sa&=\opintm(\boundedmeasfunaets)=\big[\posmap(\contcts)\big]^{\ups\downs}=\big[\posmap(\contcts)\big]^{\downs\ups};\\
	 \pos{\oa_\sa}&=\opintm(\posboundedmeasfunaets)=\posmap(\eclass{\pos{\contcts}}^{\downs\ups})=\big[\posmap(\pos{\contcts})\big]^{\downs\ups}\\
	&\phantom{=\opintm(\posboundedmeasfunaets)\,\,}=\opintm(\eclass{\pos{\contcts}}^{\ups\downs})=\big[\posmap(\pos{\contcts})\big]^{\ups\downs}.
	\end{align*}
Consequently, $\opintm\colon\oa\to\boundedmeasfunCaets$ is an isomorphism.
\end{theorem}

\begin{proof}
	
	We establish the first equation. The remaining ones follows from this and  part~\ref{3_part:positive_homomorphisms_into_partially_ordered_algebras_7} of \cref{3_res:positive_homomorphisms_into_partially_ordered_algebras}; the final statement is then also clear as $\opintm$ is injective on $\boundedmeasfunCaets$. We know that $\opintm(\boundedmeasfunCaets)\subseteq\oa$. We prove the reverse inclusion. Since $\posmap(\contoCts)\subseteq\opintm(\boundedmeasfunCaets)$, it suffices to show that  $\opintm(\boundedmeasfunCaets)$ is strongly closed. Since it is a \Calgebra, its being strongly closed is  equivalent to its self-adjoint part being closed under the taking of \SOT-limits of increasing nets; see  \cite[Theorem~2.4.4]{pedersen_C-STAR-ALGEBRAS_AND_THEIR_AUTOMORPHISM_GROUPS:1979}. We now show this. Suppose that $\net{S}\subseteq\opintm(\boundedmeasfunaets)$ is a net of self-adjoint operators and that $S_\alpha\uparrow S$ strongly for some $S\in\boundedh_\sa$. Since $\oa_\sa$ is strongly closed, we have $S\in\oa_\sa$.  Hence $S_\alpha\uparrow S$ in $\oa$ in order, and then the countability assumption entails that  $S\in\opintm{(\boundedmeasfunaets)}$ by part~\ref{3_part:positive_homomorphisms_into_partially_ordered_algebras_7} of \cref{3_res:positive_homomorphisms_into_partially_ordered_algebras}.
\end{proof}

\begin{remark}\label{3_rem:up-down}\quad
\begin{enumerate}
\item\label{3_part:up-down_1}
According to \cref{3_res:Borel_functional_calculus}, if $\oa$ is $\sigma$-finite (in particular: if $\hilbert_{\mathrm{nd}}$ is separable), then the \SOT-closed subalgebra of $\boundedh$ that is generated by $\posmap(\contoCts)$ is equal to the image of the accompanying Borel functional calculus. For a normal operator on a separable Hilbert space, this is \cite[Lemma~IX.8.7]{conway_A_COURSE_IN_FUNCTIONAL_ANALYSIS_SECOND_EDITION:1990}, which is proved by rather different methods.\footnote{Given the necessary preparations for this result, the result could, perhaps, have been called a theorem.}

When $\hilbert$ is separable, $\ts$ is compact, and $\contCts$ is realised as a \Csubalgebra\ of $\boundedh$ that contains the identity operator, a proof\textemdash again using methods different from ours\textemdash is sketched in \cite[Theorem~1.57]{folland_A_COURSE_IN_ABSTRACT_HARMONIC_ANALYSIS_SECOND_EDITION:2016}.

We are not aware of a reference for the general result in \cref{3_res:Borel_functional_calculus}, which, in our approach, is essentially a consequence of the characterisation of \SOT-closed \Csubalgebras\ of $\boundedh$ as the ones that are monotone \SOT-closed, combined with the monotone convergence theorem for the order integral.

\item\label{3_part:up-down_2}
For a separable Hilbert space, \cite[Theorem~2.4.3]{pedersen_C-STAR-ALGEBRAS_AND_THEIR_AUTOMORPHISM_GROUPS:1979} shows that $\pos{\oa_\sa}=
[\posmap(\pos{\contots)}]\sp{\ups\downs}$; \cref{3_res:Borel_functional_calculus} shows that even $\pos{\oa_\sa}=
[\posmap(\pos{\contcts)}]\sp{\ups\downs}$.
Suppose that $\hilbert_{\mathrm{nd}}=\hilbert$ and that $\oa$ is $\sigma$-finite. It then follows from \cite[Theorem~4.2.2]{takesaki_THEORY_OF_OPERATOR_ALGEBRAS_VOLUME_I_REPRINT_OF_1979_EDITION:2002} that $\oa_\sa=[\posmap(\contots)]^{\ups\downs}$;  \cref{3_res:Borel_functional_calculus} yields that $\oa_\sa=[\posmap(\contcts)]^{\ups\downs}$.  Our result, however,  does not without further effort yield any information about norms as in the cited results.
\end{enumerate}
\end{remark}

\section{JBW-algebras}\label{3_sec:JBW-algebras}

\noindent In this section, we show how the general principles in the present paper lead to stronger and new results in spectral theory for JBW-algebras. Among others, it will become clear how the spectral resolution for an element of a JBW-algebra arises from an underlying spectral measure. The use of the latter in this context appears to be new.

Let ${\poalgfont M}$ be a JBW-algebra with identity element 1. Take $a\in{\poalgfont M}$, and let ${\poalgfont C}(a,1)$ denote the norm closed Jordan subalgebra generated by $a$ and 1. It is an associative JB-algebra, and by \cite[Proposition~1.12]{alfsen_shultz_GEOMETRY_OF_STATE_SPACES_OF_OPERATOR_ALGEBRAS:2003} there exist a compact Hausdorff space $\cont{\ts}$ and an isometric  unital algebra isomorphism $b\mapsto \widehat b$ from ${\poalgfont C}(a,1)$ onto $\cont{\ts}$. Hence ${\poalgfont C}(a,1)$ is a unital Banach lattice algebra in the ordering inherited from ${\poalgfont M}$.

The spectrum $\sigma(a)$ of $a$ in ${\poalgfont M}$ is the set of all $\lambda\in\RR$ such that $\lambda 1- a$ is not Jordan invertible in ${\poalgfont M}$. By  \cite[Proposition~1.17]{alfsen_shultz_GEOMETRY_OF_STATE_SPACES_OF_OPERATOR_ALGEBRAS:2003}, $\sigma(a)$ equals the Banach algebra spectrum of $a$ in the Banach algebra ${\poalgfont C}(a,1)$. Take $b\in {\poalgfont M}$ and an associative JB-subalgebra ${\poalgfont M}^\prime$ of ${\poalgfont M}$ containing 1 and $b$. Since $C(b,1)$ is isomorphic as a Banach algebra to $C(K)$ for a compact Hausdorff space $K$, and since the Banach algebra spectra of elements of such Banach algebras are easily seen to be stable under passing to Banach superalgebras,\footnote{This ingredient appears to be missing in the proof of the spectral mapping theorem in \cite[Proposition 1.21]{alfsen_shultz_GEOMETRY_OF_STATE_SPACES_OF_OPERATOR_ALGEBRAS:2003}.} $\sigma(b)$ is also equal to the Banach algebra spectrum of $b$ in he Banach algebra ${\poalgfont M}^\prime$. If $b\in {\poalgfont C}(a,1)$, then $C(b,1)\subseteq {\poalgfont C}(a,1)$, and we can now conclude that $\sigma(b)=\widehat b(X)$. In particular, $\sigma(a)=\widehat a(\ts)$.

Let ${\poalgfont W}(a,1)$ denote the $\sigma$-weakly closed Jordan subalgebra generated by $a$ and 1. It is an associative JBW-algebra. By  \cite[Proposition~2.11]{alfsen_shultz_GEOMETRY_OF_STATE_SPACES_OF_OPERATOR_ALGEBRAS:2003}, it is isomorphic to a $\cont{K}$-space. Hence ${\poalgfont W}(a,1)$ is also a unital Banach lattice algebra in the ordering inherited from ${\poalgfont M}$.

We view the inverse isomorphism $\widehat b\mapsto b$ from $\cont{\ts}$ onto ${\poalgfont C}(a,1)$ as a positive algebra homomorphism from $\cont{\ts}$ into ${\poalgfont W}(a,1)$. Since \cref{res:JBW_algebra_monotone_continous_multiplication} shows that ${\poalgfont W}(a,1)$ is a quasi-perfect partially ordered algebra with a monotone continuous multiplication, we can apply \cref{3_res:positive_homomorphisms_into_partially_ordered_algebras}.
Hence there exists a unique regular $\pos{{\poalgfont W}(a,1)}$-valued measure $\npm$ on the Borel $\sigma$-algebra of $\ts$ such that
\begin{equation}\label{3_eq:JBW_representation}
	b=\ointm{\widehat b}
\end{equation}
for $b\in {\poalgfont C}(a,1)$.
It is a spectral measure with $\npm(\ts)=1$. If $V$ is a non-empty open subset of $\ts$, then there exists a non-zero positive $\widehat b\in\cont{X}$ with support contained in $V$. The formula for $\npm(V)$ in \cref{3_res:positive_homomorphisms_into_partially_ordered_algebras} then implies that $\npm(V)>0$.

We now consider the image $\mu_a$ of $\mu$ under the measurable map $\widehat a\colon \ts\to \sigma(a)$. On setting $\npm_a(\mss)\coloneqq\npm\left(\widehat a^{-1}(\mss)\right)$  for a Borel subset $\mss$ of $\sigma(a)$, we obtain a spectral measure $\npm_a$ on the Borel $\sigma$-algebra $\borel(\sigma(a))$ of $\sigma(a)$ such that $\npm_a(\sigma(a))=1$. If $V$ is a non-empty (relatively) open subset of $\sigma(a)$, then $\npm_a(V)=\npm\left(\widehat a^{-1}(V)\right)>0$.

We claim that $\npm_a$ is a regular Borel measure. The shortest way to prove this is by noting that $\widehat a\colon \ts\to\sigma(a)$ is a homeomorphism (see the first part of the proof of \cite[Corollary~1.19]{alfsen_shultz_GEOMETRY_OF_STATE_SPACES_OF_OPERATOR_ALGEBRAS:2003}), so that the regularity is inherited from $\mu$. The regularity is, however, automatic.  To see this, take a normal state $\rho$  on ${\poalgfont W}(a,1)$. Since every (relatively) open subset of $\sigma(a)$ is $\sigma$-compact,  \cite[Theorem~2.18]{rudin_PRINCIPLES_OF_MATHEMATICAL_ANALYSIS_THIRD_EDITION:1976} implies that $\rho\circ\npm_a$ is a regular Borel measure on $\sigma(a)$. It then follows from \cite[Proposition~3.11]{de_jeu_jiang:2022a} that $\npm_a$ itself is a regular Borel measure. It is also automatic that $\mu$ is inner regular at all elements of $\borel(\sigma(a))$. Indeed, since $\rho\circ\npm_a$ is inner regular at every $\Delta\in\borel(\sigma(a))$ by \cite[Proposition~7.5]{folland_A_COURSE_IN_ABSTRACT_HARMONIC_ANALYSIS_SECOND_EDITION:2016}, we can again use \cite[Proposition~3.11]{de_jeu_jiang:2022a} to conclude that $\npm_a$ has the same property.

Using \cref{3_res:image_measure}, we see that
\begin{equation}\label{3_eq:integrals_related}
	\orderintegral{\ts}{f\circ \widehat a}{\npm}= \orderintegral{\sigma(a)}{f}{\npm_a}
\end{equation}
for $f\in \cont{\sigma(a)}$. If $p$ is a polynomial, then
\begin{equation}\label{3_eq:JBW_spectral_integral_all}
\orderintegral{\sigma(a)}{p}{\npm_a}=\orderintegral{\ts}{p\circ \widehat a}{\npm}=\orderintegral{\ts}{p(a)^\wedge}{\npm}=p(a).
\end{equation}
\Cref{3_eq:JBW_representation,3_eq:integrals_related} show that
\[
\left(\orderintegral{\sigma(a)}{f}{\npm_a}\right)^\wedge=f\circ\widehat a
\]
for $f\in\cont{\sigma(a)}$. Hence the spectrum of $\orderintegral{\sigma(a)}{f}{\npm_a}$ equals
$(f\circ \widehat a)(X)=f(\sigma(a))$.

As a particular case of \cref{3_eq:JBW_spectral_integral_all} we have
\begin{equation}\label{3_eq:JBW_spectral_integral}
	a=\orderintegral{\sigma(a)}{\idmap}{\npm_a}.
\end{equation}

We claim that $\npm_a$ is the only unital ${\poalgfont W}(a,1)$-valued spectral measure on $\borel(\sigma(a))$ that satisfies \cref{3_eq:JBW_spectral_integral}. To see this, let $\widetilde\npm$ be another such. Since the associated operators $I_{\npm_a}, I_{\widetilde \mu}\colon \cont{\sigma(a)}\to {\poalgfont W}(a,1)$ are algebra homomorphisms, they agree on all polynomials on $\sigma(a)$. As they are automatically norm continuous because of their positivity, they are equal. We can now use \cite[Proposition~6.8]{de_jeu_jiang:2022a} to conclude that
\begin{equation}
	\int_{\sigma(a)}\!\,f \di {(\rho\circ\npm_a)}=\int_{\sigma(a)}\! f\,\di {(\rho\circ\widetilde\npm)}
\end{equation}
for every normal state $\rho$ on ${\poalgfont W}(a,1)$ and for all $f\in \cont{\sigma(a)}$. Since, again by \cite[Theorem~2.18]{rudin_PRINCIPLES_OF_MATHEMATICAL_ANALYSIS_THIRD_EDITION:1976},  $\rho\circ\widetilde\npm$ is also automatically regular, we see that $\rho\circ\npm_a=\rho\circ\widetilde\npm$ for all normal states $\rho$. Hence $\npm_a=\widetilde{\npm}$.

\cref{3_res:integrable_function_for_spectral_measure_is_bounded} shows that $\Ell^1(\sigma(a),\borel(\sigma(a),\npm_a;\RR)=\upB(\sigma(a),\borel(\sigma(a),\npm_a;\RR)$.

Now that we have the spectral measure $\npm_a$, we can define the associated operators $I_{\npm,_a}$ from various spaces of (equivalence classes of) measurable functions on $\sigma(a)$ into ${\poalgfont W}(a,1)$. Since $\npm_a(\sigma(a))=1$, \cref{3_res:eight_properties} shows that $I_{\npm_a}\colon\upB(\sigma(a),\borel(\sigma(a),\npm_a;\RR)\to {\poalgfont W}(a,1)$ is a unital vector lattice algebra homomorphism. The fact that it is a vector lattice homomorphism implies that it is injective. After making the appropriate identifications, it is a injective algebra homomorphism between two $\cont{K}$-spaces. Hence it is isometric.

The fact that $\npm_a(V)>0$ for every non-empty (relatively) open subset $V$ of $\sigma(a)$ implies that the natural map from $\cont{\sigma(a)}$ into $\upB(\sigma(a),\borel(\sigma(a),\npm_a;\RR)$ is an injective isometry. Hence $I_{\npm_a}\colon \cont{\sigma(a)}\to {\poalgfont W}(a,1)$ is an isometric unital vector lattice algebra homomorphism. The image is then clearly ${\poalgfont C}(a,1)$. By the uniqueness statement in \cite[Corollary~1.19]{alfsen_shultz_GEOMETRY_OF_STATE_SPACES_OF_OPERATOR_ALGEBRAS:2003}, the map $I_{\npm_a}\colon \cont{\sigma(a)}\to {\poalgfont C}(a,1)$ is the continuous functional calculus for $a$. Consequently, and as we have also seen above, the spectral mapping theorem \cite[Proposition~1.21]{alfsen_shultz_GEOMETRY_OF_STATE_SPACES_OF_OPERATOR_ALGEBRAS:2003} holds for it.

The operator $I_{\npm_a}\colon\cont{\sigma(a)}\to {\poalgfont W}(a,1)$ is a positive unital algebra homomorphism to which \cref{3_res:positive_homomorphisms_into_partially_ordered_algebras} applies. It yields a representing regular Borel measure, which, by uniqueness, is $\mu_a$. Combined with the above, we thus have the following. We recall from  \cref{3_subsec:the_countable_sup_property} that the Banach lattice ${\poalgfont W}(a,1)$ has the countable sup property when it is separable or, more generally, when it has strictly positive (not necessarily normal) state.

\begin{theorem}\label{3_res:spectral_theorem_JBW_algebras}
Let ${\poalgfont M}$ be a JBW-algebra. Take $a\in {\poalgfont M}$, and let ${\poalgfont C}(a,1)$ \uppars{resp.\ ${\poalgfont W}(a,1)$} be the JB-subalgebra \uppars{resp.\ JBW-subalgebra} that is generated by $a$. These are both unital Banach lattice algebras in the ordering inherited from ${\poalgfont M}$.

There exists a unique spectral measure $\npm_a\colon\borel\to\pos{{\poalgfont W}(a,1)}$ on the Borel $\sigma$-algebra $\borel(\sigma(a))$ of $\sigma(a)$ with $\npm_a(\sigma(a))=1$ such that
\begin{equation*}%\label{3_eq:JBW_spectral_integral_in_theorem}
	a=\orderintegral{\sigma(a)}{\idmap}{\npm_a}
\end{equation*}
in ${\poalgfont W}(a,1)$.
The measure $\npm_a$ is a regular Borel measure that is inner regular at all Borel subsets of $\sigma(a)$. If $V$ is a non-empty \uppars{relatively} open subset of $\sigma(a)$, then
\begin{align*}
	0<\npm_a(V)&=\psup\{I_{\npm_a}(f) : f\in\cont{\sigma(a)},\, f\prec V\}
\intertext{in ${\poalgfont W}(a,1)$. If $K$ is a compact subset of $\sigma(a)$, then}
	\npm_a(K)&=\pinf\{I_{\mu_a}(f) : f\in\cont{\sigma(a)},\, K\prec f\}
\end{align*}
in ${\poalgfont W}(a,1)$.

Furthermore:
\begin{enumerate}

\item\label{3_part:spectral_theorem_JBW_algebras_1}
 $\Ell^1(\sigma(a),\borel(\sigma(a),\npm_a;\RR)=\upB(\sigma(a),\borel(\sigma(a),\npm_a;\RR)$.
 \item\label{3_part:spectral_theorem_JBW_algebras_2}
\begin{enumerate_alpha}
	\item $\upB(\sigma(a),\borel(\sigma(a),\npm_a;\RR)$ is a $\sigma$-Dedekind complete unital Banach lattice algebra, and $I_{\npm_a}\colon\upB(\sigma(a),\borel(\sigma(a),\npm_a;\RR)\to {\poalgfont W}(a,1)$ is an isometric $\sigma$-order continuous unital vector lattice algebra homomorphism.
	\item If ${\poalgfont W}(a,1)$ has the countable sup property, $\upB(\sigma(a),\borel(\sigma(a),\npm_a;\RR)$ is Dedekind complete, and $I_{\npm_a}\colon \upB(\sigma(a),\borel(\sigma(a),\npm_a;\RR)\to {\poalgfont W}(a,1)$ is order continuous.
	
	  \end{enumerate_alpha}
\item\label{3_part:spectral_theorem_JBW_algebras_3}
\begin{enumerate_alpha}
	\item The natural map from $\cont{\sigma(a)}$ into $\upB(\sigma(a),\borel(\sigma(a),\npm_a;\RR)$ is an isometric unital vector lattice algebra homomorphism.	
	\item
	$I_{\npm_a}\colon\cont{\sigma(a)}\to {\poalgfont C}(a,1)$ is a surjective isometric unital vector lattice algebra homomorphism which coincides with the continuous functional calculus for $a$.
	\item $\sigma(I_{\mu_a}(f))=f(\sigma(a))$ for $f\in\cont{\sigma(a)}$.
\end{enumerate_alpha}

\item\label{3_part:spectral_theorem_JBW_algebras_4}
For a normal state $\rho$ on ${\poalgfont W}(a,1)$ and $\mss\in\borel(\sigma(a))$, set $\npm_{a,\rho}(\mss)\coloneqq(\npm_a(\mss),\rho)$. Then $\npm_{a,\rho}\colon \borel(\sigma(a))\to\posR$ is a regular Borel measure that is inner regular at all Borel subsets of $\sigma(a)$. For $f\in\upB(\sigma(a),\borel(\sigma(a));\RR)$, $I_{\npm_a}(f)=\int_{\sigma(a)}^{\upo}\!f \di{\npm_a}$ is the unique element of ${\poalgfont W}(a,1)$ such that
\begin{equation*}
	\left(I_{\npm_a}(f), \rho\right)=\int_{\sigma(a)}\!f \di{\npm_{a,\rho}}
\end{equation*}
for all normal states $\rho$ on ${\poalgfont W}(a,1)$.
\item\label{3_part:spectral_theorem_JBW_algebras_5}
Suppose that ${\poalgfont W}(a,1)$ has the \csp.  Then:
\begin{align*}
	I_{\npm_a}(\upB(\sigma(a),\borel(\sigma(a),\npm_a;\pos{\RR}))&=I_{\npm_a}(\upB(\sigma(a),\borel(\sigma(a),\npm_a;\pos{\RR}))^\up\\ &=I_{\npm_a}(\upB(\sigma(a),\borel(\sigma(a),\npm_a;\pos{\RR}))^\down;\\
	I_{\npm_a}(\upB(\sigma(a),\borel(\sigma(a),\npm_a;\pos{\RR}))&=I_{\npm_a}(\eclass{\pos{\cont{\sigma(a)}}}^{\downs\ups})\\
	&=\big[ \pos{{\poalgfont C}(a,1)}\big]^{\downs\ups};\\
	I_{\npm_a}(\upB(\sigma(a),\borel(\sigma(a),\npm_a;\pos{\RR}))&=I_{\npm_a}(\eclass{\pos{\cont{\sigma(a)}}}^{\ups\downs})\\
	&=\big[\pos{{\poalgfont C}(a,1)}\big]^{\ups\downs};\\
	I_{\npm_a}(\upB(\sigma(a),\borel(\sigma(a),\npm_a;\RR))&=I_{\npm_a}(\upB(\sigma(a),\borel(\sigma(a),\npm_a;\RR))^\up\\&=I_{\npm_a}(\upB(\sigma(a),\borel(\sigma(a),\npm_a;\RR))^\down;\\
	I_{\npm_a}(\upB(\sigma(a),\borel(\sigma(a),\npm_a;\RR))&=\big[{\poalgfont C}(a,1)\big]^{\ups\downs}\\
	&=\big[{\poalgfont C}(a,1)\big]^{\downs\ups}	
		\end{align*}
in ${\poalgfont W}(a,1)$.
\end{enumerate}
\end{theorem}

\begin{remark}\quad
	\begin{enumerate}
	\item The fact that $\mu_a$ is a measure, the order integral, and the suprema and infima in \cref{3_res:spectral_theorem_JBW_algebras} (with those in $\upB(\sigma(a),\borel(\sigma(a),\npm_a;\RR)$ in part~\ref{3_part:spectral_theorem_JBW_algebras_5} excepted) are all defined with respect to the ordering in ${\poalgfont W}(a,1)$. However, since these involve only monotone nets, and since bounded monotone nets in ${\poalgfont W}(a,1)$ convergence $\sigma$-weakly to their extrema by \cite[Proposition~2.5(ii)]{alfsen_shultz_GEOMETRY_OF_STATE_SPACES_OF_OPERATOR_ALGEBRAS:2003}, the fact that ${\poalgfont W}(a,1)$ is $\sigma$-weakly closed in ${\poalgfont M}$ implies that one can equivalently use the ordering in ${\poalgfont M}$.
	\item Using ad hoc methods involving von Neumann algebras, it is shown in \cite[Theorem~2.1]{roelands_wortel:2020} that there exists a unital algebra homomorphism from ${\lebfont B}(\sigma(a),\borel(\sigma(a);\RR)$ into ${\poalgfont W}(a,1)$. This also follows from the more precise part~\ref{3_part:spectral_theorem_JBW_algebras_1} of \cref{3_res:spectral_theorem_JBW_algebras}.
	
	In \cite[Theorem~2.1]{roelands_wortel:2020}, it is also stated that the pertinent algebra homomorphism is surjective. This can, however, not be concluded from its proof. \footnote{Personal communication by Marten Wortel.} Surjectivity is \emph{not} stated in part~\ref{3_part:spectral_theorem_JBW_algebras_1} of \cref{3_res:spectral_theorem_JBW_algebras}. Given the proof of \cref{3_res:Borel_functional_calculus}, it seems likely that, for the proof of such a result (presumably under additional hypotheses), an analogue of \cite[Theorem~2.4.4]{pedersen_C-STAR-ALGEBRAS_AND_THEIR_AUTOMORPHISM_GROUPS:1979} for JB- and JBW-algebras is necessary.
	\item If $\seq{f}$ is a uniformly bounded sequence in ${\lebfont B}(\sigma(a),\borel(\sigma(a));\RR)$ converging pointwise to $f$, then $\{I_{\npm_a}(f_n)\}_{n=1}^\infty$ converges $\sigma$-weakly to $I_{\npm_a}(f)$. This fact, which was already observed in \cite[Theorem~2.1]{roelands_wortel:2020}, is immediate from part~\ref{3_part:spectral_theorem_JBW_algebras_3} and the (classical) dominated convergence theorem.
	\end{enumerate}
\end{remark}

We shall now show how the spectral resolution for $a$ as in \cite[Theorem~2.20]{alfsen_shultz_GEOMETRY_OF_STATE_SPACES_OF_OPERATOR_ALGEBRAS:2003} can be found from the underlying spectral measure $\npm_a$. For $\lambda\in\RR$, set $e_\lambda\coloneqq\npm_a((-\infty,\lambda]\cap\sigma(a))$. We claim that $\{e_\lambda:\lambda\in\RR\}$ is the spectral resolution for $a$. First of all, $e_\lambda$ is a projection in the associative algebra ${\poalgfont W}(a,1)$, so that it operator commutes with $a$ by \cite[Proposition~1.47]{alfsen_shultz_GEOMETRY_OF_STATE_SPACES_OF_OPERATOR_ALGEBRAS:2003}. Furthermore, $U_{e_\lambda} a=e_\lambda\circ a$. Since $\chi_{(-\infty,\lambda]\cap(\sigma(a))}\cdot \idmap\leq \lambda \chi_{(-\infty,\lambda]\cap(\sigma(a))}$ as functions on $\sigma(a)$, we have
\begin{align*}
U_{e_\lambda} a&=e_\lambda\circ a\\
&=I_{\npm_a}(\chi_{(-\infty,\lambda]\cap(\sigma(a))})\circ I_{\npm_a}(\idmap)\\
&=I_{\npm_a}(\chi_{(-\infty,\lambda]\cap(\sigma(a))}\cdot \idmap)\\
&\leq I_{\npm_a}(\lambda\chi_{(-\infty,\lambda]\cap(\sigma(a))})\\
&=\lambda e_\lambda.
\end{align*}
It follows similarly that $U_{1-e_\lambda}\geq\lambda(1-e_\lambda)$. Since $\sigma(a)\subseteq[-\norm{a},\norm{a}]$, it is clear that $e_\lambda=0$ for $\lambda<-\norm{a}$ and that $e_\lambda=1$ for $\lambda>\norm{a}$. Certainly, $e_{\lambda_1}\leq e_{\lambda_2}$ when $\lambda_1<\lambda_2$. It follows from \cite[Proposition~4.6]{de_jeu_jiang:2022a} that $e_\lambda=\bigwedge_{n\geq 1} e_{\lambda+1/n}$, so that $e_\lambda=\bigwedge_{\lambda^\prime>\lambda}e_{\lambda^\prime}$ by the monotonicity of $\lambda\mapsto e_\lambda$. We have now verified all defining properties for the spectral resolution for $a$ in \cite[Theorem~2.20]{alfsen_shultz_GEOMETRY_OF_STATE_SPACES_OF_OPERATOR_ALGEBRAS:2003}.

It is also easy to see why $a$ can be approximated by Riemann{\textendash}Stieltjes type sums as in \cite[Theorem~2.20]{alfsen_shultz_GEOMETRY_OF_STATE_SPACES_OF_OPERATOR_ALGEBRAS:2003}. This is, in fact, true for $I_{\npm_a}(f)$ for every continuous $f\colon\sigma(a)\to\RR$, i.e., for all $b\in {\poalgfont C}(a,1)$, and then in particular for $a=I_{\npm_a}(\idmap)$. To see this, fix a uniformly continuous extension $f_{\upe}$ of $f$ to $\RR$. One can always find such $f_{\upe}$ with compact support by \cite[Theorem~20.4]{rudin_PRINCIPLES_OF_MATHEMATICAL_ANALYSIS_THIRD_EDITION:1976}, but for the identity function on $\sigma(a)$ the identity function on $\RR$ is also a possible choice. Take a finite increasing sequence $\gamma=\{\lambda_0,\dotsc,\lambda_n\}$ such that $\sigma(a)\subset(\lambda_0,\lambda_n]$. Write $\norm{\gamma}=\max_{i=1,\dotsc,n}(\lambda_{i}-\lambda_{i-1})$. For $i=1,\ldots,n$, take any $\lambda_i^\prime\in[\lambda_{i-1},\lambda_i]\cap\sigma(a)$. Then $\sum_{i=1}^n f_{\upe}(\lambda_i^\prime)\chi_{(\lambda_{i-1},\lambda_i]\cap\sigma(a)}\to f$ uniformly on $\sigma(a)$ as $\norm{\gamma}\to 0$. Applying the continuous operator $I_{\npm_a}$ yields that $\sum_{i=1}^nf_{\upe}(\lambda_i^\prime)(e_{\lambda_{i}}-e_{\lambda_{i-1}})\to I_{\npm_a}(f)$ as $\norm{\gamma}\to 0$. In particular, this is true for $f=\idmap$ with as extension $f_{\upe}$ the identity map on $\RR$, finite increasing sequences $\gamma=\{\lambda_0,\dotsc,\lambda_n\}$ such that $[-\norm{a},\norm{a}]\subset(\lambda_0,\lambda_n]$, and $\lambda_i^\prime=\lambda_i$. This implies the approximation result for $a$ by Riemann{\textendash}Stieltjes type sums for $a$ in \cite[Theorem~2.20]{alfsen_shultz_GEOMETRY_OF_STATE_SPACES_OF_OPERATOR_ALGEBRAS:2003}.\footnote{As the present argument in the paper shows, it is sufficient to have $[-\norm{a},\norm{a}]\subset(\lambda_0,\lambda_n]$ rather than $[-\norm{a},\norm{a}]\subset(\lambda_0,\lambda_n)$ for $a$ as in \cite[Theorem~2.20]{alfsen_shultz_GEOMETRY_OF_STATE_SPACES_OF_OPERATOR_ALGEBRAS:2003}.} The ease with which this can be extended to arbitrary elements of ${\poalgfont C}(a,1)$, as a consequence of the existence of the Borel functional calculus that comes with the spectral measure, may serve as evidence that it is better to work with the spectral measure than with the spectral resolution that is derived from it.

%Since $\left[(-\infty,\lambda_1]\cap\sigma(a)\right]\cap\left[(-\infty,\lambda_2]\cap\sigma(a)\right]=\left[(-\infty,\lambda_1]\cap\sigma(a)\right]$ when $\lambda_1\leq \lambda_2$, the fact that $\npm_a$ is a spectral measure implies that $e_{\lambda_1}\circ e_{\lambda_2}=e_{\lambda_1}$ for positive $\lambda_1\leq\lambda_2$. It follows that $e_{\lambda_2}-e_{\lambda_1}$ is a positive projection. In view of the above, every element of $I_{\npm_a}(\cont{\sigma(a)})$, i.e., every element of ${\poalgfont C}(a,1)$ (and in particular $a$ itself) is a norm limit of linear combinations of positive projections in ${\poalgfont W}(a,1)$. This is even true for $I_{\npm_a}(f)$ with $f$ an arbitrary bounded measurable function on $\sigma(a)$, although the projections need not be related to the spectral resolution. One need merely choose a sequence of linear combinations of characteristic functions of Borel subsets of $\sigma(a)$ that converges uniformly to $f$, and apply $I_{\npm_a}$. The simplicity of this (well known) argument may serve as evidence that it is better to work with the spectral measure than with the spectral resolution that can be derived from it.

%%%%%%%%%%%%%%%%%%%%%%%%%%%%%%%%% ACKNOWLEDGEMENTS %%%%%%%%%%%%%%%%%%%%%%%%%%%%%%%%%%%%%

\subsection*{Acknowledgements} The authors thank Onno van Gaans for his assistance with \cref{3_rem:kalauch_stennder_van_gaans}, and Marten Wortel for helpful discussions on JBW-algebras. They are grateful to Ben de Pagter for bringing Schaefer's work on lattice-subspaces and the relevant results on $\!f\!$-algebras in \cite{de_pagter_THESIS:1981} under their attention.

%%%%%%%%%%%%%%%%%%%%%%%%%%%%%%%% END OF THE ACTUAL TEXT %%%%%%%%%%%%%%%%%%%%%%

%%%%%%%%%%%%%%%%%%%%%%%%%%%%%%%% BIBLIOGRAPHY %%%%%%%%%%%%%%%%%%%%%%%%%%%%%%%%%%%%%%%%%%

\bibliographystyle{plain}
\urlstyle{same}

\bibliography{general_bibliography}

%\bibliography{../../../tex/templates/bibliography/general_bibliography}

\end{document}